\documentclass[10pt,reqno]{amsart}

\usepackage{mathtools}

\usepackage[T1]{fontenc}

\usepackage{amsmath,amsfonts,amsbsy,amsgen,amscd,mathrsfs,amssymb,amsthm}
\usepackage{enumerate}
\usepackage{bm}
\usepackage{calc}

\usepackage{stmaryrd}
\SetSymbolFont{stmry}{bold}{U}{stmry}{m}{n} 

\usepackage[usenames,dvipsnames]{xcolor}
\usepackage[colorlinks=true,citecolor=blue,linkcolor=blue]{hyperref}

\usepackage{tikz}
\usepackage[font=small]{caption} 

\newtheorem{thm}{Theorem}[section]
\newtheorem{lem}[thm]{Lemma}

\newtheorem{prop}[thm]{Proposition}
\newtheorem{cor}[thm]{Corollary}

\theoremstyle{definition}

\newtheorem{defn}[thm]{Definition}

\newtheorem*{notn}{Notation}

\theoremstyle{remark}

\newtheorem{example}[thm]{Example}

\newtheorem{rem}[thm]{Remark}
\newtheorem*{rem*}{Remark}

\numberwithin{equation}{section} 
\numberwithin{figure}{section}
\numberwithin{table}{section}

\makeatletter

\let\oldtocsection=\tocsection
\let\oldtocsubsection=\tocsubsection
\let\oldtocsubsubsection=\tocsubsubsection
\renewcommand{\tocsection}[2]{\hspace{-1.2em}\oldtocsection{#1}{#2}}
\renewcommand{\tocsubsection}[2]{\hspace{-.2em}\oldtocsubsection{#1}{#2}}
\renewcommand{\tocsubsubsection}[2]{\hspace{0.8em}\oldtocsubsubsection{#1}{#2}}

\DeclareRobustCommand{\gobblefive}[5]{}
\newcommand*{\SkipTocEntry}{\addtocontents{toc}{\gobblefive}}

\makeatother

\newcommand{\ntr}{\mathop{\mathrm{tr}}}
\newcommand{\tr}{\mathop{\mathrm{Tr}}}

\newcommand{\M}{\mathrm{M}}
\newcommand{\EE}{\mathbf{E}}

\newcommand{\spc}{\mathrm{sp}}
\newcommand{\id}{\mathbf{1}}
\newcommand{\II}{\mathbf{I}}
\newcommand{\I}{\mathrm{I}}

\newcommand{\diag}{\mathop{\mathrm{diag}}}
\newcommand{\diaginv}{\mathop{\mathrm{diag}^{-1}}}

\begin{document}

\title[Matrix concentration and free probability II]{Matrix 
concentration 
inequalities and free probability II. Two-sided bounds and applications}

\author[Bandeira]{Afonso S.\ Bandeira}
\address{Department of Mathematics, ETH Z\"urich, Switzerland}
\email{bandeira@math.ethz.ch}

\author[Cipolloni]{Giorgio Cipolloni}
\address{Department of Mathematics, University of Rome Tor Vergata,
Via della Ricerca Scientifica 1, 00133 Roma RM, Italy}
\email{cipolloni@axp.mat.uniroma2.it}

\author[Schr\"oder]{Dominik Schr\"oder}
\address{Department of Mathematics, ETH Z\"urich, Switzerland}
\email{schroeder.dominik@gmail.com}

\author[Van Handel]{Ramon van Handel}
\address{Department of Mathematics, Princeton University, Princeton, NJ 
08544, USA}
\email{rvan@math.princeton.edu}

\begin{abstract}
The first paper in this series introduced a new family of 
nonasymptotic matrix concentration inequalities that sharply capture 
the spectral properties of very general 
random matrices in terms of an associated noncommutative model. These 
methods achieved matching upper and lower bounds for smooth spectral 
statistics, but only provided upper bounds for the spectral edges. Here 
we obtain matching lower bounds for the spectral edges, completing the 
theory initiated in the first paper. 

The resulting two-sided bounds enable the study of problems that require 
an exact determination of the spectral edges to leading order, which is 
fundamentally beyond the reach of classical matrix concentration 
inequalities. To illustrate their utility, we develop two general results
that explain phase transitions of spectral outliers of a large 
class of nonhomogeneous random matrices. This enables us to 
elucidate phase transition phenomena that arise in diverse applications,
including decoding node labels on graphs, tensor PCA, contextual 
stochastic block models, and centered sample covariance matrices.
\end{abstract}

\subjclass[2010]{60B20; 
                 60E15; 
                 46L53; 
                 46L54; 
                 15B52} 

\keywords{Random matrices; matrix concentration inequalities; 
free probability}

\maketitle

\thispagestyle{empty}
{\small
\setcounter{tocdepth}{2}
\tableofcontents
}

\section{Introduction}

Let $X$ be any $d\times d$ self-adjoint random matrix with jointly 
Gaussian entries. What can we say about its spectrum? As we made no 
assumptions on the mean and covariance of the entries, classical methods 
of random matrix theory shed little light on this question. Nonetheless, 
nontrivial bounds on the spectrum are achievable at this level of 
generality by means of operator-theoretic results that are often referred 
to as matrix concentration inequalities.

The classical such result, the noncommutative Khintchine 
inequality \cite{Pis03,Buc01}, estimates any finite moment of a 
(centered) Gaussian matrix explicitly up to a 
constant factor in terms of the covariance of its entries.
Stated precisely, for any self-adjoint Gaussian random matrix $X$ with
$\mathbf{E}X=0$ we have
\begin{equation}
\label{eq:nck}
	\ntr[(\mathbf{E}X^2)^p]^{\frac{1}{2p}}
	\le
	\mathbf{E}[\ntr X^{2p}]^{\frac{1}{2p}}
	\le
	\sqrt{2p}\,
	\ntr[(\mathbf{E}X^2)^p]^{\frac{1}{2p}}
\end{equation}
for $p\in\mathbb{N}$, where we define the normalized trace $\ntr M := 
\frac{1}{d}\tr M$ for any $M\in \M_d(\mathbb{C})$.
Using the
basic fact (see \eqref{eq:whylogdmoment} below) that for any  
$M\in \M_d(\mathbb{C})$
\begin{equation}
\label{eq:logdmoment}
	\ntr[|M|^{2p}]^{\frac{1}{2p}}
	= (1+o(1))\|M\|\qquad
	\text{for }p\gg \log d,
\end{equation}
the bound
\eqref{eq:nck} also yields upper and lower 
bounds for the spectral norm $\|X\|$ up to a factor that grows 
logarithmically with dimension. Much work in the past two decades has been 
devoted to extending such bounds to a large class of non-Gaussian models, 
cf.\ \cite{Tro15} and the references therein.

Due to their generality and ease of use, matrix concentration inequalities 
have found numerous applications in pure and applied mathematics. At the 
same time, they can provide only rough bounds on the behavior of the 
spectrum that are often increasingly inaccurate in high dimension, in 
contrast to classical results in random matrix theory that become 
increasingly precise as $d\to\infty$. It has been a long-standing question 
whether there exist results at the level of generality of matrix 
concentration inequalities which can nonetheless sharply capture the 
spectral properties of many random matrix models.

Significant progress in this direction was achieved in the first part of 
this series \cite{BBV23} (inspired in part by \cite{HT05,Tro18}), which 
introduced a new family of matrix concentration inequalities that 
optimally capture the behavior of very general Gaussian random matrices to 
leading order. A key feature of these inequalities is that they do not  
bound the spectrum of $X$ directly, but rather quantify the 
deviation of the spectrum of $X$ from that of an associated 
deterministic operator $X_{\rm free}$ defined by a matrix with 
entries in a $C^*$-probability space
$(\mathcal{A},\tau)$ (we recall its definition
in section \ref{sec:main}). For example, \cite[Theorem 2.7]{BBV23} yields
the inequality
\begin{equation}
\label{eq:freenckmom}
	|\mathbf{E}[\ntr X^{2p}]^{\frac{1}{2p}} -
	({\ntr}\otimes\tau)[X_{\rm free}^{2p}]^{\frac{1}{2p}}|
	\le 2 p^{\frac{3}{4}} \tilde v(X) 
\end{equation} 
for all $p\in\mathbb{N}$, where $\tilde v(X)^4 
:=\|\mathrm{Cov}(X)\|\,\|\mathbf{E}[(X-\mathbf{E}X)^2]\|$ and 
$\mathrm{Cov}(X)$ is the $d^2\times d^2$ covariance matrix of the 
entries of $X$. Here $X$ need not be centered,
that is, both mean and covariance of $X$ are arbitrary.

Unlike \eqref{eq:nck}, which upper and lower bounds the moments of $X$ up 
to a constant factor, \eqref{eq:freenckmom} computes the moments of $X$ 
exactly to leading order when $\tilde v(X)$ is small. The latter situation 
is ubiquitous in applications. Such bounds extend also to non-Gaussian 
models by means of a universality principle \cite{TV23}. At the same time, 
tools of free probability theory \cite{HT05,Leh97} make it possible to 
compute or estimate the spectral statistics of $X_{\rm free}$ explicitly 
in terms of the mean and covariance of $X$, making \eqref{eq:freenckmom} 
genuinely applicable to concrete situations.

The inequality \eqref{eq:freenckmom} is just one example of the kind of 
results that are achieved by the theory of \cite{BBV23}; the same method 
of proof yields analogous two-sided bounds for many smooth spectral 
statistics. But arguably the most powerful aspect of this theory lies in 
its ability to capture the edges of the spectrum, which is considerably 
more delicate than the bulk spectral behavior. In this regard, however,
the theory is incomplete. For example, it is shown in 
\cite[Corollary 2.2]{BBV23} that
\begin{equation}
\label{eq:normupper}
	\mathbf{E}\|X\| \le 
	\|X_{\rm free}\| + C\tilde v(X)(\log d)^{\frac{3}{4}}
\end{equation}
for a universal constant $C$, which yields a sharp \emph{upper} bound on 
$\|X\|$ whenever $\tilde v(X)$ is small. However, unlike in 
\eqref{eq:freenckmom}, the corresponding \emph{lower} bound is missing. 
The reason for this discrepancy, as explained in \cite[\S 8.2.3]{BBV23}, 
lies in the elementary fact \eqref{eq:logdmoment}: while the norm of any 
$d\times d$ matrix can be approximated by moments of order $\log d$, it is 
far from clear whether the analogous property holds for the 
infinite-dimensional operator $X_{\rm free}$. This problem is resolved
in this paper, which completes the theory of \cite{BBV23} and opens the
door to new applications.

\subsection{Main results}

The simplest implication of the new results of this paper may be 
readily understood in the context of the above discussion: the following 
theorem extends the upper bound \eqref{eq:normupper} to a two-sided bound.

\begin{thm}
\label{thm:normtwoside}
For any $d\times d$ random matrix $X$ with jointly Gaussian entries
$$
	| \mathbf{E}\|X\| - \|X_{\rm free}\| | \le
	C\tilde v(X)(\log d)^{\frac{3}{4}},
$$
where $C$ is a universal constant. When $X$ is self-adjoint, the
same inequality holds if $\|X\|,\|X_{\rm free}\|$ are replaced by 
the upper edge of the spectrum
$\lambda_{\rm max}(X),\lambda_{\rm max}(X_{\rm free})$.
\end{thm}

However, our results are much more general than is suggested by 
Theorem~\ref{thm:normtwoside}. The main result of this paper provides a 
subgaussian matrix concentration inequality for the Hausdorff distance 
between the spectra of $X$ and $X_{\rm free}$. This makes it possible both 
to achieve high probability results, and to detect interior edges of the 
spectrum in addition to exterior edges. Our results will be stated in full 
generality in section \ref{sec:main} after we recall the relevant 
definitions.

\begin{figure}
\centering
\begin{tikzpicture}

\begin{scope}
\clip (4.2,0) rectangle (7.2,1);
\fill[gray!20]
(4,0) to
(4,0.4)
to[out=295,in=180] (4.5,0.11)
to[out=0,in=180] (7,0.1) to[out=0,in=95] (7.1,0)
to (4,0); 
\end{scope}

\draw[thick] (0,0) to[out=85,in=180] (2,2) 
to[out=0,in=115] (4,0.4)
to[out=295,in=180] (4.5,0.11)
to[out=0,in=180] (7,0.1) to[out=0,in=95] (7.1,0);

\draw[thick,->] (-.25,0) -- (8,0);
\draw[thick,->] (-.25,0) -- (-.25,2.5);

\draw[thick] (7.1,0) -- (7.1,-.1);
\draw[thick] (-.25,0) -- (-.25,-.1);
\draw (7.1,-.25) node {$\scriptstyle\|X_{\rm free}\|$};
\draw (-.25,-.25) node {$\scriptstyle0$};

\draw[thick,dashed] (4.2,0) -- (4.2,1);

\draw[->] (6,1) to[out=260,in=75] (5.6,0.2);
\draw (6.3,1.2) node {$\scriptstyle\text{mass }e^{-d}$};

\draw[->] (1.9,1) to[out=280,in=95] (2.2,0.2);
\draw (2,1.2) node {$\scriptstyle\text{eigenvalues of }X$};

\pgfmathsetseed{701}

\foreach \i in {0,...,10}
{
	\draw ({2*(\i/10)^(2/3)+rnd/10},0) node
	{$\scriptstyle\times$};
}
\foreach \i in {0,...,9}
{
	\draw ({4-2*(\i/10)^(2/3)-rnd/10},0) node
	{$\scriptstyle\times$};
}

\end{tikzpicture}
\caption{Illustration of a hypothetical obstruction to the validity of 
Theorem~\ref{thm:normtwoside}. The proof must show that this situation
cannot occur.\label{fig:bad}}
\end{figure}
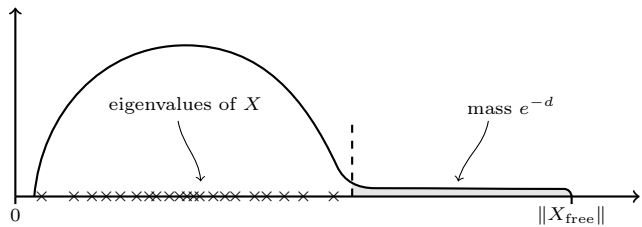

As was explained above, the key new ingredient that is needed in the proof 
of these results is that the norm of $X_{\rm free}$ (more generally, of 
its resolvent $(z-X_{\rm free})^{-1}$) is well approximated by its 
moments. That the moments are upper bounded by the norm is trivial, which 
is the reason that the upper bound \eqref{eq:normupper} was achievable in 
\cite{BBV23}. The converse direction is far from clear, however.

To understand where the difficulty lies, it is instructive to recall 
why \eqref{eq:logdmoment} holds for a self-adjoint $d\times d$ matrix $M$
with eigenvalues $\lambda_1(M),\ldots,\lambda_d(M)$: as
\begin{equation}
\label{eq:whylogdmoment}
	\ntr[|M|^{2p}] = \frac{1}{d}\sum_{i=1}^d |\lambda_i(M)|^{2p}
	\ge \frac{1}{d}\max_i|\lambda_i(M)|^{2p}
	= \frac{1}{d}\|M\|^{2p},
\end{equation}
we have $\ntr[|M|^{2p}]^{\frac{1}{2p}}\ge d^{-\frac{1}{2p}}\|M\|= 
(1+o(1))\|M\|$ for $p\gg\log d$. Thus \eqref{eq:logdmoment} holds because 
the empirical spectral distribution of $M$ has mass $\frac{1}{d}$ at 
$\|M\|$. However, it is not clear why the spectral distribution of the 
infinite-dimensional operator $X_{\rm free}$ should also have large mass 
near its edges. For example, Figure~\ref{fig:bad} illustrates a 
hypothetical scenario where the spectral distribution of $X_{\rm free}$ 
only has mass $e^{-d}$ near its upper edge; in this case, it would be 
extremely unlikely that any eigenvalue of $X$ (each of which has mass 
$\frac{1}{d}$) is located near $\|X_{\rm free}\|$, contradicting 
Theorem~\ref{thm:normtwoside}. Our proof must therefore show
that such situations cannot occur.

We have in fact developed two distinct methods of proof to achieve this 
aim. The first method is based on the work of Alt, Erd\H{o}s, and Kruger 
\cite{AEK20}, who made a detailed study of the behavior of the spectral 
distribution of $X_{\rm free}$ near the edges of the spectrum under two 
strong regularity assumptions: flatness, which requires in particular that 
all entries of $X$ have variance of the same order, and a uniform bound on 
$\|\mathbf{E} X\|$. Both assumptions are highly problematic in our 
setting, as they rule out precisely the kind of nonhomogeneous models that 
matrix concentration inequalities aim to capture. However, in fixed 
dimension $d$, any random matrix can be perturbed with negligible effect 
on the spectrum so that it satisfies these regularity assumptions with 
constants that diverge polynomially with $d$. One can therefore rule out 
situations as in Figure~\ref{fig:bad} by a combination of spectral 
perturbation theory and a quantitative refinment of the results of 
\cite{AEK20}. A proof of our main results by this approach appears in an 
early draft of this paper \cite{BCSV23}.

An entirely different approach arises, in a slightly different setting, in 
the work of Bordenave and Collins \cite[\S 7.1]{BC23}. The idea introduced 
there is that the comparison between moments and norm of $X_{\rm free}$ is 
closely connected to the ultracontractive properties of free operators. 
Exploiting this idea in the present setting considerably shortens the 
proof of our main results, as it replaces the rather technical work based 
on \cite{AEK20} by operator-theoretic tools. We therefore present the 
latter approach in this paper. Despite this simplification, 
ultracontractivity does not suffice in itself to achieve satisfactory 
bounds, so that spectral perturbation arguments 
remain crucial for the proof. Details and further discussion are given 
in section \ref{sec:ultra}.

While the above discussion focused on Gaussian models, random matrix 
models that arise in applications are often non-Gaussian. This will not 
present any additional complications, however, as the theory of this paper 
extends directly to many non-Gaussian situations using the universality 
principle of \cite{TV23}.

\subsection{Phase transitions}
\label{sec:introapp}

The development of inequalities that exactly capture the spectral edges to 
leading order enables the study of problems that are fundamentally 
beyond the reach of classical matrix concentration inequalities, such as 
phase transition phenomena for spectral outliers of nonhomogeneous random 
matrices. We will develop this theme in some detail, both as a compelling 
illustration of our general theory and for its independent interest.

\begin{figure}
\centering
\begin{tikzpicture}

\begin{scope}
\clip (-.25,0) rectangle (7.2,2.1);

\fill[gray!20] (6,0) circle (0.25);
\draw[thick] (2,0) circle (2);
\draw[thick] (6,0) circle (0.25);
\end{scope}

\draw[thick,<->] (8,0) -- (-.25,0) -- (-.25,2.5);

\pgfmathsetseed{671}

\foreach \i in {0,...,10}
{
	\draw ({1.9*(\i/10)^(2/3)+rnd/10},0) node
	{$\scriptstyle\times$};
}
\foreach \i in {0,...,9}
{
	\draw ({4-1.9*(\i/10)^(2/3)-rnd/10},0) node
	{$\scriptstyle\times$};
}

\draw (5.9,0) node {$\scriptstyle\times$};

\draw[thick] (6,0) -- (6,-.1);
\draw[thick] (2,0) -- (2,-.1);
\draw[thick] (-.01,0) -- (-.01,-.1);
\draw[thick] (4.01,0) -- (4.01,-.1);
\draw (6,-.25) node {$\scriptstyle \theta+\frac{1}{\theta}$};
\draw (-.1,-.25) node {$\scriptstyle-2$};
\draw (2,-.25) node {$\scriptstyle0$};
\draw (4.01,-.25) node {$\scriptstyle2$};

\draw[->] (6.4,1.2) to[out=260,in=75] (6,0.4);
\draw (6.55,1.4) node {$\scriptstyle\text{mass }\frac{1}{d}$};

\end{tikzpicture}
\caption{Spectral distribution of $X_{\rm free}$ for the
spiked Wigner model.\label{fig:bbp}}
\end{figure}
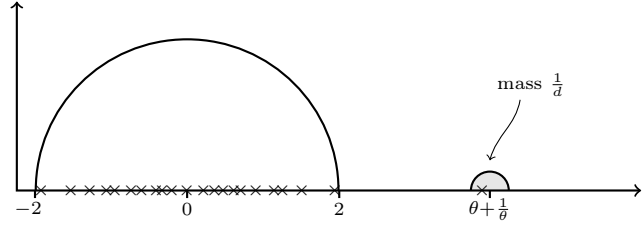

The classical study of phase transitions for spectral outliers due to 
Baik, Ben Arous, and P\'ech\'e \cite{BBP05} has led to a large body of 
work, see, e.g., the survey \cite{CD17}. Let us briefly recall one of the 
standard models in this area. Let $G$ be a $d\times d$ self-adjoint 
matrix with i.i.d.\ standard Gaussian entries above the diagonal, and let
$$
	X = \theta\,vv^* + G
$$
where $\theta\ge 0$ and $\|v\|=1$. This is the spiked 
Wigner model. It is classical that the largest eigenvalue of $G$ is 
$2+o(1)$. For the random matrix $X$, however, we observe a phase 
transition: its largest eigenvalue is still $2+o(1)$ when 
$\theta\le 1$, while an outlier eigenvalue emerges at 
$\theta+\frac{1}{\theta}+o(1)$ when $\theta>1$.

While such a sharp transition is clearly inaccessible by classical matrix 
concentration inequalities, it can be recovered as an easy exercise from 
Theorem \ref{thm:normtwoside} using an explicit formula of Lehner 
(see \eqref{eq:lehner} below) for the largest eigenvalue of $X_{\rm 
free}$.\footnote{
	It may appear surprising that $X_{\rm free}$, which has a 
	continuous spectrum, can detect the presence of a single 
	eigenvalue of $X$. This paradox is explained in Figure 
	\ref{fig:bbp}: the spectral distribution of $X_{\rm free}$ has a 
	small connected component of mass $\frac{1}{d}$ which produces the 
	outlier eigenvalue of $X$.
}
As our theory applies to arbitrarily structured random matrices, however, 
it enables the study of such phenomena in far more general 
situations:
\begin{enumerate}[$\bullet$]
\itemsep\medskipamount
\item We may replace $G$ by a much more general 
nonhomogeneous model, including models that are highly sparse
or whose entries exhibit strong dependence;
\item We may replace $\theta\,vv^*$ by much more
general perturbations whose rank may diverge at a rate determined by 
the fluctuations of the spectral statistics of $G$;
\item We can obtain fully nonasymptotic results that provide
explicit guarantees for the given random matrix $X$ of fixed dimension 
$d$.
\end{enumerate}
None of these features are readily accessible by prior work in this area. 
Most methods that have been used to study outliers of random matrices rely 
strongly on mean-field (that is, near-homogeneous) random matrix structure 
\cite{CD17,BM21,LL22}. Some results for sparse models have appeared only 
very recently \cite{LM22,Au23}, but rely on restrictive assumptions and 
make use of specialized tools.

In principle, our theory can be applied to an arbitrarily structured 
nonohomogeneous spiked model. However, the phase transition behavior will 
be determined in general by a complicated variational principle that does 
not have a simple analytic solution. Instead, we will primarily focus our 
attention on two general classes of models that can be understood in an 
explicit closed form:
\begin{enumerate}[1.]
\itemsep\medskipamount
\item 
Models of the form $X=A+G$ where $A$ has rank $o(d)$ and the noise matrix 
$G$ is \emph{isotropic}, that is, $\mathbf{E}[G]=0$ and $\mathbf{E}[G^2]=\id$.
\item
A class of explicitly solvable \emph{anisotropic} models due to
Krzakala et al.\ (e.g., \cite{KKP23}) that arise by
linearizing message passing algorithms of statistical physics.
\end{enumerate}
These general results are established in strong nonasymptotic form that 
can be applied in a black box manner in complex situations.

It should be emphasized that the specific features of spiked models are 
completely irrelevant to our sharp matrix concentration theory: our theory 
reduces the study of arbitrarily structured random matrices to the 
question of understanding $X_{\rm free}$. The difficulty in understanding 
the above models lies entirely in the latter deterministic question. To 
this end, we will make fundamental use of a formula of Lehner
\begin{equation}
\label{eq:lehner}
	\lambda_{\rm max}(X_{\rm free}) = 
	\inf_{M>0} \lambda_{\rm max}\big(
	\mathbf{E} X + 
	M^{-1} + 
	\mathbf{E}[(X-\mathbf{E}X)M(X-\mathbf{E}X)]
	\big)
\end{equation}
for arbitrary self-adjoint random matrices $X$ (cf.\ 
\cite[Corollary 1.5]{Leh97} and \cite[\S 4.1]{BBV23}).
In our study of spiked models, we will develop methods to 
capture the structure of this variational principle that may be useful
also in other applications.

\subsection{Applications}

As a direct consequence of the general phase transition results described 
above, we are able to elucidate phase transition phenomena that arise in a 
diverse set of applied mathematical problems, including:
\begin{enumerate}[$\bullet$]
\itemsep\medskipamount
\item 
We establish the sharp threshold for recovery of node labels on graphs 
from noisy observations of the edges by a simple spectral method 
\cite{ABBS14,Cuc15}, under natural geometric assumptions on the 
underlying graph.
\item We establish the sharp detection threshold
of an important family of spectral methods for the tensor PCA 
problem \cite{MR14,WEM19}, providing the only
sharp phase transition known to date for any tensor PCA algorithm for 
symmetric tensors.
\item We prove a conjecture of Pak, Ko, and Krzakala \cite[Conjecture 
1.6]{KKP23} on the outlier phase transitions in a block-structured 
variant of the spiked Wigner model.\footnote{%
	This conjecture is proved here in a strong nonasymptotic form.
	A weaker asymptotic form of this conjecture was proved 
	independently in \cite{MKK24} concurrently with our work.
}
\item We prove a conjecture of Duranthon, Krzakala, and Zdeborov\'a
(personal communication) on the recovery threshold of contextual 
stochastic block models \cite{DMMS18} by means of a spectral 
algorithm that is based on ideas of statistical physics.
\end{enumerate}
These applications will be introduced in detail in section
\ref{sec:appl}.
Each of these problems
is of considerable interest in its own right and
had remained open prior to this work. We emphasize 
that the theory of the present paper not only makes it possible to resolve 
such problems, but that each of these diverse applications is deduced with 
minimal effort from the general results described in section 
\ref{sec:introapp}.

At the same time, problems that do not belong to one of the general 
classes described in the previous section can still be studied on a 
case-by-case basis using similar techniques. As an illustration, we will 
prove a conjecture of Han \cite{Han22} on a phase transition phenomenon of 
centered sample covariance matrices.

\subsection{Organization of this paper}

This paper is organized as follows. In section \ref{sec:main}, we 
formulate the main results of this paper: two-sided sharp matrix 
concentration inequalities for spectral edges, and ``master theorems'' for 
the two general classes of spiked models described above. In section 
\ref{sec:appl}, we will illustrate the applicability of these results in a 
diverse range of applied mathematical problems.

The rest of the paper is devoted to the proofs of these results.
Section \ref{sec:ultra} develops the key ultracontractive estimates
that are needed in the proofs of two-sided bounds on the spectral edges.
The latter will subsequently be proved in section \ref{sec:pfmain}.
Sections \ref{sec:pfiso} and \ref{sec:pfaniso} are devoted to the proofs 
of the two ``master theorems'' for spiked models. Finally, section 
\ref{sec:pfappl} contains proofs of some results needed in the 
applications.

\subsection{Notation}

The following notation will be used throughout this paper. For a bounded 
operator $X$, we denote by $\|X\|$ its operator norm and by 
$|X|:=(X^*X)^{\frac{1}{2}}$ its modulus. If $X$ is self-adjoint, we denote 
its spectrum as $\spc(X)$, and denote by $\lambda_{\rm 
max}(X):=\sup\spc(X)$ and $\lambda_{\rm min}(X):=\inf\spc(X)$ the upper 
and lower edges of the spectrum. The identity operator or matrix is 
denoted as $\id$. The algebra of $d\times d$ matrices with entries in a 
$*$-algebra $\mathcal{A}$ is denoted as $\mathrm{M}_d(\mathcal{A})$, and 
its subspace of self-adjoint elements is $\mathrm{M}_d(\mathcal{A})_{\rm 
sa}$. For $M\in\mathrm{M}_d(\mathbb{C})$, we denote by $\tr M := 
\sum_{i=1}^d M_{ii}$ its unnormalized trace and by $\ntr M := 
\frac{1}{d}\tr M$ its normalized trace.
Finally, we use the convention that when a functional is followed by 
square brackets, it is applied before any other operations; for example,
$\mathbf{E}[X]^\alpha := (\mathbf{E}X)^\alpha$ and
$\ntr[M]^\alpha := (\ntr M)^\alpha$.

\section{Main results}
\label{sec:main}

\subsection{Basic model and matrix parameters}

Throughout this paper, we fix $d\ge 2$ and consider a $d\times d$ 
random matrix $X$ with jointly Gaussian entries. Such a random matrix 
can always be represented (for some $n\in\mathbb{N}$) as 
\begin{equation}
\label{eq:x}
	X = A_0 + \sum_{i=1}^n A_i g_i,
\end{equation}
where $g_1,\ldots,g_n$ are i.i.d.\ standard 
Gaussian variables and $A_0,\ldots,A_n\in\mathrm{M}_d(\mathbb{C})$.

To the given random matrix $X$, we associate a corresponding 
noncommutative model $X_{\rm free}\in\mathrm{M}_d(\mathcal{A})\simeq
\mathrm{M}_d(\mathbb{C})\otimes\mathcal{A}$ defined as
\begin{equation}
\label{eq:xfree}
	X_{\rm free} = A_0\otimes \id +
	\sum_{i=1}^n A_i\otimes s_i,
\end{equation}
where $s_1,\ldots,s_n$ is a free semicircular family in some 
$C^*$-probability space $(\mathcal{A},\tau)$. We refer to
\cite[\S 4.1]{BBV23} for a very brief introduction and
\cite{NS06} for a pedagogical treatment of the basic notions 
of free probability. The main outcome of the sharp matrix 
concentration theory of \cite{BBV23} and of the present paper is 
that, in many situations, the spectrum of $X$ is well approximated
by that of $X_{\rm free}$. 

Before we formulate our main results, we recall the definitions of
the most common parameters that will appear in our bounds. In the 
following, we denote by
$$
	\mathrm{Cov}(X)_{ij,kl} :=
	\EE\big[
	(X-\EE X)_{ij} \overline{(X-\EE X)_{kl}}\big]
$$
the $d^2\times d^2$ covariance matrix of the entries of $X$.
We now 
define 
the parameters
\begin{align*}
	\sigma(X)^2 &:=
	\|\EE[(X-\EE X)^*(X-\EE X)]\| \vee \|\EE[(X-\EE X)(X-\EE X)^*]\|,\\
	v(X)^2 &:= 
	\sup_{\tr |M|^2=1} \EE[|{\tr[M(X-\EE X)]}|^2] =
	\|\mathrm{Cov}(X)\|,
\\
	\sigma_*(X)^2 &:=
	\sup_{\|v\|=\|w\|=1} \EE[|\langle v,(X-\EE X)w\rangle|^2],
\end{align*}
as well as the frequently appearing combination
$$
	\tilde v(X)^2 := v(X)\sigma(X).
$$
We emphasize that these parameters depend only on $\mathrm{Cov}(X)$ and 
not on $\EE X=A_0$. All these parameters are readily expressed explicitly 
in terms of $A_1,\ldots,A_n$ and we have 
$\sigma_*(X)\le\sigma(X)$ and $\sigma_*(X)\le v(X)$, 
cf.\ \cite[\S 2.1]{BBV23}. 

\begin{rem}
Roughly speaking, these parameters will play the following roles in our 
theory: $\sigma(X)$ controls the scale of the spectrum of $X-\EE X$ (e.g., 
by \eqref{eq:nck}); $v(X)$ controls the degree to which the 
spectrum of $X$ is approximated by $X_{\rm free}$; and 
$\sigma_*(X)$ captures the fluctuations of the spectrum of $X$.
\end{rem}

Finally, let us note that we will often restrict attention in the 
formulation and proofs of our results to self-adjoint random matrices $X$ 
(that is, random matrices defined by 
$A_0,\ldots,A_n\in\mathrm{M}_d(\mathbb{C})_{\rm sa}$). This entails no 
loss of generality, however, as results on the spectrum of self-adjoint 
operators extend directly to the singular value spectrum of 
non-self-adjoint operators by means of a standard dilation argument 
\cite[Remark 2.6]{BBV23}. For this reason, we will formulate some of our 
main results for self-adjoint matrices whenever this leads to greater 
notational simplicity.

\subsection{Sharp matrix concentration inequalities}

\subsubsection{Gaussian random matrices}

Recall that the Hausdorff distance between two subsets 
$A,B\subseteq\mathbb{R}$ of the real line is defined as
$$
	\mathrm{d_H}(A,B) :=
	\inf\{\varepsilon>0: A\subseteq B+[-\varepsilon,\varepsilon]
	\text{ and }B\subseteq A+[-\varepsilon,\varepsilon]\}.
$$
The following is the central result of this paper.

\begin{thm}
\label{thm:main}
For any $d\times d$ self-adjoint Gaussian random matrix $X$, we have
$$
	\mathbf{P}\big[
	\mathrm{d_H}(\spc(X),\spc(X_{\rm free}))
	> C\tilde v(X) (\log d)^{\frac{3}{4}}
	+ C\sigma_*(X)t
	\big] \le e^{-t^2}
$$
for all $t\ge 0$, where $C$ is a universal constant.
\end{thm}

Theorem \ref{thm:main} controls the entire spectrum of $X$ and $X_{\rm 
free}$. As the spectral edges are often of special interest, we spell 
out the following simple corollary.

\begin{cor}
\label{cor:norm}
Let $X$ be an arbitrary (not necessarily self-adjoint) $d\times d$
Gaussian random matrix. Then we have
$$
	\mathbf{P}\big[
	|\|X\|-\|X_{\rm free}\||
	> C\tilde v(X) (\log d)^{\frac{3}{4}}
	+ C\sigma_*(X)t
	\big] \le e^{-t^2}
$$
for all $t\ge 0$ and
$$
	|\mathbf{E}\|X\|-\|X_{\rm free}\||
	\le C\tilde v(X) (\log d)^{\frac{3}{4}},
$$
where $C$ is a universal constant.  If $X$ is self-adjoint, the   
same inequalities hold if $\|X\|,\|X_{\rm free}\|$ are replaced by
$\lambda_{\rm max}(X),\lambda_{\rm max}(X_{\rm free})$ or by
$\lambda_{\rm min}(X),\lambda_{\rm min}(X_{\rm free})$.
\end{cor}

Theorem \ref{thm:main} and Corollary \ref{cor:norm} will be proved in 
section \ref{sec:pfmain}.
Note that Theorem~\ref{thm:normtwoside} in the introduction is merely a 
special case of Corollary \ref{cor:norm}.

\subsubsection{Non-Gaussian random matrices}

While the above inequalities are formulated for Gaussian random matrices, 
random matrices that arise in applications are often non-Gaussian. The 
Gaussian case is nonetheless of central importance, as the behavior of 
many non-Gaussian matrices can be understood in terms of an associated 
Gaussian model. For ease of reference, we presently state two general 
results of this kind that will be used in the applications in section 
\ref{sec:appl}.

One widely used non-Gaussian random matrix model is
\begin{equation}
\label{eq:series}
	Z = Z_0 + \sum_{i=1}^n Z_i,
\end{equation}
where $Z_0\in\mathrm{M}_d(\mathbb{C})$ is a nonrandom matrix and
$Z_1,\ldots,Z_n$ are arbitrary independent $d\times d$ random matrices
with $\EE Z_i=0$. Such models arise naturally in many applications 
\cite{Tro15,TV23}. We presently state a universality principle 
\cite[Theorem 2.6]{TV23} that reduces the study of such models to the 
Gaussian case.

\begin{thm}[\cite{TV23}]
\label{thm:univ}
Let $Z$ be a $d\times d$ self-adjoint random matrix as in 
\eqref{eq:series}, and suppose that $\|Z_i\|\le R$ a.s.\ for 
$i=1,\ldots,n$. Let $X$ be the $d\times d$ 
Gaussian random matrix whose entries have the same mean and covariance as 
those of $Z$. Then
$$
	\mathbf{P}\big[
	\mathrm{d_H}(\spc(Z),\spc(X)) > 
	C\sigma_*(X)t^{\frac{1}{2}} +
	CR^{\frac{1}{3}}\sigma(X)^{\frac{2}{3}}t^\frac{2}{3} +
	CRt\big] \le de^{-t}
$$
for all $t\ge 0$, where $C$ is a universal constant.
\end{thm}

Theorem \ref{thm:univ} shows that $Z$ behaves as a Gaussian random matrix 
$X$, while Theorem \ref{thm:main} shows that $X$ behaves as its 
noncommutative model $X_{\rm free}$. The combination of these two theorems 
therefore provides a powerful tool to study a large class of non-Gaussian 
random matrices. A variant of Theorem \ref{thm:univ} that is applicable 
when $Z_i$ are unbounded may be found in \cite[Theorem 2.7]{BBV23}.

A non-Gaussian model of a different nature arises in the study of sample 
covariance matrices, which requires an understanding of the spectra of 
quadratic polynomials of a Gaussian random matrix $X$ such as $XX^*-\EE 
XX^*$. Such models are captured by the following quadratic analogue of 
Theorem \ref{thm:main}.

\begin{thm}
\label{thm:quad}
Let $X$ be any (not necessarily self-adjoint) $d\times d$ Gaussian random 
matrix, and let $B\in\mathrm{M}_d(\mathbb{C})_{\rm sa}$. Then we have
$$
	\mathbf{P}\big[
	\mathrm{d_H}(\spc(XX^*+B),
		\spc(X_{\rm free}X_{\rm free}^* + B\otimes\id))
	>
	C\{\|X_{\rm free}\|+\|B\|^{\frac{1}{2}}\}t
	+Ct^2
	\big]
	\le e^{-\frac{t^2}{\sigma_*(X)^2}}
$$
for all $t\ge \tilde v(X)(\log d)^{\frac{3}{4}}$,
where $C$ is a universal constant.
\end{thm}

Theorem \ref{thm:quad} will be proved in section \ref{sec:pfmain}.

\begin{rem}
Theorem \ref{thm:quad} can be extended in two directions. 
On the one hand, we may consider models where $X$ itself is 
non-Gaussian as in \cite[\S 3.4]{TV23}. On the other hand, the method of 
proof of Theorem \ref{thm:main} can be adapted to bound
general noncommutative polynomials $P(X_1,\ldots,X_m)$ of Gaussian 
random matrices $X_1,\ldots,X_m$ in terms of their noncommutative models 
$P(X_{1,\rm free},\ldots,X_{m,\rm free})$. As such extensions digress
from the main theme of this paper, we do not develop them further here.
\end{rem}

\subsection{Phase transitions: isotropic case}
\label{sec:isophase}

The theorems stated above explain the spectral properties of very general 
random matrices in terms of the noncommutative model $X_{\rm free}$. To 
understand specific phenomena, it therefore remains to understand the 
corresponding behavior of $X_{\rm free}$. We presently formulate a number 
of results that enable the study of spectral outliers in a broad range of 
models.

It will be convenient to define the function
$$
	\mathrm{B}(\theta) := \begin{cases}
	2 & \text{for }\theta\le 1,\\
	\theta+\frac{1}{\theta} & \text{for }\theta>1.
	\end{cases}
$$
As was discussed in section \ref{sec:introapp}, this function describes 
the largest eigenvalue of the classical spiked Wigner model. The following 
result may be viewed as a far-reaching generalization of this phenomenon.

\begin{thm}
\label{thm:bbp}
Let $X$ be any $d\times d$ self-adjoint random matrix.
Suppose that
$\EE[(X-\EE X)^2]=\id$
and that $\EE X$ has rank $r$ with $\sigma_*(X)\sqrt{r}\le 1$. Then
$$
	|\lambda_{\rm max}(X_{\rm free}) -
	\mathrm{B}(\lambda_{\rm max}(\EE X))|
	\le 2\sigma_*(X)\sqrt{r}.
$$
\end{thm}

When combined with sharp matrix concentration inequalities, this yields a 
phase transition for a broad range of models: the isotropic assumption 
$\EE[(X-\EE X)^2]=\id$ holds in many (including sparse or dependent) 
applications, while $\sigma_*(X)\sqrt{r}=o(1)$ typically allows the rank 
to grow rapidly with dimension. 

\begin{rem}
\label{rem:bulk}
Let us briefly explain in what sense Theorem \ref{thm:bbp} captures an 
outlier of the spectrum. Denote by $\mu_X:=\frac{1}{d}\sum_{i=1}^d
\delta_{\lambda_i(X)}$ the empirical spectral distribution of $X$.
We begin by recalling the deterministic fact (see, e.g., \cite{HS02}) that 
$$
	\bigg|
	\int f\,d\mu_X - \int f\,d\mu_{X-\EE X}
	\bigg| =
	|{\ntr f(X)}-{\ntr f(X-\EE X)}| 
	\le \frac{r}{d}\|f'\|_{L^1(\mathbb{R})} = o(1)
$$
for any $f\in C^\infty_0(\mathbb{R})$
when $\EE X$ has rank $r=o(d)$.

Suppose the relevant matrix parameters are 
sufficiently small that the spectrum of $X$ is well modelled by that of
$X_{\rm free}$. Then Theorem~\ref{thm:bbp} (applied to $X\leftarrow X-\EE 
X$) implies that $\lambda_{\rm max}(X-\EE X)=2+o(1)$. Therefore, by the 
above deterministic fact, a fraction $1-o(1)$ of the eigenvalues of $X$ 
are bounded by $2+o(1)$.

On the other hand, Theorem~\ref{thm:bbp} (applied to $X$ itself) shows 
that $X$ has an eigenvalue at $\mathrm{B}(\lambda_{\rm max}(\EE 
X))>2+\varepsilon$ when $\lambda_{\rm max}(\EE X)>1$. This outlier 
eigenvalue is therefore bounded away from the bulk of the spectrum.
\end{rem}

When there is an outlier, it is expected that the eigenvector associated 
to the largest eigenvalue of $X$ yields information on the eigenvectors of 
$\EE X$. This behavior is readily deduced from Theorem \ref{thm:bbp} 
by the following device. Here $1_A(M)$ is defined by functional calculus,
that is, it is the projection onto the space spanned by the 
eigenvectors of $M\in\mathrm{M}_d(\mathbb{C})_{\rm sa}$ with eigenvalues 
in $A\subseteq\mathbb{R}$.

\begin{thm}
\label{thm:eigenvec}
Let $X$ be any $d\times d$ self-adjoint random matrix with
$\lambda_{\rm max}(\EE X)=:\theta$, and fix $0<t\le\delta$.
Define $X_s := X + s 1_{(\theta-\delta,\theta]}(\EE X)$, and suppose that
for $s\in\{0,\pm t\}$
$$
	\mathbf{P}\big[
	|\lambda_{\rm max}(X_s)
	-\mathrm{B}(\lambda_{\rm max}(\EE X_s))|
	> \varepsilon\big]\le \rho.
$$
Then any unit norm eigenvector $v_{\rm max}(X)$ of 
$X$ with eigenvalue $\lambda_{\rm max}(X)$ satisfies
$$
	\mathbf{P}\bigg[
	\bigg|
	\langle v_{\rm max}(X),
	1_{(\theta-\delta,\theta]}(\EE X)
	v_{\rm max}(X)\rangle 
	-
	\bigg(
	1-\frac{1}{\theta^2}
	\bigg)_+	
	\bigg| >
	t+\frac{2\varepsilon}{t}
	\bigg]\le 3\rho.
$$
\end{thm}

For example, suppose the largest eigenvalue $\theta$ of $\EE X$ is simple,
and that there is a gap of size $\delta$ between the 
largest and second-largest eigenvalues of $\EE X$. Then Theorem 
\ref{thm:eigenvec} yields $|\langle v_{\rm max}(\EE X),v_{\rm 
max}(X)\rangle|^2= (1-\frac{1}{\theta^2})_++o(1)$,
provided that the lower-order terms that arise from the sharp matrix 
concentration inequalities and from Theorem \ref{thm:bbp} (i.e., 
$\varepsilon$ in Theorem \ref{thm:eigenvec}) are $o(\delta)$.

We have deliberately formulated the above results independently of any 
specific random matrix model so that they can be applied equally easily 
in the context of either Theorem \ref{thm:main} or Theorem \ref{thm:univ}; 
applications to concrete models will be illustrated in in section 
\ref{sec:appl}. The above results are proved in section \ref{sec:pfiso}.

\begin{rem}
While we have focused our treatment of outliers on the largest 
eigenvalue, one may also investigate 
other outliers in the spectrum using Theorem \ref{thm:main}.
As the above results already suffice for all the applications
we will consider, we omit further development of such questions in the
interest of space.
\end{rem}

\subsection{Phase transitions: an anisotropic model}
\label{sec:krzdefn}

In principle, there is nothing special about the isotropic assumption 
$\EE[(X-\EE X)^2]=\id$ made in Theorem \ref{thm:bbp}: we can establish an 
analogous phase transition for an arbitrarily structured self-adjoint 
random matrix $X$ so that $\EE X$ has low rank. However, for anisotropic 
models the phase transition will generally not admit a simple description; 
it is determined by the solution to the variational problem 
\eqref{eq:lehner}, which cannot be expected to yield an analytic solution 
in the absence of some special structure. We presently discuss a class of 
anisotropic models where such special structure is present.

To define the model, we fix the following parameters:
\medskip
\begin{enumerate}[$\bullet$]
\itemsep\medskipamount
\item A partition of $[d]=C_1\sqcup\cdots\sqcup C_q$ 
into $q$ disjoint sets of size $|C_k|>1$.
\item A matrix $B\in\mathrm{M}_q(\mathbb{R})_{\rm sa}$ with
nonnegative entries.
\item A vector $z\in\mathbb{R}^d$ such that
$\sum_{i\in C_k} z_i^2 = |C_k|$ for $k=1,\ldots,q$.
\end{enumerate}
\medskip
Let $\mathbf{B}\in\mathrm{M}_{d}(\mathbb{R})_{\rm sa}$ be the block
matrix defined by $\mathbf{B}_{ij}:=B_{kl}$ for all $i\in C_k,j\in C_l$.
Then we consider the $d\times d$ random matrix $X$ defined as\footnote{%
Here $1_d\in\mathbb{R}^d$ is the vector whose entries are all equal to 
one. As the analysis of this model involves both $q$- and $d$-dimensional
vectors and matrices, we will indicate the dimension of the ones vector 
$1_d$ and of the identity matrix $\id_d$ in subscript to avoid confusion.}
\begin{equation}
\label{eq:krz}
\begin{aligned}
	&X = \frac{1}{d}\diag(z)\mathbf{B}\diag(z) + X_\varnothing,
	\\
	&X_{\varnothing} = 
	-\diag\bigg(\frac{1}{d}\mathbf{B} 1_d\bigg)
	+ G,
\end{aligned}
\end{equation}
where $G$ is a $d\times d$ real symmetric random matrix whose entries
$(G_{ij})_{i\ge j}$ are independent with $\mathbf{E}[G_{ij}]=0$
and $\mathbf{E}[G_{ij}^2]=\frac{1+1_{i=j}}{d}\mathbf{B}_{ij}$. 
To interpret the structure of this model, note that $X$ is a low-rank 
perturbation of $X_{\varnothing}$ when $q\ll d$ as 
$\mathrm{rank}(\mathbf{B})\le q$. 
We aim to understand the resulting
outlier phase transition.

\begin{rem}
The significance of this model is that random matrices of the form 
\eqref{eq:krz} arise in applications as a linearization of message 
passing algorithms of statistical physics. Two such applications are 
discussed in sections \ref{sec:krz} and \ref{sec:csbm}.
\end{rem}

In the following, we will assume that the nonnegative matrix $B$ is 
irreducible. This entails no loss of generality: if $B$ is reducible, then 
$X$ is block-diagonal and it suffices to consider its irreducible blocks. 
We denote by $c,b\in\mathbb{R}^q$ the vector with entries
$c_k=\frac{|C_k|}{d}$ and the Perron-Frobenius (right) eigenvector $b>0$ 
of $B\diag(c)$.

We are now ready to formulate the analogue of Theorem \ref{thm:bbp} in the 
present setting. Such a phase transition was first conjectured in
\cite{KKP23} (see section \ref{sec:krz}).

\begin{thm}
\label{thm:krz}
Let $X,X_\varnothing$ be defined as in \eqref{eq:krz}, and suppose all
the above assumptions are in force. Then there exist 
$\lambda,\lambda_\varnothing\in\mathbb{R}$ so that
$$
	|\lambda_{\rm max}(X_{\rm free})-\lambda| \le
	\sqrt{\frac{8\|B1_q\|_{\infty}}{d}},
	\qquad
	|\lambda_{\rm max}(X_{\varnothing,\rm free})-\lambda_\varnothing| 
	\le
	\sqrt{\frac{8\|B1_q\|_{\infty}}{d}},
$$
where $\lambda_\varnothing$ satifies
$$
	\lambda_\varnothing \le 
	1-
	\frac{\min_i b_i}{\max_i b_i}
	\big(1-\lambda_{\rm max}(\diag(c)^{\frac{1}{2}}B
		\diag(c)^{\frac{1}{2}})^{\frac{1}{2}}\big)^2,
$$
while $\lambda$ exhibits the following phase transition.
\begin{enumerate}[a.]
\item 
If $\lambda_{\rm max}(\diag(c)^{\frac{1}{2}}B\diag(c)^{\frac{1}{2}}) < 1$,
then $\lambda_\varnothing=\lambda<1$.
\item 
If $\lambda_{\rm max}(\diag(c)^{\frac{1}{2}}B\diag(c)^{\frac{1}{2}}) = 1$,
then $\lambda_\varnothing=\lambda=1$.
\item 
If $\lambda_{\rm max}(\diag(c)^{\frac{1}{2}}B\diag(c)^{\frac{1}{2}}) > 1$,
then $\lambda_\varnothing<\lambda=1$.
\end{enumerate}
\end{thm}

To interpret this result, note that by the same argument as in Remark 
\ref{rem:bulk}, the bulk of the spectrum of $X$ is bounded by 
$\lambda_\varnothing$. The largest eigenvalue of $X$ is therefore an 
outlier precisely when $\lambda_{\rm 
max}(\diag(c)^{\frac{1}{2}}B\diag(c)^{\frac{1}{2}}) > 1$, and when the 
outlier appears it is always at $\lambda=1$. Moreover, the bound on 
$\lambda_\varnothing$ yields an explicitly computable estimate on how far 
the outlier lies from the bulk, which is essential for the applicability 
of the result in nonasymptotic situations.

The proof of Theorem \ref{thm:krz} in section \ref{sec:pfaniso} will 
provide explicit variational expressions for $\lambda,\lambda_\varnothing$ 
that are not analytically tractable in general. It is a remarkable feature 
of this model that we can nonetheless describe the phase transition in 
terms of the explictly computable parameter $\lambda_{\rm 
max}(\diag(c)^{\frac{1}{2}}B\diag(c)^{\frac{1}{2}})$. The proof of this 
fact requires a number of ideas and tools for analyzing Lehner-type 
variational principles that may be useful also in other applications.

\begin{rem}
\label{rem:ohnonooverlap}
It would be of interest to establish a counterpart of
Theorem \ref{thm:eigenvec} in the present setting, which yields
quantitative bounds on the overlap between $z$ and the largest eigenvector 
of $X$. While it is rather easy to read off the correct behavior of the 
overlap from the proof of Theorem \ref{thm:krz} in the asymptotic setting where
$q,B,c$ are fixed and $d\to\infty$ (see section \ref{sec:krz}), it has
proved more challenging to obtain an explicitly computable nonasymptotic
estimate for the overlap in the present setting. We leave this as an open 
problem.
\end{rem}

\section{Applications}
\label{sec:appl}

\subsection{Simple examples}
\label{sec:simpleex}

For sake of illustration, we begin by spelling out the simplest form of 
the phase transition phenomenon that arises from section 
\ref{sec:isophase}.

\begin{thm}
\label{thm:simplebbp}
Let $G$ be any $d\times d$ self-adjoint random matrix with $\EE G=0$ and
$\EE G^2=\id$, which either has jointly Gaussian entries or is of the 
form \eqref{eq:series}. Let $\theta\ge 0$ and $v\in S^{d-1}$. 
Then $X = \theta\,vv^* + G$ satisfies
$$
	\mathbf{P}\big[
	|\lambda_{\rm max}(X)-\mathrm{B}(\theta)|
	> C\varepsilon(t)
	\big] \le Cde^{-t}
$$
for all $t>0$, and
$$
	\mathbf{P}\Big[
	\Big| |\langle v,v_{\rm max}(X)\rangle|^2
	- \Big(1-\tfrac{1}{\theta^2}\Big)_+\Big|^2
	>
	C\varepsilon(t)
	\Big] \le Cde^{-t}
$$
whenever $C\varepsilon(t)\le\theta^2$. Here
$\varepsilon(t)=
v(G)^{\frac{1}{2}}(\log d)^{\frac{3}{4}}+\sigma_*(G)t^{\frac{1}{2}}$ in 
the Gaussian 
case and $\varepsilon(t) = 
v(G)^{\frac{1}{2}}(\log d)^{\frac{3}{4}}+\sigma_*(G)t^{\frac{1}{2}} +
R^{\frac{1}{3}}t^{\frac{2}{3}}+Rt$
in the setting of Theorem \ref{thm:univ}.
\end{thm}

\begin{proof}
The first inequality follows by combining Corollary \ref{cor:norm},
Theorem \ref{thm:univ}, and Theorem \ref{thm:bbp},
where we note that $\sigma(X)=\sigma(G)=1$, $v(X)=v(G)$,
$\sigma_*(X)=\sigma_*(G)$ and the parameter $R$ in Theorem \ref{thm:univ} 
are independent of $\theta$, and that $\sigma_*(G)\le\tilde v(G)$.
The second inequality follows from the first by applying Theorem 
\ref{thm:eigenvec} with $\delta=\theta$ and optimizing over the parameter 
$t$ that appears in its statement.
\end{proof}

As a simple example, let us consider sparse Wigner matrices.

\begin{example}
Let $([d],E)$ be a $k$-regular graph with $d$ vertices. Then we can define 
a $d\times d$ self-adjoint random matrix $G$ with 
$G_{ij}=k^{-\frac{1}{2}}\eta_{ij} 1_{\{i,j\}\in E}$ for $i\ge j$, where
$\eta_{ij}$ are independent random variables such that $\EE[\eta_{ij}]=0$,
$\EE[|\eta_{ij}|^2]=1$, 
$\|\eta_{ij}\|_\infty\le K$. In other words, $G$ is a sparse Wigner matrix
with an arbitrary deterministic sparsity pattern that has $k$ nonzero 
entries in each row and column.

It is readily verified that we have $\EE[G]=0$, $\EE[G^2]=\id$, 
$\sigma_*(G)\le v(G)\lesssim k^{-\frac{1}{2}}$, and
$R\lesssim Kk^{-\frac{1}{2}}$. Thus choosing $t=(1+c)\log d$ in Theorem 
\ref{thm:simplebbp} shows that
$$
	\lambda_{\rm max}(\theta\,vv^*+G) = \mathrm{B}(\theta)+o(1),
	\qquad
	|\langle v,v_{\rm max}(\theta\,vv^*+G)\rangle|^2  
        = \Big(1-\tfrac{1}{\theta^2}\Big)_+ + o(1)
$$
with probability at least $1-\frac{C}{d^c}$ whenever
$k\gg K^2(\log d)^4$.

One very special case of this model is a periodic random band matrix with 
band width $k$, which is illustrated in Figure \ref{fig:simplebbp}. In 
this case, as long as $K=O(1)$, we find that the 
classical phase transition for the spiked Wigner model (as described in 
section \ref{sec:introapp}) extends to this 
highly sparse setting as soon as the width of the band grows at least 
polylogarithmically in the dimension of the matrix. This special case was 
recently investigated using entirely different methods in \cite{Au23}.
\end{example}
\begin{figure}
\centering
\includegraphics[width=0.9\textwidth]{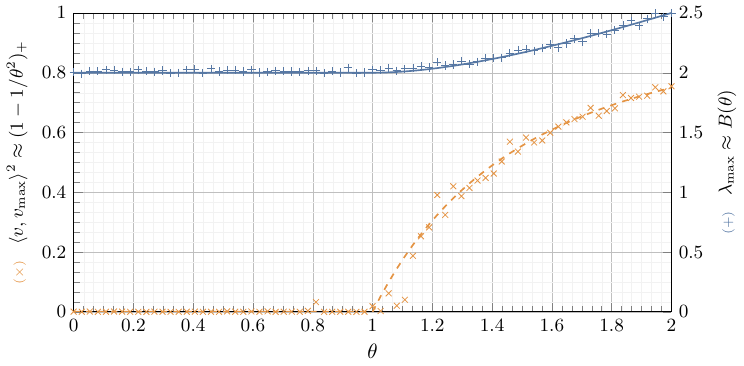}
\caption{Outlier phase transition of a $2000\times 2000$ random band
matrix with band width $101$. The markers represent the empricial result,
the lines represent the theoretical prediction
of Theorem \ref{thm:simplebbp}.\label{fig:simplebbp}}
\end{figure}

While Theorem \ref{thm:simplebbp} provides precise information
on the largest eigenvalue and eigenvector, it does not in itself explain 
why this largest eigenvalue is an outlier of the spectrum when $\theta>1$.
The rough explanation given in Remark \ref{rem:bulk} can however readily 
be made precise in concrete situations such as the present one.

\begin{cor}
\label{cor:minmax}
Consider the setting of Theorem \ref{thm:simplebbp}, and denote by 
$\lambda_2(X)$ the second largest eigenvalue of $X$. Then we have for
all $t>0$
$$
	\mathbf{P}\big[\lambda_2(X) > 2+C\varepsilon(t)\big]
	\le Cde^{-t}.
$$
\end{cor}

\begin{proof}
Let $P=\id-vv^*$. Then the min-max theorem yields
$$
	\lambda_2(X) \le \lambda_{\rm max}(PXP)=\lambda_{\rm max}(PGP)
	\le\lambda_{\rm max}(G).
$$
The conclusion follows by applying Theorem \ref{thm:simplebbp} with 
$\theta=0$.
\end{proof}

Corollary \ref{cor:minmax} shows that whenever $\varepsilon(t)=o(1)$, at 
most one eigenvalue of $X$ can exceed $2+o(1)$. When $\theta>1$, Theorem 
\ref{thm:simplebbp} then implies that the largest eigenvalue of $X$ is 
simple and is separated from the rest of the spectrum.

Theorem \ref{thm:simplebbp} could be applied directly or with minimal 
modifications to various models that appear in applications (both with 
independent and dependent entries), such as adjacency matrices of 
nonhomogeneous random graphs \cite{LM22}, stochastic block models 
\cite{LL22}, or synchronization problems \cite[\S 7]{Sin11}. In the 
following sections, we will investigate applications that exhibit
more complex structures.

\subsection{Decoding node labels on graphs}

The following model is considered in \cite{ABBS14}. Let $\Gamma=([d],E)$ 
be a given $k$-regular graph with $d$ vertices, and let $x\in \{-1,+1\}^d$ 
be an (unknown) binary labeling of the vertices. For each edge $\{i,j\}\in 
E$, we are given a noisy observation $Y_{ij}=x_ix_j\xi_{ij}$ 
of the correlation between the labels of the incident vertices, where 
$\xi_{ij}=\xi_{ji}$ are i.i.d.\ random variables for $i\ge j$ such that
$$
	\mathbf{P}[\xi_{ij}=1]=1-p,\qquad\quad
	\mathbf{P}[\xi_{ij}=-1]=p
$$
with $p\le\frac{1}{2}$.
The aim is to understand when it is possible to recover the vertex labels 
(up to a global sign) from the noisy observations. 

Here we investigate a simple spectral method for this problem (see, e.g., 
\cite{Cuc15}).
Let us augment the above definitions by setting $Y_{ij}=0$ for 
$\{i,j\}\not\in E$, so that $Y$ defines a $d\times d$ self-adjoint 
random matrix with independent entries. Note that
$$
	\mathbf{E}[Y_{ij}] = (1-2p) x_ix_j1_{\{i,j\}\in E} ,\qquad
	\mathrm{Var}(Y_{ij}) = 4p(1-p)1_{\{i,j\}\in E} .
$$
Thus $\EE[(Y-\EE Y)^2] = 4kp(1-p)\id$, where we used that $\Gamma$ is 
$k$-regular. We may therefore apply Theorems
\ref{thm:bbp} and \ref{thm:eigenvec} to a suitable modification of $Y$.
In the following result, we will consider a sequence of graphs
with $d,k\to\infty$ for simplicity of exposition; a nonasymptotic 
statement can be read off from the proof.

\begin{thm}
\label{thm:masked}
Let $A$ be the adjacency matrix of $\Gamma$, and denote its singular 
values as $k=\mathrm{s}_1\ge\mathrm{s}_2\ge \cdots\ge\mathrm{s}_d$
and its spectral gap as $\lambda = k-\lambda_2(A)$.
Parametrize the error probability as
$p = \frac{1}{2}-\frac{1}{2}k^{-\frac{1}{2}}\theta$
with $0\le\theta\ll\sqrt{k}$. If
$$
	\min_{1\le r\le k}
	\big\{\theta\,\mathrm{s}_{r+1}+ 
	\sqrt{rk}\big\}
	+ k^{\frac{5}{6}} (\log d)^{\frac{2}{3}}  
	\ll \min\{\theta\lambda,k\},
$$
then 
$$
	\frac{1}{d} |\langle x,v_{\rm max}(Y)\rangle|^2 =
	\Big(1-\tfrac{1}{\theta^2}\Big)_+ + o(1)	
$$
with probability $1-o(1)$. In particular, there exists an estimator
$\hat x(Y)\in \{-1,+1\}^d$ so that $\frac{1}{d}|\langle x,\hat x(Y)\rangle|>
\delta+o(1)$ for some $\delta>0$ as soon as
$\theta\ge 
1+\varepsilon$ for some $\varepsilon>0$.
\end{thm}

The proof of this result is given in section \ref{sec:pfmasked}. The basic 
idea of the proof is approximate $\EE Y$ by a matrix of rank $r$. We then 
apply Theorem~\ref{thm:bbp} to the low-rank part of the model, and 
optimize the resulting bound over $r$ to trade off between the width of 
the phase transition and the approximation error.

Let us note that the recovery condition $\theta>1$ in Theorem 
\ref{thm:masked} is the best possible in general: for example, when 
$\Gamma$ is the complete graph, no estimator can recover a nontrivial 
fraction of the vertex labels when $\theta<1$ (this follows, e.g., from 
\cite[Theorem 6.3]{PWBM16}). A key feature of Theorem \ref{thm:masked} is 
that it enables us to achieve this recovery threshold for a large class of 
deterministic graphs. An analogous problem for Erd\H{o}s-R\'enyi graphs 
was previously considered in \cite{SLKZ15}.

Let us give two examples to illustrate the assumptions of 
Theorem \ref{thm:masked}.

\begin{example}[Good expanders]
Suppose that $\mathrm{s}_2(A)\le c\sqrt{k}$; 
this is the case, for example, when $\Gamma$ is a random $k$-regular 
graph \cite{TY19}. Under mild conditions, such graphs achieve the largest
possible spectral gap by the Alon-Boppana theorem \cite{Nil91}.
In this situation, we can choose $r\leftarrow 1$ in the
assumption of Theorem \ref{thm:masked}. The conclusion of the theorem
then follows for any $1\le\theta\ll\sqrt{k}$ and $k\gg (\log d)^4$.
\end{example}

\begin{example}[Graphs of intermediate degree]
The previous example considered expanders, that is, graphs whose
spectral gap $\lambda$ is of order $k$. However,
expansion is not necessary for the conclusion to hold when the degree $k$ 
is sufficiently large, as we will presently 
illustrate. Note first that we can trivially estimate
$$
	\sum_{i=1}^d \mathrm{s}_i^2 =
	\tr[A^2] = dk,
$$
which implies $\mathrm{s}_i \le (\frac{dk}{i})^{\frac{1}{2}}$. 
Assuming $\theta\ge 1$ for simplicity, we can estimate
$$
        \min_{1\le r\le k}
        \big\{\theta\,\mathrm{s}_{r+1}+
        \sqrt{rk}\big\}
	\le	
	\sqrt{k}
        \min_{1\le r\le k}
        \big\{r^{-\frac{1}{2}} \theta\sqrt{d}
        +r^{\frac{1}{2}}\big\}
	\lesssim
	d^{\frac{1}{4}}\sqrt{\theta k} 
$$
when $k\ge\theta\sqrt{d}$. Then the conclusion of 
Theorem 
\ref{thm:masked} holds whenever
$$
	d^{\frac{1}{4}}\sqrt{\theta k} +
	k^{\frac{5}{6}}(\log d)^{\frac{2}{3}}
	\ll
	\min\{\theta\lambda,k\}.
$$
For example, if $k\sim d^a$ for some 
$\frac{1}{2}<a\le 1$, then the conclusion of Theorem
\ref{thm:masked} holds whenever $1\le\theta\ll d^{a-\frac{1}{2}}$
and $\frac{\lambda}{k}\gg 
\max\{\theta^{-1}d^{-\frac{a}{6}}(\log 
d)^{\frac{2}{3}},\theta^{-\frac{1}{2}}d^{-\frac{1}{2}(a-\frac{1}{2})}\}$.
\end{example}

It is instructive to note that the spectral gap $\lambda$ appears in 
Theorem \ref{thm:masked} only in order to resolve the top eigenvector of 
$Y$. This is necessary: for example, if $\Gamma$ were not connected, then 
the distribution of $Y$ would be unchanged if we flip the signs of all 
vertex labels in one connected component, so that it is fundamentally 
impossible to recover $x$ up to a global sign. On the other hand, if we 
were interested only in detecting the presence of an outlier eigenvalue in 
the spectrum of $Y$, no assumption on the spectral gap would be needed in 
the proof.

\subsection{Tensor PCA}

The following may be viewed as an analogue of the classical spiked 
Wigner model for tensors of order $p$.
Fix $\lambda>0$, a signal $x\in\{-1,+1\}^n$, and i.i.d.\ standard 
Gaussian variables $(Z_S)_{S\subseteq[n]:|S|=p}$. We are given a noisy 
observation tensor $Y=\lambda\, x^{\otimes p}+Z$; more precisely, 
we observe
$$
	Y_S := \lambda X_S + Z_S
$$
for all $S\subseteq[n]$ with $|S|=p$, 
where $X_S := \prod_{i\in S}x_i$. This is the tensor PCA model,
cf.\ \cite{MR14,WEM19} and the references therein. The key problems
in this context are detecting whether a signal is present, and 
recovering the signal. In the interest of space we focus 
on detection, though recovery may be similarly investigated.

We presently describe a general family of spectral methods for the tensor 
PCA problem that was proposed in \cite{WEM19}. Let $p\ge 4$ be even,
fix an integer 
$\ell\in[\frac{p}{2},n-\frac{p}{2}]$, and define a symmetric
${n\choose\ell}\times{n\choose\ell}$ random matrix 
$M=(M_{S,T})_{S,T\subseteq[n]:|S|=|T|=\ell}$ as
$$
	M_{S,T} := \begin{cases}
	Y_{S\triangle T} & \text{if }|S\triangle T|=p,\\
	0 & \text{otherwise},
	\end{cases}
$$
where $\triangle$ denotes the symmetric difference. The presence of a 
signal is then detected by the presence of an outlier eigenvalue in the 
spectrum of $M$. It is shown in \cite{WEM19} that correct detection of the 
presence or absence of a signal with probability $1-o(1)$ is achieved
by this method when $\lambda\gg n^{-\frac{p}{4}}\sqrt{\log n}$.

Here we achieve a much more precise understanding of this method for a 
certain range of the design parameter $\ell$ of the detection algorithm.

\begin{thm}
\label{thm:kikuchi}
Fix $\frac{p}{2}\le \ell < \frac{3p}{4}$ and
$\alpha>0$, and define
$$
	k_* := {\ell \choose p/2} {n-\ell \choose p/2}.
$$
Then there is a constant $C$ that depends only on
$p,\ell,\alpha$ so that
$$
	\mathbf{P}\big[
	|\lambda_{\rm max}(k_*^{-\frac{1}{2}}M)
	-\mathrm{B}(\lambda k_*^{\frac{1}{2}})|
	>
	C n^{-1} \lambda k_*^{\frac{1}{2}} +
	C n^{\frac{4\ell-3p}{4}} +
	C n^{\frac{\ell-p}{4}}(\log n)^{\frac{3}{4}}
	\big] \le \frac{C}{n^\alpha}.
$$
In particular, for any $\varepsilon>0$, the test
$$
	f(M) := \begin{cases}
	1 &\text{if }\lambda_{\rm max}(k_*^{-\frac{1}{2}}M)>2+n^{-\frac{1}{5}},\\
	0 &\text{otherwise},
	\end{cases}
$$
has $\mathbf{P}[f(M)=1]=o(1)$ if $\lambda\le k_*^{-\frac{1}{2}}$ and
$\mathbf{P}[f(M)=1]=1-o(1)$ if $\lambda\ge 
(1+\varepsilon)k_*^{-\frac{1}{2}}$.
\end{thm}

The order $\lambda \sim k_*^{-\frac{1}{2}} \sim c(p,\ell) 
n^{-\frac{p}{4}}$ of the signal strength is believed to be the weakest 
that can be detected by computationally efficient algorithms, cf.\ 
\cite{HSS15,WEM19}. To the best of our knowledge, however, Theorem 
\ref{thm:kikuchi} is the first result to establish a sharp phase 
transition for any tensor PCA algorithm for symmetric 
tensors.\footnote{%
For the asymmetric analogue of the tensor PCA model, a sharp transition 
was established in \cite[\S 3.2]{MR14} . This case is 
considerably simpler, as the natural counterpart of $M$ with 
$\ell=\frac{p}{2}$ has
independent entries and thus its analysis reduces
to that of the classical spiked Wigner model.}

The proof of Theorem \ref{thm:kikuchi} is given in section 
\ref{sec:pftensor}. As in the previous section, the idea of the proof is 
to approximate $\EE M$ by a low rank matrix. However, in the present case 
the entries of $M$ exhibit a complicated dependence structure, which is 
nonetheless captured effortlessly by our main results.

\begin{rem}
The design parameter $\ell$ provides a tradeoff between computational cost 
and the detection threshold: the larger $\ell$, the more costly is the 
computational method (as the dimension of $M$ is of order $n^\ell$) but 
the smaller is the detection threshold $c(p,\ell)$. It is conjectured in 
\cite[Conjecture 3.6]{WEM19} that an arbitrarily small detection threshold 
$c(p,\ell)$ can be achieved by choosing $\ell$ sufficiently large. This 
regime is not captured by Theorem \ref{thm:kikuchi}, however, as its 
validity is restricted to the range $\frac{p}{2}\le \ell < \frac{3p}{4}$. 
When $\ell$ is large compared to $p$, the dependence structure of $M$ is 
so strong that it is unclear whether it could be accurately modeled by 
$M_{\rm free}$. Nonetheless, even the detection threshold achieved by 
Theorem \ref{thm:kikuchi} for $\ell=\frac{p}{2}$ is already of smaller 
order than is captured by the analysis of \cite{WEM19} for any value of 
$\ell$.
\end{rem}

\begin{rem}
The matrix $M$ used here is known as a Kikuchi matrix. Such matrices have
had a number of unexpected applications in recent years, such as to the 
study
of Moore bounds for hypergraphs. We refer to \cite{HKM23}
for more on this topic.
\end{rem}

\subsection{Spike detection in block-structured models}
\label{sec:krz}

The above applications feature various nonhomogeneous but isotropic 
models. We now study an anisotropic model proposed by Pak, Ko, 
and Krzakala \cite{KKP23}.

Let $x\in\{-1,+1\}^d$, and 
let $H$ be a $d\times d$ self-adjoint random matrix whose entries 
$(H_{ij})_{i\ge j}$ are independent with $H_{ij}\sim 
N(0,\frac{1+1_{i=j}}{d}\boldsymbol{\Delta}_{ij})$; here
$\boldsymbol{\Delta}_{ij}=\boldsymbol{\Delta}_{ji}\ge 0$ define
an arbitrary variance profile of the entries of $H$. Then
$$
	\tilde X = \frac{1}{d}\,xx^* + H
$$
is an anisotropic variant of the spiked Wigner model of section 
\ref{sec:introapp}. The question is when it is possible to detect the
presence of the spike $xx^*$, and to recover
the entries of $x$, when we can only observe $\tilde X$.
One may expect that this question can be addressed using the largest 
eigenvalue and eigenvector of $\tilde X$ as in section \ref{sec:simpleex}. 
Unfortunately, the variational principle \eqref{eq:lehner} is generally 
not analytically tractable for anisotropic models. More surprisingly, 
the detection threshold achieved by this method turns out to be  
information-theoretically suboptimal \cite{GKKZ22}.

Instead, \cite{KKP23} propose to consider the largest eigenvalue and 
eigenvector of a deterministic transformation of $\tilde X$ 
that is motivated by statistical physics:
$$
	X = \frac{1}{\boldsymbol{\Delta}}\odot \tilde X - 
	\diag\bigg(\frac{1}{d\boldsymbol{\Delta}}1_d\bigg),
$$
where $\frac{1}{\boldsymbol{\Delta}}$ denotes the elementwise inverse and 
$\odot$ denotes elementwise (Hadamard) product. This procedure can be 
implemented provided all variances $\boldsymbol{\Delta}_{ij}>0$ are 
positive and known. It is conjectured in 
\cite[Conjecture 1.6]{KKP23} that this approach achieves the optimal 
detection threshold when the variance profile $\boldsymbol{\Delta}$ has 
block structure. Here we prove this conjecture in a strong form.

In the following theorem, we denote by $X_\varnothing$ the null model
associated to $X$, that is, the model where we replace 
$x\leftarrow 0$ in the definition of $X$.

\begin{thm}
\label{thm:aniso}
Let $q\ge 1$, let $\Delta$ be a $q\times q$ self-adjoint matrix
with positive entries, and let $C_1\sqcup\cdots\sqcup C_q$
be a partition of $[d]$ into $q$ disjoint sets of size $|C_k|=:c_kd>1$.
Assume the variance profile $\boldsymbol{\Delta}$ is defined
by $\boldsymbol{\Delta}_{ij}=\Delta_{kl}$ for $i\in C_k,j\in C_l$.
Let
$$
	\mathrm{SNR}(\Delta) :=
	\lambda_{\rm max}\big(
	\diag(c)^{\frac{1}{2}}\tfrac{1}{\Delta}\diag(c)^{\frac{1}{2}}
	\big).
$$
Then there exists $\mu\le 
1-\kappa\,\big(1-\mathrm{SNR}(\Delta)^{\frac{1}{2}}\big)^2$
so that the following hold.
\begin{enumerate}[a.]
\item If $\mathrm{SNR}(\Delta)>1$, we have with probability
$1-e^{-d^{\frac{1}{2}}}$
$$
	|\lambda_{\rm max}(X)-1| \le 
	C\beta^{\frac{1}{2}}\Big(\tfrac{(\log 
		d)^{\frac{3}{4}}}{d^{\frac{1}{4}}}+
		\tfrac{q^{\frac{1}{2}}}{d^{\frac{1}{2}}}\Big),
	\quad
	|\lambda_{\rm max}(X_\varnothing)-\mu| \le
	C\beta^{\frac{1}{2}}\Big(\tfrac{(\log 
		d)^{\frac{3}{4}}}{d^{\frac{1}{4}}}+
		\tfrac{q^{\frac{1}{2}}}{d^{\frac{1}{2}}}\Big).
$$
\item If $\mathrm{SNR}(\Delta)\le 1$, 
we have with probability $1-e^{-d^{\frac{1}{2}}}$
$$
	|\lambda_{\rm max}(X)-\mu| \le 
	C\beta^{\frac{1}{2}}\Big(\tfrac{(\log 
		d)^{\frac{3}{4}}}{d^{\frac{1}{4}}}+
		\tfrac{q^{\frac{1}{2}}}{d^{\frac{1}{2}}}\Big),
	\quad
	|\lambda_{\rm max}(X_\varnothing)-\mu| \le
	C\beta^{\frac{1}{2}}\Big(\tfrac{(\log 
		d)^{\frac{3}{4}}}{d^{\frac{1}{4}}}+
		\tfrac{q^{\frac{1}{2}}}{d^{\frac{1}{2}}}\Big).
$$
\end{enumerate}
Here $\beta := \max_{i,j}\frac{1}{\Delta_{ij}}$, $\kappa:= 
\frac{\min_i b_i}{\max_i b_i}$ where
$b$ is the Perron eigenvector of $\frac{1}{\Delta}\diag(c)$.
\end{thm}

The proof of this result in section \ref{sec:aniso} is a 
straightforward consequence of the fact that $X$ is of the form
\eqref{eq:krz} with $\mathbf{B}=\frac{1}{\boldsymbol{\Delta}}$ and
$z = x$.

Theorem \ref{thm:aniso} shows that the largest eigenvalue of $X$ can 
detect the presence or absence of a signal with probability $1-o(1)$ when 
$\mathrm{SNR}(\Delta)\ge 1+\varepsilon$ for any $\varepsilon>0$, which 
coincides with the information-theoretic detection limit for this problem 
\cite{GKKZ22}. We emphasize that this is established here in a much 
stronger \emph{nonasymptotic} form than was conjectured in \cite{KKP23} 
(and proved in \cite{MKK24} concurrently with our work). In particular, 
the conclusion is valid when $q\ll d$ and $\beta,\frac{1}{\kappa}\lesssim 
1$, which means the number of blocks may be chosen to diverge rapidly as 
the dimension increases.

\begin{rem}
In \cite[Conjecture 1.6]{KKP23}, the largest eigenvalue of $X$ is 
conjectured to detach from the bulk of the spectrum if and only if 
$\mathrm{SNR}(\Delta)>1$. This formulation must be interpreted with care 
in the nonasymptotic setting, however, as it is possible when both 
$d,q\to\infty$ that even the spectrum of $X_\varnothing$ has components 
that detach from the bulk. The formulation of Theorem \ref{thm:aniso} 
avoids this pitfall.
\end{rem}

\begin{rem}
We assumed that $x\in\{-1,1\}^d$ primarily for simplicity of 
exposition. In \cite{KKP23}, it is assumed instead that the entries of $x$ 
are i.i.d.\ with $\EE[x_i^2]=1$. It is completely 
straightforward to extend Theorem \ref{thm:aniso} to this setting, see 
Remark~\ref{rem:gensignal}. However, the quantitative rates in 
Theorem \ref{thm:aniso} must then depend on what assumptions are 
made on the distribution of $x_i$. As no new insights are obtained from
such an extension, we have chosen to focus on the above concrete setting.
\end{rem}

Beside the behavior of the top eigenvalue, \cite{KKP23} 
also conjectured a corresponding phase transition for the overlap between 
$x$ and the largest eigenvector of $X$. As was explained in 
Remark \ref{rem:ohnonooverlap}, it has proved challenging to achieve 
nonasymptotic bounds for this quantity. However, in the asymptotic setting 
where all the model parameters are fixed as $d\to\infty$, the correct 
behavior may be readily established.

\begin{thm}
\label{thm:anisoeigen}
Consider the setting of Theorem \ref{thm:aniso}, where $q,\Delta,c$ 
are taken to be fixed as $d\to\infty$. Then we have
$$
	\frac{1}{d}
	|\langle x,v_{\rm max}(X)
	\rangle|^2 \to 0\text{ in probability as }d\to\infty
$$ 
if and only if $\mathrm{SNR}(\Delta)\le 1$.
\end{thm}

\subsection{Contextual stochastic block models}
\label{sec:csbm}

We now discuss an entirely different anisotropic application: a spike 
detection problem with side information.

Let $G$ be an $n\times n$ self-adjoint random matrix whose entries 
$(G_{ij})_{i\ge j}$ are independent with $G_{ij}\sim 
N(0,\frac{1+1_{i=j}}{n})$, $H$ be a $p\times n$ random matrix all 
of whose entries are i.i.d.\ with distribution $N(0,\frac{1}{p})$,
and $u\sim N(0,\frac{1}{p}\id_p)$ be an independent random vector.
We further fix $v\in\{-1,+1\}^n$ and
$\lambda,\mu\ge 0$. We now define
$$
	A = \frac{\lambda}{n}\, vv^* + G,\qquad\quad
	Y = \sqrt{\frac{\mu}{n}}\,uv^* + H.
$$
We aim to recover the signal $v$ from observation of both $A$ and $Y$.
Note that $A$ is precisely the classical spiked Wigner model, while $Y$ 
provides additional information or ``context'' about $v$. The 
availability of side information should make it possible to detect
weaker signals than is possible otherwise.

The above model was proposed in \cite[eq.\ (8)--(9)]{DMMS18} as a Gaussian 
counterpart of an analogous discrete model called the contextual 
stochastic block model. In particular, it is shown in the proof of 
\cite[Theorem 4]{DMMS18} by an indirect argument that the existence of any 
nontrivial estimator of $v$ in the Gaussian model will imply the existence 
of such an estimator in the discrete model. We will therefore focus our
attention here for simplicity on the Gaussian model.

In \cite[Theorem 6]{DMMS18}, the authors propose a detection algorithm 
that combines a spectral method with a univariate optimization problem. 
Here we investigate a simpler detection algorithm that is purely 
spectral in nature, which was suggested by O.\ Duranthon, F.~Krzakala, 
and L.~Zdeborov{\'a} (personal communication) based on ideas of
statistical physics. To define this algorithm, let
$$
	\hat X = 
	\begin{bmatrix}
	\lambda A - (\lambda^2+\frac{\mu p}{n})\id_n & Y^*\sqrt{\frac{\mu p}{n}} 
	\\[.2cm]
	Y\sqrt{\frac{\mu p}{n}} & - \mu\id_p
	\end{bmatrix},
$$
and let $\hat v\in\mathbb{R}^n$ be the restriction of $v_{\rm 
max}(\hat X)\in\mathbb{R}^{n+p}$ to the first $n$ coordinates. The 
following theorem shows that $\hat v$ has positive correlation 
with $v$ up to the information-theoretic detection limit for this problem 
(cf.\ \cite[Theorem 6]{DMMS18}).

\begin{thm}
\label{thm:context}
Let $\lambda,\mu$ be fixed as $n,p\to\infty$ with $\frac{n}{p}\to 
\gamma\in (0,\infty)$. Then 
$$
	\frac{1}{n}|\langle v,\hat v\rangle|^2 \ge \varepsilon - o(1)
	\text{ with probability }1-o(1)\text{ for some }\varepsilon>0
$$
if and only if $\lambda^2+\frac{\mu^2}{\gamma}>1$.
\end{thm}

The proof of this result is given in section \ref{sec:pfcontext}. The main 
idea of the proof is that we can identify $\hat X$ as another special 
instance of the model \eqref{eq:krz}. The model differs from that of the 
previous section in that the matrix $\mathbf{B}$ now has many vanishing 
entries; this will, however, not cause any complications in our analysis.

\begin{rem}
Our analysis could be extended using Theorem \ref{thm:univ} to achieve the 
same result directly for the discrete contextual stochastic block model, 
provided that its average degrees grow polylogarithmically in the 
dimension. The latter is necessary, as direct spectral methods 
fundamentally cannot work in the regime of constant average degrees due to 
the presence of high degree nodes (see, e.g., \cite{BBK19}). This does not 
contradict the indirect universality argument used in \cite{DMMS18}, which 
merely ensures the existence of an estimator in the discrete setting. 
\end{rem}

\subsection{Sample covariance error}

Let $X_1,\ldots,X_n$ be i.i.d.\ Gaussian random vectors in $\mathbb{R}^p$
with distribution $N(0,\Sigma)$. Then the sample covariance matrix
\begin{equation}
\label{eq:scov}
	\hat \Sigma := \frac{1}{n} \sum_{i=1}^n X_iX_i^*
\end{equation}
is a natural estimator of the covariance matrix $\Sigma$. If 
$X$ is the $p\times n$ random matrix whose columns are $X_1,\ldots,X_n$, 
we may write $\hat\Sigma = \frac{1}{n}XX^*$.

Sample covariance matrices exhibit outlier phase transitions much like in
the spiked Wigner model when the covariance matrix $\Sigma$ is defined by 
a low-rank perturbation. This is in fact the original setting studied by 
Baik et al.\ \cite{BBP05,BS06}. For sake of illustration, we will consider
here the simplest such model where
\begin{equation}
\label{eq:scovrk1}
	\Sigma = \lambda\, vv^* + \id_p,
\end{equation}
where $\lambda\ge 0$ and $\|v\|=1$. In this case, it is shown in 
\cite{BBP05,BS06} that in the asymptotic regime $n,p\to\infty$ with 
$\frac{p}{n}\to\delta$,
the largest eigenvalue of $\hat\Sigma$ is a spectral outlier if and 
only if $\lambda>\sqrt{\delta}$ (cf.\ Theorem \ref{thm:scov} below).

If we view $\hat\Sigma$ as an estimator of $\Sigma$, however, then it is 
often more relevant to understand the norm of the estimation error 
$\|\hat\Sigma-\Sigma\|$, rather than the norm of the estimator itself 
$\|\hat\Sigma\|$ as in \cite{BBP05,BS06}. It was conjectured by Han 
\cite{Han22} that the sample covariance 
error also exhibits a phase transition; however, this transition occurs at 
a larger value of the signal strength $\lambda>1+\sqrt{\delta}$. 

We will prove a nonasymptotic form of the above phase 
transitions and of yet another transition for the smallest
eigenvalue of $\hat\Sigma-\Sigma$. Here we define
\begin{align*}
	\mathrm{S}(\lambda,\delta) &:= \begin{cases}
	(1+\sqrt{\delta})^2\hspace{2cm}
	& \text{for }\lambda\le \sqrt{\delta},\\
	(1+\lambda)(1+\tfrac{\delta}{\lambda}) & \text{for }\lambda>\sqrt{\delta},
	\end{cases}
	\\
	\mathrm{H}_+(\lambda,\delta) &:= \begin{cases}
	\delta + 2\sqrt{\delta}
	& \text{for }\lambda\le 1+\sqrt{\delta},\\
	\tfrac{1+\lambda}{2\lambda}(\sqrt{\delta}+\sqrt{\delta+4\lambda})
	\sqrt{\delta} & \text{for }\lambda>1+\sqrt{\delta},
	\end{cases}
	\\
	\mathrm{H}_-(\lambda,\delta) &:= \begin{cases}
	\delta -
	2\sqrt{\delta} & \text{for }\lambda\le 1-\sqrt{\delta},\\
	\tfrac{1+\lambda}{2\lambda}(\sqrt{\delta}-\sqrt{\delta+4\lambda})
	\sqrt{\delta} & \text{for }\lambda>1-\sqrt{\delta}.
	\end{cases}
\end{align*}
These functions play the same role for $\hat\Sigma$ and $\hat\Sigma-\Sigma$
as does $\mathrm{B}(\theta)$ in section \ref{sec:isophase}.
The following theorem is illustrated in Figures \ref{fig:scov1} and 
\ref{fig:scov2}. 
\begin{figure}
\centering
\includegraphics[width=0.9\textwidth]{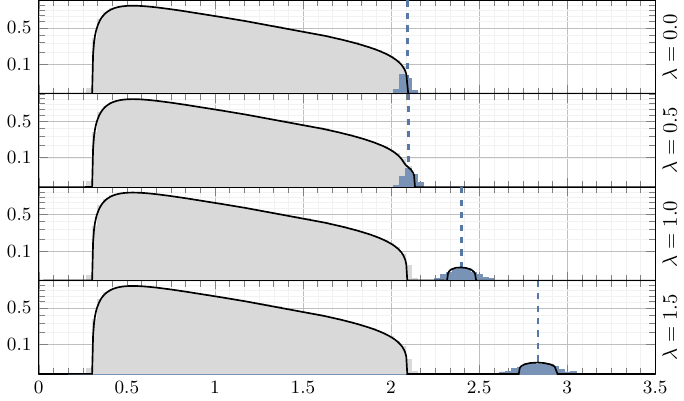}
\caption{Illustration of $200$ samples of  $\hat\Sigma$ in Theorem 
\ref{thm:scov} with
$p = 400$ and $n=2000$ (so that $\sqrt{\delta}\approx 0.45$).
The light shaded area is the empirical histogram of all 
eigenvalues normalized to have total area $1$, while the colored shaded 
area is the empirical histogram of the largest eigenvalue normalized to 
have area $\frac{1}{p}$. The solid line is the spectrum of the free model, 
and the dashed vertical line marks 
$\mathrm{S}(\lambda,\delta)$. The vertical axis follows a 
square-root scale to visualize the density of the 
outlier.\label{fig:scov1}}
\end{figure}

\begin{thm}
\label{thm:scov}
Let $\delta=\frac{p}{n}$ and $d=\max\{n,p\}$, and assume 
$n\ge (\log d)^3$. Then
\begin{align*}
&	\mathbf{P}\big[\big|\|\hat\Sigma\| - \mathrm{S}(\lambda,\delta)\big|
	>
	C(1+\lambda+\delta)\, n^{-\frac{1}{4}} (\log d)^{\frac{3}{4}}] \le
	e^{-Cn^{\frac{1}{2}}},
\\
&
	\mathbf{P}\big[\big|\lambda_{\rm max}(\hat\Sigma-\Sigma)
	- \mathrm{H}_+(\lambda,\delta)\big|
	>
	C(1+\lambda+\delta)\, n^{-\frac{1}{4}} (\log d)^{\frac{3}{4}}] \le
	e^{-Cn^{\frac{1}{2}}},
\\
&
	\mathbf{P}\big[\big|\lambda_{\rm min}(\hat\Sigma-\Sigma)
	- \mathrm{H}_-(\lambda,\delta)\big|
	>
	C(1+\lambda+\delta)\, n^{-\frac{1}{4}} (\log d)^{\frac{3}{4}}] \le
	e^{-Cn^{\frac{1}{2}}}.
\end{align*}
\end{thm}
\begin{figure}
\centering
\includegraphics[width=0.9\textwidth]{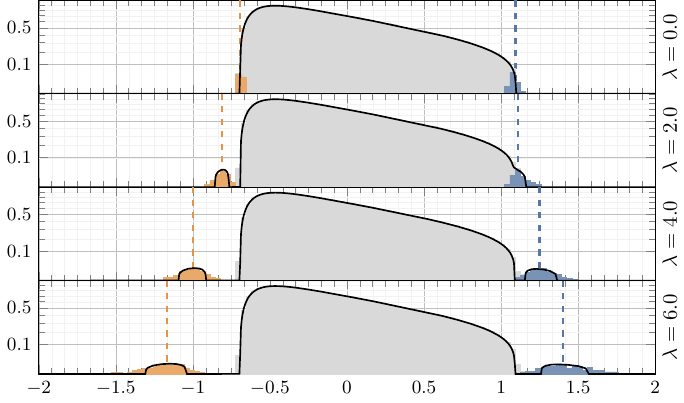}
\caption{
Illustration of $200$ samples of $\hat\Sigma-\Sigma$ in Theorem 
\ref{thm:scov} with 
the same parameters as in Figure \ref{fig:scov1}. Here the two colored 
shaded areas are the empirical histograms of the smallest and largest 
eigenvalues, and the dashed vertical lines mark the locations of
$\mathrm{H}_\pm(\lambda,\delta)$.\label{fig:scov2}}
\end{figure}

\begin{rem}
Note that as $|\mathrm{H}_-(\lambda,\delta)|\le 
\mathrm{H}_+(\lambda,\delta)$, Theorem \ref{thm:scov} also implies that
$\|\hat\Sigma-\Sigma\|=(1+o(1))\mathrm{H}_+(\lambda,\delta)$ with
probability $1-o(1)$ as conjectured in \cite{Han22}.
\end{rem}

The proof of Theorem \ref{thm:scov} is given in section \ref{sec:pfscov}. 
We will use Theorem \ref{thm:quad} to reduce the problem to the analysis 
of its deterministic model, which can be performed using a quadratic 
counterpart of the Lehner formula \cite{EvH24}.

\section{Ultracontractive bounds}
\label{sec:ultra}

The aim of this section is to prove that suitable analogues of the 
elementary fact \eqref{eq:logdmoment} for matrices with scalar entries 
hold for matrices whose entries are polynomials of semicircular variables. 
In the following, we let $s_1,\ldots,s_n$ be a free semicircular family, 
and denote by $\|Z\|_{L^p(\tau)} := \tau(|Z|^p)^{\frac{1}{p}}$ the 
noncommutative $L^p$-norm.

The following may be viewed as a direct analogue of the scalar case.

\begin{thm}
\label{thm:polymoments}
Let $P\in \M_d(\mathbb{C})\otimes\mathbb{C}\langle x_1,\ldots,x_n\rangle$ 
be any noncommutative 
polynomial of degree $k$ with matrix coefficients. Then we have 
$$
	\|P(s_1,\ldots,s_n)\| \le
	d^{\frac{3}{4q}}
	(2qk+1)^{\frac{3}{4q}} 
	\|P(s_1,\ldots,s_n)\|_{L^{4q}(\ntr\otimes\tau)}
$$
for every $q\in\mathbb{N}$.
\end{thm}

However, we will require such a property not for the norm of the matrix 
itself, but rather for the norm of its resolvent. For the resolvent of a 
self-adjoint polynomial, we can deduce an analogous result up to a small 
error.

\begin{thm}
\label{thm:polyres}
Let $P\in \M_d(\mathbb{C})\otimes\mathbb{C}\langle x_1,\ldots,x_n\rangle$ 
be any self-adjoint noncommutative 
polynomial of degree $k$ with matrix coefficients. Then we have 
\begin{multline*}
	\|(z-P(s_1,\ldots,s_n))^{-1}\| \le \\
	d^{\frac{3}{4q}}
	(2qkr+1)^{\frac{3}{4q}} 
	\bigg(
	\|(z-P(s_1,\ldots,s_n))^{-1}\|_{L^{4q}(\ntr\otimes\tau)}
	+
	\frac{12 r^{-1}\|P(s_1,\ldots,s_n)\|}{(\mathrm{Im}\,z)^2}
	\bigg)
\end{multline*}
for every $q,r\in\mathbb{N}$ and $z\in\mathbb{C}$, $\mathrm{Im}\,z>0$.
\end{thm}

Of primary interest in this paper is the noncommutative model $X_{\rm 
free}$ associated to a self-adjoint Gaussian random matrix $X$. It is 
evident from the definition \eqref{eq:xfree} that $X_{\rm free}$ is a 
self-adjoint noncommutative polynomial of degree $k=1$ with matrix 
coefficients. However, while Theorem \ref{thm:polyres} can be applied 
directly to $X_{\rm free}$ as a special case, this result is not adequate 
for our purposes because the error term in the resulting inequality is 
proportional to $\|X_{\rm free}\|$. Such a bound would give rise to sharp 
matrix concentration inequalities that become increasingly inaccurate as 
$\|\mathbf{E} X\|=\|A_0\|\to\infty$, which is unnatural in applications.

The following bound, which eliminates the dependence of the error term on 
$A_0$, is essential for the main results of this paper. Its proof is based 
on the fact that the resolvent $(z-X_{\rm free})^{-1}$ primarily captures 
the part of the spectrum of $X_{\rm free}$ that is close to 
$\mathrm{Re}\,z$, and is insensitive to eigenvalues of $A_0$ that are far 
from $\mathrm{Re}\,z$.

\begin{thm}
\label{thm:xfreeres}
Let $X_{\rm free}$ be the noncommutative model associated to a
self-adjoint random matrix $X$.
For any $z\in\mathbb{C}$ with $0<\mathrm{Im}\,z\le\sigma(X)$
and $q,r\in\mathbb{N}$, $K>4$
\begin{multline*}
	\|(z-X_{\rm free})^{-1}\| \le  \\
	d^{\frac{3}{4q}}
	(2qr+1)^{\frac{3}{4q}} 
	\bigg(
	\|(z-X_{\rm free})^{-1}\|_{L^{4q}(\ntr\otimes\tau)}
	+
	\bigg(
	\frac{K}{r}
	+
	\frac{1}{K-2}
	\bigg)
	\frac{18d\sigma(X)}{(\mathrm{Im}\,z)^2}
	\bigg).
\end{multline*}
\end{thm}

We will apply this theorem in the following form.

\begin{cor}
\label{cor:xfreeres}
Let $X_{\rm free}$ be the noncommutative model associated to a
self-adjoint random matrix $X$.
For any $z\in\mathbb{C}$ with $\mathrm{Re}\,z\in\spc(X_{\rm free})$
and $\mathrm{Im}\,z>0$, we have
$$
	\|(z-X_{\rm free})^{-1}\| \le  
	C
	\bigg(
	\|(z-X_{\rm free})^{-1}\|_{L^{4q}(\ntr\otimes\tau)}
	+
	\frac{\sigma_*(X)}{(\mathrm{Im}\,z)^2}
	\bigg)
$$
for all $q\ge \log d$, where $C$ is a universal constant.
\end{cor}

The remainder of this section is devoted to the proofs of these results.

\subsection{Proof of Theorems \ref{thm:polymoments} and \ref{thm:polyres}}

The proof of Theorem \ref{thm:polymoments} is essentially equivalent to 
that of \cite[Corollary 7.2]{BC23}. The basic tool we will use is 
ultracontractivity, due in the present setting to Bozejko and Biane 
\cite{Bia97}.

\begin{lem}
\label{lem:biane}
Let $P\in\mathbb{C}\langle x_1,\ldots,x_n\rangle$ be any noncommutative 
polynomial of degree $k$ with scalar coefficients. Then we have
$$
	\|P(s_1,\ldots,s_n)\| \le
	(k+1)^{\frac{3}{2}}\, \|P(s_1,\ldots,s_n)\|_{L^2(\tau)}.
$$
\end{lem}

\begin{proof}
We adopt the notation of \cite[\S 1]{Bia97}. Let $e_1,\ldots,e_n$ be the 
coordinate basis of $\mathbb{C}^n$. Then we can represent 
$s_i=l(e_i)+l^*(e_i)$ in terms of the annihilation and creation operators 
on the free Fock space $F_0(\mathbb{C}^n)$. It follows readily from the 
definitions that any polynomial of $s_1,\ldots,s_n$ of degree $k$ is a 
linear combination of eigenvectors of the number operator $N^0$ with 
eigenvalue at most $k$, that is,
$$
	P(s_1,\ldots,s_n) = P_0 + \cdots + P_k
$$
where $N^0P_r = rP_r$. Thus 
\begin{align*}
	\|P(s_1,\ldots,s_n)\| &\le
	\|P_0\|+\cdots+\|P_k\| \\ &\le
	(k+1)\{\|P_0\|_{L^2(\tau)} + \cdots + \|P_k\|_{L^2(\tau)}\}
	 \\ &\le
	(k+1)^{\frac{3}{2}}
	\|P(s_1,\ldots,s_n)\|_{L^2(\tau)}
\end{align*}
where the first line uses the triangle inequality, the second 
line uses \cite[Theorem 4]{Bia97}, and the last line uses Cauchy-Schwarz 
and that the 
eigenspaces of 
$N^0$ are orthogonal with respect to the inner product of $L^2(\tau)$.
\end{proof}

We now recall a standard trick to handle matrix coefficients.

\begin{lem}
\label{lem:flat}
Let $(\mathcal{A},\tau)$ be a noncommutative probability space,
let $P\in\M_d(\mathbb{C})\otimes\mathcal{A}$, and denote by
$P_{ij}\in \mathcal{A}$ its matrix elements. Then we have
$$
	\|P\| \le d \max_{ij}\|P_{ij}\|,\qquad\quad
	\max_{ij}\|P_{ij}\|_{L^2(\tau)} \le
	d^{\frac{1}{2}}\|P\|_{L^2(\ntr\otimes\tau)}.
$$
\end{lem}

\begin{proof}
For the first inequality, we may assume that $\mathcal{A}$ has been 
represented as a subalgebra of $B(H)$ for some Hilbert space $H$. Then
by Cauchy-Schwarz
$$
	\|P\| =
	\sup
	\Bigg|
	\sum_{i,j} \langle v_i,P_{ij}w_j\rangle
	\Bigg|
	\le
	\sup
	\sum_{i,j} \|v_i\| \|P_{ij}\| \|w_j\|
	\le
	d \max_{ij}\|P_{ij}\|,
$$
where the supremum is taken over $v_i,w_j\in H$ so that
$\sum_i \|v_i\|^2 = \sum_j \|w_j\|^2 = 1$.

For the second inequality, we note that
$$
	d^{-\frac{1}{2}}
	\|P_{ij}\|_{L^2(\tau)}=
	\|e_1e_1^*\otimes P_{ij}\|_{L^2(\ntr\otimes\tau)}
	=
	\|(e_1e_i^*\otimes\id)P(e_je_1^*\otimes\id)\|_{L^2(\ntr\otimes\tau)}
$$
and use that $\|e_1e_i^*\otimes\id\|=\|e_je_1^*\otimes\id\|=1$.
\end{proof}

We can now prove Theorem \ref{thm:polymoments}.

\begin{proof}[Proof of Theorem \ref{thm:polymoments}]
Fix $q\in\mathbb{N}$ and set $Q = |P(s_1,\ldots,s_n)|^{2q}$. Then the 
matrix elements $Q_{ij}$ are polynomials of $s_1,\ldots,s_n$ of degree 
$2qk$. Thus
$$
	\|Q\| \le
	d \max_{i,j} \|Q_{ij}\|
	\le
	d\, (2qk+1)^{\frac{3}{2}} \max_{ij}\|Q_{ij}\|_{L^2(\tau)}
	\le
	d^{\frac{3}{2}} (2qk+1)^{\frac{3}{2}} \|Q\|_{L^2(\ntr\otimes\tau)}
$$
by Lemmas \ref{lem:biane} and \ref{lem:flat}.
It remains to note that
$$
	\|Q\|=\|P(s_1,\ldots,s_n)\|^{2q},\qquad\quad
	\|Q\|_{L^2(\ntr\otimes\tau)} =
	\|P(s_1,\ldots,s_n)\|_{L^{4q}(\ntr\otimes\tau)}^{2q},
$$
and the conclusion follows immediately.
\end{proof}

To deduce a corresponding bound on the resolvent (in the case that the 
polynomial $P$ is self-adjoint), we apply an approximation argument.

\begin{proof}[Proof of Theorem \ref{thm:polyres}]
Fix $z\in\mathbb{C}$, $\mathrm{Im}\,z>0$, and
write $P:=P(s_1,\ldots,s_n)$.
As $x\mapsto |(z-x)^{-1}|$ is 
$(\mathrm{Im}\,z)^{-2}$-Lipschitz, Jackson's theorem 
\cite[Corollary 1.4.1]{Riv69} yields
$$
	\sup_{x\in [-\|P\|,\|P\|]}
	\big{|}|(z-x)^{-1}|-Q_r(x)\big{|} \le 
	\frac{6r^{-1}\|P\|}{(\mathrm{Im}\,z)^2}
$$
for a polynomial $Q_r$ of degree $r$. Therefore
\begin{align*}
	\|(z-P)^{-1}\| &\le
	\|Q_r(P)\| +
	\frac{6r^{-1}\|P\|}{(\mathrm{Im}\,z)^2}
	\\
	&\le
	d^{\frac{3}{4q}}
	(2qkr+1)^{\frac{3}{4q}} 
	\|Q_r(P)\|_{L^{4q}(\ntr\otimes\tau)} +
	\frac{6r^{-1}\|P\|}{(\mathrm{Im}\,z)^2}
	\\
	&\le
	d^{\frac{3}{4q}}
	(2qkr+1)^{\frac{3}{4q}} 
	\bigg(
	\|(z-P)^{-1}\|_{L^{4q}(\ntr\otimes\tau)}
	+
	\frac{12r^{-1}\|P\|}{(\mathrm{Im}\,z)^2}
	\bigg)
\end{align*}
by Theorem \ref{thm:polymoments}, where we used that $Q_r\circ P$ is a 
polynomial of degree at most $kr$ and that
$d^{\frac{3}{4q}}(2qkr+1)^{\frac{3}{4q}}\ge 1$. This completes
the proof.
\end{proof}

\subsection{Proof of Theorem \ref{thm:xfreeres} and Corollary 
\ref{cor:xfreeres}}

The difficulty in the proof of Theorem \ref{thm:xfreeres} is that we must 
capture the fact that the resolvent $(z-X_{\rm free})^{-1}$ is insensitive 
to the eigenvalues of $A_0$ that are far from $\mathrm{Re}\,z$. To 
implement this idea, we will use spectral perturbation theory to show that 
shrinking gaps in the spectrum of $A_0$ of size $\gg\sigma(X)$ will only 
result in a small perturbation of the resolvent.

We first recall an elementary fact.

\begin{lem}
\label{lem:spcov}
$\spc(X_{\rm free}) \subseteq \spc(A_0) + 2\sigma(X)[-1,1]$.
\end{lem}

\begin{proof}
We obtain $\spc(X_{\rm free}) \subseteq
\spc(A_0) + \|X_{\rm free}-A_0\otimes\id\|[-1,1]$
as in the proof of \cite[Theorem VI.3.3]{Bha97},
while $\|X_{\rm free}-A_0\otimes\id\|\le 2\sigma(X)$ follows from
\cite[p.\ 208]{Pis03}.
\end{proof}

We now implement the spectral perturbation argument for a single gap in 
$\spc(A_0)$. This bound will be iterated below in order to prove Theorem
\ref{thm:xfreeres}.

\begin{prop}
\label{prop:daviskah}
Fix $z\in\mathbb{C}$ with $0<\mathrm{Im}\,z\le\sigma(X)$, fix $K>4$, 
and fix $a,b\in\mathbb{R}$ with
$\delta := b-a-K\sigma(X)\ge 0$.
If $(a,b)\subset (\mathrm{Re}\,z,\infty)\backslash\spc(A_0)$ we have
$$
	\|(z-X_{\rm free})^{-1}-
	(z-X_{\rm free}+\delta 1_{[b,\infty)}(A_0)\otimes\id)^{-1}\|
	\le
	\frac{9\sigma(X)}{K-2}\frac{1}{(\mathrm{Im}\,z)^2},
$$
while if $(a,b)\subset 
(-\infty,\mathrm{Re}\,z)\backslash\spc(A_0)$ 
we have
$$
	\|(z-X_{\rm free})^{-1}-
	(z-X_{\rm free}-\delta 1_{(-\infty,a]}(A_0)\otimes\id)^{-1}\|
	\le
	\frac{9\sigma(X)}{K-2}\frac{1}{(\mathrm{Im}\,z)^2}.
$$
\end{prop}

\begin{proof}
We will only consider the case $(a,b)\subset 
(\mathrm{Re}\,z,\infty)\backslash\spc(A_0)$, as the complementary case 
follows in a completely analogous fashion. Throughout the proof, we assume 
without loss of generality that $X_{\rm free}$ is represented concretely 
as an operator on a Hilbert space, so that we may work with its spectral 
projections $1_I(X_{\rm free})$.

Define $X_{\rm free}^{(r)} := X_{\rm free}-r
1_{[b,\infty)}(A_0\otimes\id)$ for $r\in [0,\delta]$. Then
\begin{align*}
	\bigg\|
	\frac{d}{dr} (z-X_{\rm free}^{(r)})^{-1}\bigg\| &=
	\big\|
	(z-X_{\rm free}^{(r)})^{-1} 
	1_{[b,\infty)}(A_0\otimes\id)
	(z-X_{\rm free}^{(r)})^{-1}
	\big\|
\\	&\le
	\big\|
	(z-X_{\rm free}^{(r)})^{-1} 
	1_{[b,\infty)}(A_0\otimes\id)\big\|
	\,
	\big\|
	1_{[b,\infty)}(A_0\otimes\id)
	(z-X_{\rm free}^{(r)})^{-1}
	\big\|.
\end{align*}
Define $A_0^{(r)} := A_0 -r 1_{[b,\infty)}(A_0)$, that is, $A_0^{(r)}$ 
is obtained from $A_0$ by subtracting $r$ from all eigenvalues of $A_0$ that 
are greater than $b$, while leaving all other eigenvalues unchanged.
As $r\le\delta<b-a$ and $A_0$ has no eigenvalues in $(a,b)$, this implies
$$
	1_{[b,\infty)}(A_0\otimes\id) =
	1_{[b-r,\infty)}(A_0^{(r)}\otimes\id).
$$
On the other hand, as 
$\spc(X_{\rm free}^{(r)})\subseteq
\spc(A_0^{(r)}) + 2\sigma(X)[-1,1]$ by Lemma \ref{lem:spcov}, we have
$$
	(a+2\sigma(X),b-r-2\sigma(X))\cap\spc(X_{\rm free}^{(r)})=\varnothing,
$$
where we note that $a+2\sigma(X)<b-r-2\sigma(X)$ as $r\le\delta$ and 
$K>4$. Thus
\begin{align*}
	&\big\|1_{[b,\infty)}(A_0\otimes\id)
        (z-X_{\rm free}^{(r)})^{-1}
        \big\|
	\\ &\quad\le
	\big\|1_{[b-r,\infty)}(A_0^{(r)}\otimes\id)
	1_{(-\infty,a+2\sigma(X)]}(X_{\rm free}^{(r)})
        (z-X_{\rm free}^{(r)})^{-1}
        \big\| 
	\\ &\quad\qquad+
	\big\|1_{[b-r,\infty)}(A_0^{(r)}\otimes\id)
	1_{[b-r-2\sigma(X),\infty)}(X_{\rm free}^{(r)})
        (z-X_{\rm free}^{(r)})^{-1}
        \big\|
\\
	&\quad\le
	\frac{
	\big\|1_{[b-r,\infty)}(A_0^{(r)}\otimes\id)
	1_{(-\infty,a+2\sigma(X)]}(X_{\rm free}^{(r)})
        \big\|}{\mathrm{Im}\,z}
	+
	\frac{1}{b-a-r-2\sigma(X)},
\end{align*}
using $\frac{1}{|z-x|} \le \frac{1}{b-a-r-2\sigma(X)}$
for $x\ge b-r-2\sigma(X)$ (as $\mathrm{Re}\,z\le a\le 
b-r-2\sigma(X)$). The identical bound clearly holds for
$\|(z-X_{\rm free}^{(r)})^{-1}1_{[b,\infty)}(A_0\otimes\id)\|$ as well.

We now apply the Davis-Kahan theorem \cite[Theorem VII.3.1]{Bha97} to 
estimate
\begin{align*}
	\big\|1_{[b-r,\infty)}(A_0^{(r)}\otimes\id)
        1_{(-\infty,a+2\sigma(X)]}(X_{\rm free}^{(r)})
        \big\|
	&\le
	\frac{\|X_{\rm free}^{(r)}-A_0^{(r)}\otimes\id\|}{b-a-r-2\sigma(X)}
	\\ &\le
	\frac{2\sigma(X)}{b-a-r-2\sigma(X)},
\end{align*}
where we used the free Khintchine inequality \cite[p.\ 208]{Pis03} in the 
second inequality. Putting together the above estimates and using
the assumption
$1\le \frac{\sigma(X)}{\mathrm{Im}\,z}$, we get
$$
	\bigg\|
	\frac{d}{dr} (z-X_{\rm free}^{(r)})^{-1}\bigg\| \le
	\frac{9\sigma(X)^2}{(b-a-r-2\sigma(X))^2}
	\frac{1}{(\mathrm{Im}\,z)^2}.
$$
The fundamental theorem of calculus yields
$$
	\big\|(z-X_{\rm free})^{-1}
	- (z-X_{\rm free}^{(\delta)})^{-1}\big\|
	\le
	\int_0^\delta
	\frac{9\sigma(X)^2}{(b-a-r-2\sigma(X))^2}
	\frac{1}{(\mathrm{Im}\,z)^2}\,dr,
$$
and the proof is readily completed.
\end{proof}

We can now conclude the proof of Theorem \ref{thm:xfreeres}.

\begin{proof}[Proof of Theorem \ref{thm:xfreeres}]
As $(z-X_{\rm free})^{-1}$ is unchanged if we replace
$z\leftarrow i\,\mathrm{Im}\,z$ and 
$A_0\leftarrow A_0-(\mathrm{Re}\,z)\id$, we may assume without loss
of generality that $\mathrm{Re}\,z=0$. Now note that the
connected components of $\mathbb{R}\backslash(\spc(A_0)\cup\{0\})$ 
contain at most $d$ bounded intervals.
If any such interval has length exceeding $K\sigma(X)$, 
we modify $A_0$ using Proposition \ref{prop:daviskah} to 
shrink the size of that interval to $K\sigma(X)$ while incurring an
error $\frac{9\sigma(X)}{K-2}\frac{1}{(\mathrm{Im}\,z)^2}$.
Repeating this procedure
for each such interval yields a new matrix $A_0'$ such that
$\|A_0'\|\le Kd\sigma(X)$ and
$X_{\rm free}' := A_0'\otimes\id + \sum_{i=1}^n A_i\otimes s_i$ satisfies
$$
	\|(z-X_{\rm free})^{-1} -
	(z-X_{\rm free}')^{-1}\| \le
	\frac{9d\sigma(X)}{K-2}\frac{1}{(\mathrm{Im}\,z)^2}.
$$
As $\|X_{\rm free}'\| \le (Kd+2)\sigma(X)$, we further obtain 
$$
	\|(z-X_{\rm free}')^{-1}\| \le 
	d^{\frac{3}{4q}}
	(2qr+1)^{\frac{3}{4q}} 
	\bigg(
	\|(z-X_{\rm free}')^{-1}\|_{L^{4q}(\ntr\otimes\tau)}
	+
	\frac{12 r^{-1}(Kd+2)\sigma(X)}{(\mathrm{Im}\,z)^2}
	\bigg)
$$
by Theorem \ref{thm:polyres}. Combining the above bounds and using
$d^{\frac{3}{4q}}(2qr+1)^{\frac{3}{4q}}\ge 1$
and $Kd+2\le \frac{3Kd}{2}$ readily yields the conclusion.
\end{proof}

It remains to prove Corollary \ref{cor:xfreeres}.

\begin{proof}[Proof of Corollary \ref{cor:xfreeres}]
Assume first that $\mathrm{Im}\,z\le\sigma(X)$. As
$$
	\sigma(X)^2 = \sup_{\|v\|=1} 
	\sum_{i=1}^n 
	\sum_{k=1}^d 
	\langle e_k,A_iv\rangle^2 \le
	d \sup_{\|v\|=\|w\|=1} \sum_{i=1}^n\langle w,A_iv\rangle^2
	= d\,\sigma_*(X)^2,
$$
the conclusion follows from
Theorem \ref{thm:xfreeres} with
$K=2d^{\frac{3}{2}}+2$, $r = \lceil 2d^{\frac{3}{2}}K\rceil 
\le 8d^3$.

We now consider the case that $\mathrm{Im}\,z>\sigma(X)$
and $\mathrm{Re}\,z\in \spc(X_{\rm free})$. Then
\begin{align*}
	\|(z-X_{\rm free})^{-1}\|_{L^{4q}(\ntr\otimes\tau)}^{4q} &=
	(\ntr\otimes\tau)[(|z-X_{\rm free}|^2)^{-2q}]
	\\ &\ge
	\ntr
	[((\mathrm{id}\otimes\tau)|z-X_{\rm free}|^2)^{-2q}] \\
	&=
	\ntr
	[(
	(\mathrm{Im}\,z)^2 +
	\mathbf{E}[(X-\mathbf{E}X)^2] +
	(\mathrm{Re}\,z - A_0)^2 
	)^{-2q}]  \\
	&\ge
	\ntr
	[(
	(\mathrm{Im}\,z)^2 + \sigma(X)^2 +
	(\mathrm{Re}\,z - A_0)^2
	)^{-2q}], 
\end{align*}
where we used Jensen's inequality in $C^*$-algebras \cite{Pet87} in the 
first inequality, and that
$\mathbf{E}[(X-\mathbf{E}X)^2]\le \sigma(X)^2\id$ and trace monotonicity 
\cite[\S 2.2]{Car10} in the second inequality.
Now note that Lemma 
\ref{lem:spcov} ensures there is an eigenvalue of $A_0$ within 
distance $2\sigma(X)$ of $\mathrm{Re}\,z\in \spc(X_{\rm free})$.
We can therefore estimate
$$
	\|(z-X_{\rm free})^{-1}\|_{L^{4q}(\ntr\otimes\tau)}^{4q} \ge
	\frac{1}{d} \frac{1}{((\mathrm{Im}\,z)^2 + 5\sigma(X)^2
        )^{2q}} \ge
	\frac{1}{d} \frac{1}{(6(\mathrm{Im}\,z)^2
        )^{2q}}.
$$
Thus we obtain
$$
	\|(z-X_{\rm free})^{-1}\| \le
	\frac{1}{\mathrm{Im}\,z} \le
	d^{\frac{1}{4q}}\sqrt{6}\,
	\|(z-X_{\rm free})^{-1}\|_{L^{4q}(\ntr\otimes\tau)},
$$
and the conclusion follows readily.
\end{proof}

\section{Sharp matrix concentration inequalities}
\label{sec:pfmain}

\subsection{Proof of Theorem \ref{thm:main}}

As the upper bound on $\spc(X)$ was already proved in \cite{BBV23}, we 
only need to prove the corresponding lower bound. To this end, we will 
follow the approach of \cite[\S 6]{BBV23} with the crucial input of 
Corollary \ref{cor:xfreeres}. 

The basis for the proof is the following.

\begin{lem}
\label{lem:tailsingle}
Fix $z\in\mathbb{C}$ with $\mathrm{Re}\,z\in\spc(X_{\rm free})$ and
$\mathrm{Im}\,z>0$. Then
$$
	\mathbf{P}\bigg[
	c\|(z-X_{\rm free})^{-1}\| 
	\ge
	\|(z-X)^{-1}\| 
	+ \frac{\tilde v(X)^4}{(\mathrm{Im}\,z)^5}(\log d)^3
	+ \frac{\sigma_*(X)}{(\mathrm{Im}\,z)^2}(\sqrt{\log d}+t)
	\bigg]
	\le e^{-t^2}
$$
for all $t\ge 0$, where $c$ is a universal constant.
\end{lem}

\begin{proof}
We begin by noting that \cite[Corollary 4.14 and (6.2)]{BBV23} implies
\begin{equation}
\label{eq:resconctail}
	\mathbf{P}\bigg[
	\pm (\|(z-X)^{-1}\| - \mathbf{E}\|(z-X)^{-1}\|) \ge
	\frac{\sigma_*(X)}{(\mathrm{Im}\,z)^2}\,t
	\bigg]
	\le e^{-t^2/2}
\end{equation}
for $t\ge 0$. This further implies
\begin{equation}
\label{eq:resconcmom}
	\mathbf{E}[\|(z-X)^{-1}\|^{2p}]^{\frac{1}{2p}} \le
	\mathbf{E}\|(z-X)^{-1}\| +
	\frac{\sigma_*(X)}{(\mathrm{Im}\,z)^2}\,2\sqrt{p}
\end{equation}
for all $p\in\mathbb{N}$ by \cite[Theorem 2.1]{BLM13}

On the other hand, \cite[Theorem 6.1]{BBV23} yields for $q\in\mathbb{N}$
$$
	\|(z-X_{\rm free})^{-1}\|_{L^{4q}(\ntr\otimes\tau)}
	\le
	\mathbf{E}[\ntr |z-X|^{-4q}]^{\frac{1}{4q}} +
	\frac{8}{3}
	\frac{\tilde v(X)^4}{(\mathrm{Im}\,z)^5}(q+1)^3.
$$
Applying \eqref{eq:resconcmom} and Corollary \ref{cor:xfreeres}
with $q=\lceil\log d\rceil$ yields
$$
	c\|(z-X_{\rm free})^{-1}\| \le  
	\mathbf{E}\|(z-X)^{-1}\| +
	\frac{\tilde v(X)^4}{(\mathrm{Im}\,z)^5}(\log d)^3
	+ \frac{\sigma_*(X)}{(\mathrm{Im}\,z)^2}\sqrt{\log d}
$$
for a universal constant $c$,
where we used that $\ntr |z-X|^{-4q} \le \|(z-X)^{-1}\|^{4q}$.
The conclusion follows by applying \eqref{eq:resconctail}.
\end{proof}

We now deduce a uniform analogue of the previous lemma.

\begin{lem}
\label{lem:tailunif}
Fix $\varepsilon>0$. Then
\begin{multline*}
	\mathbf{P}\bigg[
	c\|(z-X_{\rm free})^{-1}\| 
	\le
	\|(z-X)^{-1}\|  
	+ \frac{\tilde v(X)^4}{\varepsilon^5}(\log d)^3
	+ \frac{\sigma_*(X)}{\varepsilon^2}(\sqrt{\log d}+t)
	\\
	\text{ for all }z\in \spc(X_{\rm free})+i\varepsilon
	\bigg]
	\ge 1-e^{-t^2}
\end{multline*}
for all $t\ge 0$, where $c$ is a universal constant.
\end{lem}

\begin{proof}
Using $\sigma(X)\le\sqrt{d}\,\sigma_*(X)$ as in 
the proof of Corollary \ref{cor:xfreeres}, Lemma \ref{lem:spcov} yields
$$
	\spc(X_{\rm free}) \subseteq 
	\spc(A_0) + 2\sqrt{d}\,\sigma_*(X)[-1,1].
$$
As $|\spc(A_0)|\le d$, it follows that
$\spc(X_{\rm free})$ can be covered by at most $d$ intervals of length
$4\sqrt{d}\,\sigma_*(X)$. We can therefore find 
$N\subset\spc(X_{\rm free})$
of cardinality $|N|\le 4d^{\frac{3}{2}}$ so that each point in
$\spc(X_{\rm free})$ is within distance $\sigma_*(X)$ of a point in $N$.

Now note that for any $\lambda,\lambda'\in\mathbb{R}$ and $\varepsilon>0$
$$
	\big| \|(\lambda+i\varepsilon-X)^{-1}\| -
	\|(\lambda'+i\varepsilon-X)^{-1}\|\big| \le
	\frac{|\lambda-\lambda'|}{\varepsilon^2},
$$
and analogously when $X$ is replaced by $X_{\rm free}$. We can therefore 
estimate
\begin{align*}
	&\mathbf{P}\bigg[
	c'\|(z-X_{\rm free})^{-1}\| 
	\ge
	\|(z-X)^{-1}\|  
	+ \frac{\tilde v(X)^4}{\varepsilon^5}(\log d)^3
	+ \frac{\sigma_*(X)}{\varepsilon^2}(\sqrt{\log d}+t)
\\
	&\qquad\qquad\qquad\qquad\qquad\qquad\qquad\qquad
	\text{ for some }z\in \spc(X_{\rm free})+i\varepsilon
	\bigg] \le
\\
	&\mathbf{P}\bigg[
	c\|(z-X_{\rm free})^{-1}\| 
	\ge
	\|(z-X)^{-1}\|  
	+ \frac{\tilde v(X)^4}{\varepsilon^5}(\log d)^3
	+ \frac{\sigma_*(X)}{\varepsilon^2}(\sqrt{\log d}+t)
\\
	&\qquad\qquad\qquad\qquad\qquad\qquad\qquad\qquad
	\text{ for some }z\in N+i\varepsilon
	\bigg] \le 4d^{\frac{3}{2}}e^{-t^2}
\end{align*}
for any $t\ge 0$
using Lemma \ref{lem:tailsingle} and the union bound, where $c,c'$ 
are universal constants. The conclusion follows by replacing
$t\leftarrow t+2\sqrt{\log d}$ and noting that
$4d^{\frac{3}{2}} e^{-(t+2\sqrt{\log d})^2} \le e^{-t^2}$ for all
$t\ge 0$ (recalling the standing assumption $d\ge 2$).
\end{proof}

We can now conclude the proof of Theorem \ref{thm:main}.

\begin{proof}[Proof of Theorem \ref{thm:main}]
It was shown in \cite[Theorem 2.1]{BBV23} that
$$
	\mathbf{P}\big[
	\spc(X)\subseteq \spc(X_{\rm free}) +
	C\{\tilde v(X) (\log d)^{\frac{3}{4}}
	+ \sigma_*(X)t \}[-1,1]
	\big] \ge 1-e^{-t^2}.
$$
On the other hand, combining Lemma \ref{lem:tailunif} with
\cite[Lemma 6.4]{BBV23} yields
$$
	\mathbf{P}\big[
	\spc(X_{\rm free})\subseteq \spc(X) +
	C\{\tilde v(X) (\log d)^{\frac{3}{4}}
	+ \sigma_*(X)t \}[-1,1]
	\big] \ge 1-e^{-t^2},
$$
where we used that $\sigma_*(X)\sqrt{\log d}\le
\tilde v(X) (\log d)^{\frac{3}{4}}$. The union bound yields
$$
	\mathbf{P}\big[
	\mathrm{d_H}(\spc(X),\spc(X_{\rm free}))
	> C\{\tilde v(X) (\log d)^{\frac{3}{4}}
	+ \sigma_*(X)t \}
	\big] \le 2e^{-t^2}.
$$
We conclude by replacing $t\leftarrow t+\sqrt{\log 2}$
and using again that $\sigma_*(X)\le\tilde v(X)$.
\end{proof}

\subsection{Proof of Corollary \ref{cor:norm}}

The proof is entirely straightforward.

\begin{proof}[Proof of Corollary \ref{cor:norm}]
Assume first that $X$ is self-adjoint. Then the tail bound
follows immediately from Theorem \ref{thm:main} as
$|\|X\|-\|X_{\rm free}\| | \le \mathrm{d_H}(\spc(X),\spc(X_{\rm free}))$.
To deduce the bound in expectation, we estimate
\begin{multline*}
	|\EE\|X\| - \|X_{\rm free}\| | 
	\le 
	\int_0^\infty \mathbf{P}\big[|\|X\|-\|X_{\rm free}\|| > x\big]\,dx 
\\
	\le
	C\tilde v(X)(\log d)^{\frac{3}{4}} +
	\int_0^\infty \mathbf{P}\big[|\|X\|-\|X_{\rm free}\|| > 
	C\tilde v(X)(\log d)^{\frac{3}{4}} + x\big] \,dx
\end{multline*}
and use that $\sigma_*(X)\le \tilde v(X)$. The corresponding results for
$\lambda_{\rm max}(X),\lambda_{\rm max}(X_{\rm free})$ and $\lambda_{\rm 
min}(X),\lambda_{\rm min}(X_{\rm free})$ follow by an identical 
argument. Finally, the norm bounds extend directly to the non-self-adjoint
case by \cite[Remark 2.6]{BBV23}.
\end{proof}

\subsection{Proof of Theorem \ref{thm:quad}}

The basis for the proof is the following variant of the linearization
argument of \cite[Lemma 3.13]{BBV23}.

\begin{lem}
\label{lem:lin}
Fix $\varepsilon>0$, and define $B_\varepsilon :=
B + (\|B\|+4\varepsilon^2)\id$ and
$$
        \breve X_\varepsilon = \begin{bmatrix}
        0 & X & B^{\frac{1}{2}}_\varepsilon \\
        X^* & 0 & 0 \\
        B^{\frac{1}{2}}_\varepsilon & 0 & 0 
        \end{bmatrix},
        \qquad\quad
        \breve X_{\rm free,\varepsilon} = \begin{bmatrix}
        0 & X_{\rm free} & B^{\frac{1}{2}}_\varepsilon\otimes\id \\
        X_{\rm free}^* & 0 & 0 \\
        B^{\frac{1}{2}}_\varepsilon\otimes\id & 0 & 0
        \end{bmatrix}.
$$
Then 
$$
	\mathrm{d_H}(\spc(\breve X_\varepsilon),\spc(\breve X_{\rm 
	free,\varepsilon}))\le\varepsilon
$$
implies
$$
	\mathrm{d_H}(\spc(XX^*+B),\spc(X_{\rm free}X_{\rm 
	free}^*+B\otimes\id) \le
	4 \{\|X_{\rm free}\| + \|B\|^{\frac{1}{2}}\} \varepsilon
	+ 5\varepsilon^2.
$$
\end{lem}

\begin{proof}
By \cite[Remark 2.6]{BBV23}, we have
$$
	\spc(\breve X_\varepsilon)\cup\{0\} =
	\spc((XX^*+B_\varepsilon)^{\frac{1}{2}}) \cup
	{-\spc((XX^*+B_\varepsilon)^{\frac{1}{2}})}\cup \{0\},
$$
and analogously for $\breve X_{\rm free,\varepsilon}$.

Consider first any $\lambda\in\spc(XX^*+B_\varepsilon)$.
Then $\lambda^{\frac{1}{2}}\ge 2\varepsilon$ by the definition of
$B_\varepsilon$, and thus
$\lambda^{\frac{1}{2}}\in \spc(\breve X_\varepsilon)$.
By assumption,
there exists $\mu\in \spc(\breve X_{{\rm free},\varepsilon})$ so that
$|\lambda^{\frac{1}{2}}-\mu|\le\varepsilon$. This implies
$\mu\ge \varepsilon$, so it must be that
$\mu\in \spc((X_{\rm free}X_{\rm 
free}^*+B_\varepsilon\otimes\id)^{\frac{1}{2}})$. Moreover,
$$
	|\lambda-\mu^2| =
	(\lambda^{\frac{1}{2}}+\mu)
	|\lambda^{\frac{1}{2}}-\mu|
	\le
	2\mu\varepsilon + \varepsilon^2.
$$
As $\mu^2\in \spc(X_{\rm free}X_{\rm free}^*+B_\varepsilon\otimes\id)$,
we have shown that
$$
	\spc(XX^*+B_\varepsilon) \subseteq
	\spc(X_{\rm free}X_{\rm free}^*+B_\varepsilon\otimes\id) +
	(2\|X_{\rm free}X_{\rm 
	free}^*+B_\varepsilon\otimes\id\|^{\frac{1}{2}}
	\varepsilon + \varepsilon^2)[-1,1].
$$
Reversing the roles of
$X,X_{\rm free}$ yields
$$
	\mathrm{d_H}(
	\spc(XX^*+B_\varepsilon),
	\spc(X_{\rm free}X_{\rm free}^*+B_\varepsilon\otimes\id))
	\le 
	2\|X_{\rm free}X_{\rm 
	free}^*+B_\varepsilon\otimes\id\|^{\frac{1}{2}}
	\varepsilon + \varepsilon^2
$$
by the identical argument.

To conclude the proof, note first that
$$
	\mathrm{d_H}(
	\spc(XX^*+B_\varepsilon),
	\spc(X_{\rm free}X_{\rm free}^*+B_\varepsilon\otimes\id)) =
	\mathrm{d_H}(
	\spc(XX^*+B),
	\spc(X_{\rm free}X_{\rm free}^*+B\otimes\id))
$$
as Hausdorff distance is translation-invariant $\mathrm{d_H}(I+t,J+t)=
\mathrm{d_H}(I,J)$. On the other hand, we can estimate
$$
	\|X_{\rm free}X_{\rm
        free}^*+B_\varepsilon\otimes\id\|
	\le
	\|X_{\rm free}\|^2 + 2\|B\| + 4\varepsilon^2,
$$
and the proof is readily completed.
\end{proof}

We can now complete the proof of Theorem \ref{thm:quad}.

\begin{proof}[Proof of Theorem \ref{thm:quad}]
By Lemma \ref{lem:lin}, we can estimate
\begin{multline*}
	\mathbf{P}\big[
	\mathrm{d_H}(\spc(XX^*+B),\spc(X_{\rm free}X_{\rm 
	free}^*+B\otimes\id) >
	4 \{\|X_{\rm free}\| + \|B\|^{\frac{1}{2}}\} \varepsilon
	+ 5\varepsilon^2\big] \\
	\le
	\mathbf{P}\big[
	\mathrm{d_H}(\spc(\breve X_\varepsilon),\spc(\breve X_{\rm
        free,\varepsilon}))>\varepsilon	
	\big].
\end{multline*}
We now recall that $\sigma_*(\breve
X_\varepsilon)=\sigma_*(X)$ and $\tilde v(\breve
X_\varepsilon)\le 2^{\frac{1}{4}}\tilde v(X)$ by \cite[Remark 2.6]{BBV23}.
Thus $t\ge \tilde v(X)(\log d)^{\frac{3}{4}}$ implies
$2t\ge 2^{-\frac{1}{4}}\tilde v(\breve X_\varepsilon)(\log 
d)^{\frac{3}{4}}+\sigma_*(\breve X_\varepsilon)\frac{t}{\sigma_*(X)}$. 
The conclusion 
follows from Theorem \ref{thm:main} by choosing $\varepsilon=C't$ for
a universal constant $C'$.
\end{proof}

\section{Phase transitions: isotropic case}
\label{sec:pfiso}

\subsection{Proof of Theorem \ref{thm:bbp}}

Theorem \ref{thm:bbp} is based on the Lehner formula \eqref{eq:lehner}. At 
its core, the reason that this variational principle exhibits phase 
transitions in the presence of low-rank structure is contained in the 
following simple observation: the last term in the Lehner formula
is small when $M$ has low rank.

\begin{lem}
\label{lem:Srank}
Let $M\in\mathrm{M}_d(\mathbb{C})_{\rm sa}$ have rank $r$. Then
$$
	\|\EE[(X-\EE X)M(X-\EE X)]\| \le 
	\sigma_*(X)^2r\,\|M\|.
$$
\end{lem}

\begin{proof}
Writing $M=\sum_{i=1}^r \lambda_i v_iv_i^*$ with 
$|\lambda_i|\le\|M\|$ and $\|v_i\|=1$, we obtain
$$
	\|\EE[(X-\EE X)M(X-\EE X)]\| =
	\sup_{\|w\|=1}
	\bigg|\sum_{i=1}^r \lambda_i\,\EE[|\langle v_i,(X-\EE X)w\rangle^2]
	\bigg|
	\le \sigma_*(X)^2r\|M\|
$$
by the triangle inequality and the definition of $\sigma_*(X)$.
\end{proof}

We first prove the upper bound in Theorem \ref{thm:bbp}.

\begin{lem}
\label{lem:bbpupper}
Let $X$ be any $d\times d$ self-adjoint random matrix with
$\EE[(X-\EE X)^2]=\id$ and such that $\EE X$ has rank $r$. Then we have
$$
        \lambda_{\rm max}(X_{\rm free}) \le
        \mathrm{B}(\lambda_{\rm max}(\EE X)) +
	2\sigma_*(X)\sqrt{r}.
$$
\end{lem}

\begin{proof}
Denote by $P$ the projection onto the range of $\EE X$. Then we can upper
bound $\lambda_{\rm max}(X_{\rm free})$ by restricting the infimum in
\eqref{eq:lehner} only to matrices of the form $M=sP+t(\id -P)$ for 
$s,t>0$, and using that for such $M$
$$
	\mathbf{E}[(X-\mathbf{E}X)M(X-\mathbf{E}X)]
	\le t + \sigma_*(X)^2rs
$$
by the isotropic assumption $\mathbf{E}[(X-\mathbf{E}X)^2]=\id$ and
Lemma \ref{lem:Srank}. This yields
\begin{align*}
	\lambda_{\rm max}(X_{\rm free}) 
	&\le
	\inf_{s,t>0} \lambda_{\rm max}\big(
	\mathbf{E} X + 
	s^{-1}P + t^{-1}(\id-P) +
	t + 
	\sigma_*(X)^2rs
	\big)
\\
	&\le
	\inf_{t>0}
	\max\{\lambda_{\rm max}(\mathbf{E} X)+t,t^{-1}+t\}
	+ 2\sigma_*(X)\sqrt{r},
\end{align*}
where we used $\EE X \le \lambda_{\rm max}(\EE X)P$ and
$s^{-1}P\le s^{-1}$ in the second inequality. It remains to note that
$\inf_{t>0} \max\{\theta+t,t^{-1}+t\}=\mathrm{B}(\theta)$.
\end{proof}

We now turn to the lower bound.

\begin{lem}
\label{lem:bbplower}
Let $X$ be any $d\times d$ self-adjoint random matrix with
$\EE[(X-\EE X)^2]=\id$ and such that $\EE X$ has rank $r$ with
$\sigma_*(X)\sqrt{r}\le 1$. Then we have
$$
        \lambda_{\rm max}(X_{\rm free}) \ge
        \mathrm{B}(\lambda_{\rm max}(\EE X)) -
	2\sigma_*(X)\sqrt{r}.
$$
\end{lem}

\begin{proof}
Let $r'\le r$ be the number of (strictly) negative eigenvalues of $\EE X$, 
and assume without loss of generality that the associated eigenvectors
are $e_{d-r'+1},\ldots,e_d$. Let $Q:\mathbb{C}^d\to\mathbb{C}^{d-r'}$ be 
the coordinate projection on the first $d-r'$ coordinate directions.
Then the Lehner formula \eqref{eq:lehner} yields 
\begin{align*}
	\lambda_{\rm max}(X_{\rm free}) &\ge
	\lambda_{\rm max}((Q\otimes\id)X_{\rm free}(Q^*\otimes\id))
	\\ &= 
	\inf_{M>0}
	\lambda_{\rm max}\big(
	Q\,\EE X\,Q^* + M^{-1} +
	Q\mathbf{E}[(X-\mathbf{E}X)Q^*MQ(X-\mathbf{E}X)]Q^*\big),
\end{align*}
where the infimum is taken over $(d-r')$-dimensional matrices $M$.

To proceed, note that $\id-Q^*Q$ is a projection of rank $r'\le r$. Thus
$$
	Q\mathbf{E}[(X-\mathbf{E}X)Q^*Q(X-\mathbf{E}X)]Q^*
	\ge 
	1 - \sigma_*(X)^2r =: \beta
$$
where we used the isotropic assumption $\EE[(X-\EE X)^2]=\id$, that
$QQ^*=\id$, and Lemma \ref{lem:Srank}. Moreover, note that
$\beta\ge 0$ by assumption. We can therefore bound
\begin{align*}
	&\lambda_{\rm max}\big(
	Q\,\EE X\,Q^* + M^{-1} +
	Q\mathbf{E}[(X-\mathbf{E}X)Q^*MQ(X-\mathbf{E}X)]Q^*\big)
\\ &\ge
	\lambda_{\rm max}\big(Q\,\EE X\,Q^* + M^{-1}\big)
	+
	\beta\lambda_{\rm min}(M)
\\ &\ge
	\max\{\lambda_{\rm max}(\EE X),(\lambda_{\rm min}(M))^{-1}\}
	+
	\beta\lambda_{\rm min}(M),
\end{align*}
where we used $M^{-1}\ge 0$, $\lambda_{\rm max}(Q\,\EE X\,Q^*)=
\lambda_{\rm max}(\EE X)$ and $Q\,\EE X\,Q^*\ge 0$, respectively, to
obtain the two terms in the maximum on the last line.
Thus
$$
	\lambda_{\rm max}(X_{\rm free}) \ge
	\inf_{t>0} \max\{\lambda_{\rm max}(\EE X) + \beta t,
	t^{-1}+\beta t\}
	=
	\beta^{\frac{1}{2}}\mathrm{B}(\beta^{-\frac{1}{2}}
	\lambda_{\rm max}(\EE X)).
$$
It remains to show that 
$\beta^{\frac{1}{2}}\mathrm{B}(\beta^{-\frac{1}{2}}\theta) \ge
\mathrm{B}(\theta)-2\sigma_*(X)\sqrt{r}$.

To this end, note first that $\beta^{\frac{1}{2}}\ge
1-\sigma_*(X)\sqrt{r}$ as
$\sqrt{1-x^2}\ge 1-x$ for $x\in[0,1]$. We now consider two regimes.
If $\theta\le 1$, we have
$$
	\beta^{\frac{1}{2}}\mathrm{B}(\beta^{-\frac{1}{2}}\theta) \ge
	2\beta^{\frac{1}{2}} \ge
	2-2\sigma_*(X)\sqrt{r} =
	\mathrm{B}(\theta)-2\sigma_*(X)\sqrt{r}.	
$$
On the other hand, if $\theta>1$, then we have
$$
	\beta^{\frac{1}{2}}\mathrm{B}(\beta^{-\frac{1}{2}}\theta) 
	= \theta + \frac{\beta}{\theta} =
	\mathrm{B}(\theta) - \frac{\sigma_*(X)^2r}{\theta}
	\ge \mathrm{B}(\theta)-\sigma_*(X)\sqrt{r}
$$
as $\sigma_*(X)^2r\le \sigma_*(X)\sqrt{r}$. The proof is complete.
\end{proof}

Theorem \ref{thm:bbp} follows immediately by combining
Lemmas \ref{lem:bbpupper} and \ref{lem:bbplower}.

\subsection{Proof of Theorem \ref{thm:eigenvec}}

Despite that we formulated Theorem \ref{thm:eigenvec} in the context
of random matrices, the argument is entirely deterministic in nature.
The proof is based on the following basic observation.

\begin{lem}
\label{lem:perturb}
Let $X,P\in\mathrm{M}_d(\mathbb{C})_{\rm sa}$ and $t>0$. Then
$$
	\frac{\lambda_{\rm max}(X)-\lambda_{\rm max}(X-tP)}{t}
	\le
	\langle v_{\rm max}(X),P v_{\rm max}(X)\rangle 
	\le
	\frac{\lambda_{\rm max}(X+tP)-\lambda_{\rm max}(X)}{t}
$$
for any unit norm eigenvector $v_{\rm max}(X)$ of $X$ with eigenvalue
$\lambda_{\rm max}(X)$.
\end{lem}

\begin{proof}
To prove the upper bound, note that we obtain
\begin{align*}
	\lambda_{\rm max}(X+tP)-\lambda_{\rm max}(X) &=
	\sup_{\|v\|=1}\langle v,(X+tP)v\rangle -
	\langle v_{\rm max}(X),X v_{\rm max}(X)\rangle \\ &\ge
	t\langle v_{\rm max}(X),Pv_{\rm max}(X)\rangle
\end{align*}
by choosing $v\leftarrow v_{\rm max}(X)$ in the supremum. The
lower bound follows immediately if we replace $t\leftarrow {-t}$ in the 
above inequality.
\end{proof}

To exploit these inequalities in the setting of Theorem \ref{thm:bbp},
we must estimate the bounded differences of the function 
$\mathrm{B}(\cdot)$.

\begin{lem}
\label{lem:Bderiv}
For any $t>0$, we have
$$
	\frac{\mathrm{B}(\theta+t)-\mathrm{B}(\theta)}{t} \le
	\bigg(1-\frac{1}{\theta^2}\bigg)_+
	+
	t,\qquad
	\frac{\mathrm{B}(\theta)-\mathrm{B}(\theta-t)}{t} \ge
	\bigg(1-\frac{1}{\theta^2}\bigg)_+
	- t.
$$
\end{lem}

\begin{proof}
We readily compute
$$
	\frac{d \mathrm{B}(\theta)}{d\theta} =
	\bigg(1-\frac{1}{\theta^2}\bigg)_+,\qquad\quad
	\frac{d^2 \mathrm{B}(\theta)}{d\theta^2} =
	\frac{2}{\theta^3}1_{\theta\ge 1} \le 2.
$$
Taylor expanding to first order yields
$$
	\frac{\mathrm{B}(\theta+t) -
	\mathrm{B}(\theta)}{t} =
	\bigg(1-\frac{1}{\theta^2}\bigg)_+
	+
	t \int_0^1 
	\frac{2}{(\theta+rt)^3}1_{\theta+rt\ge 1}\,(1-r)\,dr.
$$
The upper bound in the statement follows using 
$\frac{2}{(\theta+rt)^3}1_{\theta+rt\ge 1}\le 2$. The lower bound follows 
by the identical argument once we replace $t\leftarrow -t$.
\end{proof}

We can now complete the proof.

\begin{proof}[Proof of Theorem \ref{thm:eigenvec}]
We begin by writing
$$
	\frac{\lambda_{\rm max}(X)-\lambda_{\rm max}(X_{-t})}{t}
	\le
	\langle v_{\rm max}(X),1_{(\theta-\delta,\theta]}(\EE X)
		v_{\rm max}(X)\rangle 
	\le
	\frac{\lambda_{\rm max}(X_t)-\lambda_{\rm max}(X)}{t}
$$
using Lemma \ref{lem:perturb}.
If in addition
$$
	|\lambda_{\rm max}(X_s)
	-\mathrm{B}(\lambda_{\rm max}(\EE X_s))| \le \varepsilon
	\text{ for }s\in\{0,\pm t\},
$$
Lemma \ref{lem:Bderiv} yields
$$	
	\bigg|
	\langle v_{\rm max}(X),1_{(\theta-\delta,\theta]}(\EE X)
		v_{\rm max}(X)\rangle 
	-
	\bigg(1-\frac{1}{\theta^2}\bigg)_+
	\bigg|
	\le
	t+\frac{2\varepsilon}{t},
$$
where we used that $\lambda_{\rm max}(\EE X_t) = \theta+t$ and 
$\lambda_{\rm max}(\EE X_{-t}) = \theta-t$ (because $t\le\delta$).
It remains to note that the above condition holds with high probability
$$
	\mathbf{P}\big[
	|\lambda_{\rm max}(X_s)
	-\mathrm{B}(\lambda_{\rm max}(\EE X_s))|
	\le \varepsilon\text{ for }s\in\{0,\pm t\}\big]\ge 1-3\rho
$$
by the union bound, concluding the proof.
\end{proof}

\section{Phase transitions: anisotropic case}
\label{sec:pfaniso}

The aim of this section is to prove Theorem \ref{thm:krz}. The 
proof consists of several distinct parts. In section \ref{sec:reduction},
 we apply a general reduction principle to reduce the dimension of the 
Lehner variational formula \eqref{eq:lehner}. In section 
\ref{sec:anisoapprox}, we approximate the Lehner formula for $X_{\rm 
free},X_{\varnothing,\rm free}$ by simplified parameters 
$\lambda,\lambda_\varnothing$ using the low-rank structure of the model. 
We also obtain the quantitative bound on $\lambda_\varnothing$. We 
subsequently prove the phase transition of $\lambda$ in section 
\ref{sec:anisophase}.

\begin{notn}
The following notations will be used primarily in this section.
For any vector $v\in\mathbb{C}^d$ and 
matrix $M\in\mathrm{M}_d(\mathbb{C})$, we will denote
$$
	\frac{1}{v} := \begin{bmatrix} v_1^{-1} \\ \vdots \\ v_d^{-1}
	\end{bmatrix},\qquad
        \diag(v) := \begin{bmatrix}
        v_1      &        &      \\
                 & \ddots &      \\
                 &        & v_d
        \end{bmatrix},\qquad
        \diaginv(M) :=
        \begin{bmatrix} M_{11} \\ \vdots \\ M_{dd}
        \end{bmatrix}.
$$
We will denote by $\I_{C_k}\in\mathbb{C}^d$ the 
indicator $(\I_{C_k})_i = 1_{i\in C_k}$ and by 
$\II_{C_k}:=\diag(\I_{C_k})$.
The elementwise (Hadamard) product of vectors or matrices is denoted
as $\odot$.
\end{notn}

\subsection{Reduction}
\label{sec:reduction}

\subsubsection{A general reduction principle}

The Lehner formula \eqref{eq:lehner} is a minimization problem over 
$d\times d$ matrices. However, one can often reduce the dimension of the 
variational problem in models with invariant structure. The following 
general reduction principle greatly facilitates the analysis of such 
models.

\begin{lem}[Reduction principle]
\label{lem:reduction}
Let $X$ be any $d\times d$ self-adjoint random matrix and let
$\mathcal{A}$ be any $*$-subalgebra of
$\mathrm{M}_d(\mathbb{C})$. Suppose that 
$$
	\EE X\in\mathcal{A},
	\qquad\quad
	\EE[(X-\EE X)M(X-\EE X)]\in\mathcal{A}
	~\text{ for all }~
	M\in\mathcal{A}.
$$
Then we have
$$
	\lambda_{\rm max}(X_{\rm free}) = 
	\inf_{M\in\mathcal{A}:M>0} 
	\lambda_{\rm max}\big(
	\mathbf{E} X + 
	M^{-1} + 
	\mathbf{E}[(X-\mathbf{E}X)M(X-\mathbf{E}X)]
	\big).	
$$
\end{lem}

\begin{proof}
That $\lambda_{\rm max}(X_{\rm free})$ is upper bounded by the expression
in the statement is obvious from \eqref{eq:lehner}. It remains 
to prove the corresponding lower 
bound. To this end, let $\pi:\mathrm{M}_d(\mathbb{C})\to\mathcal{A}$ be 
the conditional
expectation given $\mathcal{A}$ (cf.\ \cite[\S 4.3]{Car10}).
As conditional expectations are monotone,
any $M\in\mathrm{M}_d(\mathbb{C})$ with $M>0$ satisfies
\begin{multline*}
	\lambda_{\rm max}\big(
	\mathbf{E} X + 
	M^{-1} + 
	\mathbf{E}[(X-\mathbf{E}X)M(X-\mathbf{E}X)]
	\big)
	\ge \\
	\lambda_{\rm max}\big(
	\mathbf{E} X + 
	\pi(M)^{-1} + 
	\pi(\mathbf{E}[(X-\mathbf{E}X)M(X-\mathbf{E}X)])
	\big),
\end{multline*}
where we used $\EE X\in\mathcal{A}$ and that
$\pi(M^{-1})\ge \pi(M)^{-1}$ by
\cite[Theorem 4.16]{Car10}.

We now claim that
$$
	\pi(\mathbf{E}[(X-\mathbf{E}X)M(X-\mathbf{E}X)]) =
	\mathbf{E}[(X-\mathbf{E}X)\pi(M)(X-\mathbf{E}X)].
$$
Indeed, note that $\sigma:M\mapsto 
\mathbf{E}[(X-\mathbf{E}X)M(X-\mathbf{E}X)]$ is a self-adjoint linear map 
on $\mathrm{M}_d(\mathbb{C})$ with respect to the Hilbert-Schmidt inner 
product. As we assumed $\sigma$ leaves $\mathcal{A}$ invariant and as it 
is self-adjoint, it leaves $\mathcal{A}^\perp$ invariant as well. Thus the 
claimed identity follows by writing $M=\pi(M)+M^\perp$ with 
$M^\perp\in\mathcal{A}^\perp$.

Combining the above observations with \eqref{eq:lehner} yields
$$
	\lambda_{\rm max}(X_{\rm free}) \ge
	\inf_{M>0}
	\lambda_{\rm max}\big(
	\mathbf{E} X + 
	\pi(M)^{-1} + 
	\mathbf{E}[(X-\mathbf{E}X)\pi(M)(X-\mathbf{E}X)]
	\big),
$$
and the conclusion follows immediately.
\end{proof}

\subsubsection{The invariant algebra}

From now on we assume that 
$X,X_{\varnothing}$ are defined according to the model in section 
\ref{sec:krzdefn}.
The first 
step 
in our analysis will be to introduce a specific invariant $*$-algebra 
$\mathcal{A}$ for this model, to which Lemma \ref{lem:reduction} can be 
applied.
To this end,
define $f_k\in\mathbb{C}^d$ and $P_k\in\mathrm{M}_d(\mathbb{C})$ as
$$
	f_k := \frac{z\odot\I_{C_k}}{\sqrt{|C_k|}},\qquad\qquad
	P_k := \II_{C_k}-f_kf_k^*.
$$
The assumptions of section \ref{sec:krzdefn}
imply that $f_1,\ldots,f_q$ are orthonormal, $P_1,\ldots,P_q$
are orthogonal projections onto nontrivial (as $|C_k|>1$) orthogonal 
subspaces, and $P_1+\cdots+P_q$ is the orthogonal projection onto
$\{f_k:k\in[q]\}^\perp$.

\begin{defn}
Define the $*$-subalgebra
$$
	\mathcal{A} := \{\mathrm{A}(M,v):M\in\mathrm{M}_q(\mathbb{C}),~
	v\in\mathbb{C}^q\}
$$
of $\mathrm{M}_d(\mathbb{C})$, where
$$
	\mathrm{A}(M,v):=
	\sum_{k,l=1}^q M_{kl} f_kf_l^* +
	\sum_{k=1}^q v_k P_k.
$$
\end{defn}

\begin{rem}
Note that $\mathrm{A}(M,v)\simeq
M\oplus v_1\id_{|C_1|-1}\oplus\cdots\oplus v_q\id_{|C_q|-1}$, so that we 
have
$\lambda_{\rm max}(\mathrm{A}(M,v))=\max\{\lambda_{\rm max}(M),\max_i 
v_i\}$. This will be used repeatedly below.
\end{rem}

The following two lemmas show that $\mathcal{A}$ satisfies the assumptions 
of Lemma \ref{lem:reduction}.

\begin{lem}
\label{lem:meanA}
$
	\EE X = \mathrm{A}\big(
	\diag(c)^{\frac{1}{2}}B\diag(c)^{\frac{1}{2}} -
	\diag(Bc),
	-Bc
	\big)
	\in\mathcal{A}.
$
\end{lem}

\begin{proof}
Note that $\mathbf{B}=\sum_{k,l=1}^q B_{kl} \I_{C_k}\I_{C_l}^*$, so
$$
	\frac{1}{d}\diag(z)\mathbf{B}\diag(z)=
	\sum_{k,l=1}^q \sqrt{c_k}B_{kl}\sqrt{c_l}\, f_kf_l^*,
	\qquad
	\frac{1}{d}\mathbf{B}1_d =
	\sum_{k,l=1}^q B_{kl}c_l\, \I_{C_k}.
$$
The conclusion follows using $\II_{C_k} = P_k + f_kf_k^*$.
\end{proof}

\begin{lem}
\label{lem:varA}
$\mathbf{E}[(X-\EE X)A(X-\EE X)]\in\mathcal{A}$ for every
$A\in\mathcal{A}$. More precisely,
\begin{align*}
	\mathbf{E}[(X-&\EE X)\,\mathrm{A}(M,v)\,(X-\EE X)]
	=\\
	\mathrm{A}\bigg(
	&\diag\Big(
	B\Big(c\odot v + \frac{1}{d}(\diaginv(M)-v)
	\Big)\Big) +
	\frac{1}{d}\,B\odot M^T,\\
	&
	B\Big(c\odot v + \frac{1}{d}(\diaginv(M)-v)
        \Big) + \frac{1}{d}\,v\odot\diaginv(B)
	\bigg)
\end{align*}
for all $M\in\mathrm{M}_q(\mathbb{C})$ and $v\in\mathbb{C}^q$, 
where $M^T$ denotes the transpose of $M$.
\end{lem}

\begin{proof}
Recall that $X-\EE X=G$ (cf.\ section \ref{sec:krzdefn}).
For any $A\in\mathrm{M}_d(\mathbb{C})$, we compute
\begin{align*}
	\EE[GAG]_{kl} &= 
	\sum_{r,s} A_{rs}\,\EE[G_{kr}G_{sl}] =
	\frac{1}{d} \sum_{r,s} \mathbf{B}_{kr}A_{rs}
	(1_{k=s,r=l}+1_{k=l,r=s}) \\
	&=
	\Big(
	\frac{1}{d}\,\mathbf{B}
	\odot A^T +
	\frac{1}{d}\diag(\mathbf{B}(\diaginv(A)) 
	\Big)_{kl}.
\end{align*}
We readily compute
$\mathbf{B}\odot \mathrm{A}(M,v)^T =
\mathrm{A}(B\odot M^T,v\odot \diaginv(B))$, 
while
$$
	\frac{1}{d}\diag(\mathbf{B}(\diaginv(\mathrm{A}(M,v)))
        =
	\sum_{k} 
	\Big[
	B\Big(c\odot v +
	\frac{1}{d}
	(\diaginv(M)-v)
	\Big)\Big]_k
	\II_{C_k}
$$
using $\mathbf{B}=\sum_{k,l} B_{kl} \I_{C_k}\I_{C_l}^*$.
The result follows as $\II_{C_k} = P_k + f_kf_k^*$.
\end{proof}

\subsection{The simplified parameters}
\label{sec:anisoapprox}

We now aim to approximate $\lambda_{\rm max}(X_{\rm free})$ and 
$\lambda_{\rm max}(X_{\varnothing,\rm free})$ by simplified parameters 
$\lambda,\lambda_\varnothing$: we will use the reduction principle of the 
previous section to reduce the variational principle \eqref{eq:lehner} for 
$d$-dimensional matrices to a variational principle for $q$-dimensional 
vectors, and we will eliminate all the terms of order $\frac{1}{d}$ in 
Lemmas \ref{lem:meanA} and \ref{lem:varA}. We first consider $\lambda$.

\begin{prop}
\label{prop:lambda}
Define
\begin{align*}
	\lambda :=
	\inf_{v>0} \max\big\{
	&
	\lambda_{\rm max}\big(
		\diag(c)^{\frac{1}{2}}B\diag(c)^{\frac{1}{2}} +
		\diag\big(B\diag(c)(v-1_q)\big)
	\big), \\
	&
	\lambda_{\rm max}\big(
		\diag(v)^{-1} + \diag\big(B\diag(c)(v-1_q)\big)
	\big)
	\big\}.
\end{align*}
Then
$$
	|\lambda_{\rm max}(X_{\rm free})-\lambda| \le
	\sqrt{\frac{8\|B1_q\|_{\infty}}{d}}.
$$
\end{prop}

\begin{proof}
The Lehner formula \eqref{eq:lehner} and Lemmas \ref{lem:reduction},
\ref{lem:meanA}, and \ref{lem:varA} yield
$$
	\lambda_{\rm max}(X_{\rm free}) = 
	\inf_{M,v>0} 
	\lambda_{\rm max}\big(
	\mathbf{E} X + 
	\mathrm{A}(M^{-1},\tfrac{1}{v})
	+ 
	\mathbf{E}[G\,\mathrm{A}(M,v)\,G]
	\big)
$$
using $\mathrm{A}(M,v)^{-1}=\mathrm{A}(M^{-1},\tfrac{1}{v})$.
Moreover,
$$
	\lambda = \inf_{v>0}
	\lambda_{\rm max}\big(
	\mathbf{E} X + 
	\mathrm{A}(0,\tfrac{1}{v})
	+ 
	\mathrm{A}(\diag(B\diag(c)v),B\diag(c)v)
	\big)
$$
by Lemma \ref{lem:meanA}. We must upper and lower bound
$\lambda_{\rm max}(X_{\rm free})$ in terms of $\lambda$.

\medskip

\textbf{Upper bound.} We can read off from Lemma \ref{lem:varA} that
$$
	\EE[G\,\mathrm{A}(\id_q,0)\,G]
	\le \frac{2\|B1_q\|_\infty}{d},\quad
	\EE[G\,\mathrm{A}(0,v)\,G] \le
	\mathrm{A}(\diag(B\diag(c)v),B\diag(c)v)
$$
for $v>0$, where we used $v\odot\diaginv(B)-Bv\le 0$. Restricting the
infimum over $M$ in the variational principle for $\lambda_{\rm 
max}(X_{\rm free})$ to $M=\gamma\id_q$ for $\gamma>0$ yields
\begin{align*}
	\lambda_{\rm max}(X_{\rm free}) &\le
	\inf_{v>0} 
	\lambda_{\rm max}\big(
	\mathbf{E} X + 
	\mathrm{A}(0,\tfrac{1}{v})
	+ 
	\mathbf{E}[G\,\mathrm{A}(0,v)\,G]
	\big) + 
	\inf_{\gamma>0}
	\bigg(\frac{1}{\gamma} +
	\frac{2\gamma\|B1_q\|_\infty}{d}\bigg) \\
	&\le
	\lambda + \sqrt{\frac{8\|B1_q\|_\infty}{d}},
\end{align*}
where we used that $\mathrm{A}(M,v)=\mathrm{A}(M,0) + \mathrm{A}(0,v)$.

\medskip

\textbf{Lower bound.} We can read off from Lemma \ref{lem:varA} that
$$
	\EE[G\,\mathrm{A}(0,v)\,G] \ge
	\mathrm{A}(\diag(B\diag(c-\tfrac{1}{d}1_q)v),
	B\diag(c-\tfrac{1}{d}1_q)v)
$$
for $v>0$. Then we can estimate
\begin{align*}
	\lambda &=
	\inf_{v,w>0}\lambda_{\rm max}\big(
	\mathbf{E} X + 
	\mathrm{A}(0,\tfrac{1}{v\wedge w}) +
	\tfrac{1}{d}\mathrm{A}(\diag(B(v\wedge w)),B(v\wedge w))
\\
	&
	\phantom{\mbox{}=\inf_{v,w>0}\lambda_{\rm max}\big(\mathbf{E} X}
	+ 
	\mathrm{A}(\diag(B\diag(c-\tfrac{1}{d}1_q)(v\wedge w)),
			 B\diag(c-\tfrac{1}{d}1_q)(v\wedge w))
	\big) 
\\
	&\le
	\inf_{v,w>0}\lambda_{\rm max}\big(
	\mathbf{E} X + 
	\mathrm{A}(0,\tfrac{1}{v}) +
	\mathrm{A}(0,\tfrac{1}{w}) +
	\tfrac{1}{d}\mathrm{A}(\diag(Bw),Bw)
	+ \EE[G\,\mathrm{A}(0,v)\,G]\big)	
\\
	&\le
	\lambda_{\rm max}(X_{\rm free}) +
	\inf_{w>0}\lambda_{\rm max}\big(
	\mathrm{A}(0,\tfrac{1}{w})+\tfrac{1}{d}\mathrm{A}(\diag(Bw),Bw)
	\big),
\end{align*}
where $v\wedge w$ denotes the elementwise minimum, and
we used $\frac{1}{v\wedge w}\le \frac{1}{v}+\frac{1}{w}$. Choosing
$w\leftarrow \sqrt{d}\|B1_q\|_\infty^{-\frac{1}{2}}1_q$ on the last line 
concludes the proof.
\end{proof}

The parameter $\lambda_{\varnothing}$ arises in a completely analogous 
fashion.

\begin{prop}
\label{prop:lambda0}
Define
$$
	\lambda_\varnothing :=
	\inf_{v>0} 
	\lambda_{\rm max}\big(
		\diag(v)^{-1} + \diag\big(B\diag(c)(v-1_q)\big)
	\big).
$$
Then
$$
	|\lambda_{\rm max}(X_{\varnothing,\rm free})-\lambda_\varnothing| \le
	\sqrt{\frac{8\|B1_q\|_{\infty}}{d}}.
$$
\end{prop}

\begin{proof}
Note that the only difference between the definitions of $X$ and 
$X_\varnothing$ is that $\EE X$ is replaced by $\EE X_\varnothing = 
\mathrm{A}({-\diag(Bc)},{-Bc})\in\mathcal{A}$ in Lemma \ref{lem:meanA}. 
Thus the proof of Proposition \ref{prop:lambda} carries over verbatim to
the present setting.
\end{proof}

We can now prove the upper bound on $\lambda_\varnothing$ in Theorem 
\ref{thm:krz}. Recall that $b>0$ denotes the Perron-Frobenius (right) 
eigenvector of $B\diag(c)$.

\begin{lem}
\label{lem:lambda0est}
We have
$$
	\lambda_\varnothing \le 1-
	\frac{\min_i b_i}{\max_i b_i}
	\big(1-\lambda_{\rm 
	max}(\diag(c)^{\frac{1}{2}}B\diag(c)^{\frac{1}{2}})^{\frac{1}{2}}\big)^2.
$$
\end{lem}

\begin{proof}
Denote $\lambda_{\rm cr}:=\lambda_{\rm 
max}(\diag(c)^{\frac{1}{2}}B\diag(c)^{\frac{1}{2}})$ for simplicity. As
$\diag(c)^{\frac{1}{2}}b$ is a positive eigenvector of 
$\diag(c)^{\frac{1}{2}}B\diag(c)^{\frac{1}{2}}$, the Perron-Frobenius
theorem implies that its eigenvalue must be maximal, and thus
$B\diag(c)b=\lambda_{\rm cr}b$.
We now upper bound $\lambda_\varnothing$ by restricting the infimum in its 
definition to $v=1_q+t b$. This yields
$$
	\lambda_\varnothing \le
	\inf_{t: 1_q+tb>0} 
	\max_i
	\bigg\{
	\frac{1}{1+tb_i} +
	t\lambda_{\rm cr}b_i
	\bigg\}
	=
	1 -
	\sup_{t: 1_q+tb>0} 
	\min_i
	\bigg\{
	\frac{tb_i}{1+tb_i} -
	t\lambda_{\rm cr}b_i
	\bigg\}.
$$
If we choose $t=
(\lambda_{\rm cr}^{-\frac{1}{2}}-1)\frac{1}{\max_i b_i}$, then
$1_q+tb>0$ and
$$
	\min_i
	\bigg\{
	\frac{tb_i}{1+tb_i} -
	t\lambda_{\rm cr}b_i
	\bigg\} =
	\min_i
	\bigg\{
	\frac{1}{
	\lambda_{\rm cr}^{\frac{1}{2}}
	+(1-\lambda_{\rm cr}^{\frac{1}{2}})\frac{b_i}{\max_i b_i}}
	-
	\lambda_{\rm cr}^{\frac{1}{2}}
	\bigg\} 
	(1-\lambda_{\rm cr}^{\frac{1}{2}})\frac{b_i}{\max_i b_i}.
$$
The conclusion follows as
$$
	\bigg(\frac{1}{x+(1-x)a} - x\bigg)(1-x) \ge
	(1-x)^2
$$
for all $x>0$ and $0\le a\le 1$.
\end{proof}

\subsection{The phase transition}
\label{sec:anisophase}

It remains to prove the phase transition for $\lambda$. To this end, we 
first develop in section \ref{sec:aniso1} some basic properties of the 
minimizers in the definitions of $\lambda$ and $\lambda_\varnothing$. 
While we will restrict attention to the present model, the methods used 
here are quite general and extend to other Lehner-type variational 
principles. We then exploit the special structure of the present model in 
section~\ref{sec:aniso2} to complete the proof of Theorem \ref{thm:krz}.

Before we begin the proof, let us make a minor simplification: while we 
assumed only that the matrix $B$ is irreducible, we can assume without 
loss of generality that $B$ has strictly positive entries in the remainder 
of the proof. Indeed, it is clear that all the quantities that appear in 
Theorem \ref{thm:krz} are continuous in $B$.  When $\lambda_{\rm 
max}(\diag(c)^{\frac{1}{2}}B\diag(c)^{\frac{1}{2}})\ne 1$, we can apply 
the result for $B\leftarrow B+\varepsilon 1_q1_q^*$ and let 
$\varepsilon\downarrow 0$ (the preservation of the strict inequality 
$\lambda_\varnothing<1$ in the limit follows from the quantitative 
estimate on $\lambda_\varnothing$). The case $\lambda_{\rm 
max}(\diag(c)^{\frac{1}{2}}B\diag(c)^{\frac{1}{2}})=1$ now follows by 
applying the result to $B\leftarrow tB$ and letting $t\to 1$ from above 
and below.

\subsubsection{Basic properties of the minimizers}
\label{sec:aniso1}

The following basic but important result collects a number of general 
properties of the variational principle that defines 
$\lambda_\varnothing$: existence and uniqueness of a minimizer, and 
first-order optimality conditions.

\begin{lem}
\label{lem:firstordbulk}
The infimum in the definition of $\lambda_\varnothing$
(Proposition \ref{prop:lambda0}) is attained at a unique vector
$v_\varnothing^*>0$. Moreover, this minimizer satisfies the optimality
conditions
\begin{equation}
\label{eq:firstopt1}
	\frac{1}{v_\varnothing^*} + B\diag(c)(v_\varnothing^*-1_q) =
	\lambda_\varnothing 1_q
\end{equation}
and
\begin{equation}
\label{eq:firstopt2}
	\lambda_{\rm max}\big(\diag(c)^{\frac{1}{2}}B\diag(c)^{\frac{1}{2}}
	- \diag(v_\varnothing^*)^{-2}\big) = 0.
\end{equation}
\end{lem}

\begin{proof}
Let us write $\lambda_\varnothing=\inf_{v>0}f(v)$ with 
$f(v)=\max_i(\frac{1}{v} + B\diag(c)(v-1_q))_i$. The existence of a 
minimizer $v_\varnothing^*>0$ follows by a routine compactness argument 
and as $f(v)$ diverges if $v_i\to\{0,\infty\}$ for any $i$.

Next, we show that \eqref{eq:firstopt1} must hold for \emph{any} 
minimizer. Suppose $v>0$ satisfies
$$
	\bigg(\frac{1}{v} + B\diag(c)(v-1_q)\bigg)_j <
	\max_i
	\bigg(\frac{1}{v} + B\diag(c)(v-1_q)\bigg)_i = \lambda_\varnothing
$$
for some $j$. As $B,c$ have positive entries, slightly decreasing $v_j$ 
will strictly
decrease all $(\frac{1}{v} + B\diag(c)(v-1_q))_i$ for $i\ne j$ while
preserving $(\frac{1}{v} + B\diag(c)(v-1_q))_j<\lambda_{\varnothing}$.
The perturbed $v$ would therefore satisfy $f(v)<\lambda_{\varnothing}$,
contradicting the
definition of $\lambda_\varnothing$.
We conclude that any minimizer must satisfy \eqref{eq:firstopt1}.

We now prove uniqueness. Let
$\lambda_\varnothing=f(v_0)=f(v_1)$ and define $v_t=(1-t)v_0+tv_1$ for
$t\in[0,1]$. As $f$ is convex, we have
$\lambda_\varnothing \le f(v_t) \le (1-t)f(v_0)+tf(v_1) =
\lambda_\varnothing$,
so $v_t$ is also a minimizer. As we have shown \eqref{eq:firstopt1} holds
for any minimizer, we have
$$
	0 = \frac{d^2}{dt^2}\bigg(
	\frac{1}{v_t} + B\diag(c)(v_t-1_q)
	\bigg)_i =
	2v_{t,i}^{-3}(v_1-v_0)_i^2
$$
for all $i$, which implies $v_0=v_1$. Thus the minimizer is unique.

It remains to prove \eqref{eq:firstopt2}. Note that as $B,c$ 
have positive entries, the Perron-Frobenius theorem yields\footnote{%
If $M$ is a self-adjoint matrix with nonnegative off-diagonal entries, 
$M+c\id$ is a nonnegative matrix for sufficiently large $c$. We can 
therefore apply the Perron-Frobenius theorem to the latter to deduce the
existence of a positive eigenvector of $M$ associated to its maximal 
eigenvalue.}
an eigenvector
$w>0$ associated to the maximal eigenvalue $\mu$ of
$\diag(c)^{\frac{1}{2}}B\diag(c)^{\frac{1}{2}}-\diag(v_\varnothing^*)^{-2}$.
Let $v_t = v_\varnothing^* - 
t\mu\diag(c)^{-\frac{1}{2}}w$, so that
$$
	\frac{d}{dt}\bigg(
        \frac{1}{v_t} + B\diag(c)(v_t-1_q)
        \bigg)\bigg|_{t=0} =
	-\mu^2\diag(c)^{-\frac{1}{2}}w.
$$
If $\mu\ne 0$, all entries of this vector are strictly negative, which would 
imply that
$f(v_t)<f(v_0)=\lambda_\varnothing$ for $t$ sufficiently small. This 
contradicts the
definition of $\lambda_\varnothing$. We must therefore have
$\mu=0$, which is \eqref{eq:firstopt2}.
\end{proof}

We now prove a partial counterpart of the above lemma for the variational
principle that defines $\lambda$. While more information could be obtained 
also in this case, we only prove the properties that will be needed below.

\begin{lem}
\label{lem:firstordedge}
The infimum in the definition of $\lambda$
(Proposition \ref{prop:lambda}) is attained at a vector
$v^*>0$. Moreover, this minimizer satisfies 
\begin{equation}
\label{eq:firstopt1edge}
	\frac{1}{v^*} + B\diag(c)(v^*-1_q) =
	\lambda 1_q
\end{equation}
and
\begin{equation}
\label{eq:firstopt2edge}
	\lambda_{\rm max}\big(\diag(c)^{\frac{1}{2}}B\diag(c)^{\frac{1}{2}}
	- \diag(v^*)^{-1}\big) \le 0.
\end{equation}
\end{lem}

\begin{proof}
The existence of a minimizer $v_*>0$ follows as in the proof of Lemma 
\ref{lem:firstordbulk}. Now suppose there is a coordinate $j$ so that 
$v^*$ satisfies
$$
	\bigg(\frac{1}{v^*} + B\diag(c)(v^*-1_q)\bigg)_j < \lambda.
$$
As $B,c$ have positive entries, we can reason as in 
the proof of Lemma \ref{lem:firstordbulk} that slightly decreasing the
$j$th coordinate of $v^*$ will yield a strictly smaller value of the 
function being minimized in the definition of $\lambda$, contradicting 
the minimality of $v^*$. We conclude that $v^*$ must satisfy 
\eqref{eq:firstopt1edge}. Finally, \eqref{eq:firstopt2edge} follows
from \eqref{eq:firstopt1edge} and as
$$
	\lambda_{\rm max}\big(
                \diag(c)^{\frac{1}{2}}B\diag(c)^{\frac{1}{2}} +
                \diag\big(B\diag(c)(v^*-1_q)\big)\big)\le\lambda
$$
by the definition of $\lambda$.
\end{proof}

It is obvious from the definitions of $\lambda,\lambda_\varnothing$ that 
$\lambda_\varnothing\le\lambda$. The aim of the remainder of the proof is 
to characterize the phase transition from $\lambda_\varnothing<\lambda$ to 
$\lambda_\varnothing=\lambda$. A basic characterization of the phase 
regions follows directly from the variational principles.

\begin{lem}
\label{lem:phasereg}
$\lambda=\lambda_\varnothing$ if and only if
$\lambda_{\rm max}(\diag(c)^{\frac{1}{2}}B\diag(c)^{\frac{1}{2}}
- \diag(v^*_\varnothing)^{-1}) \le 0$.
\end{lem}

\begin{proof}
If $\lambda_{\rm max}(\diag(c)^{\frac{1}{2}}B\diag(c)^{\frac{1}{2}}
- \diag(v^*_\varnothing)^{-1}) \le 0$, then choosing
$v\leftarrow v^*_\varnothing$ in the definition of $\lambda$ 
(cf.\ Proposition \ref{prop:lambda}) and using \eqref{eq:firstopt1} yields
$\lambda\le\lambda_\varnothing$. As $\lambda_\varnothing\le\lambda$ holds 
trivially by the definitions of $\lambda,\lambda_\varnothing$,
we conclude that $\lambda=\lambda_\varnothing$.

Now suppose that $\lambda=\lambda_\varnothing$. Then
$$
	\lambda_{\rm max}\big(
		\diag(v^*)^{-1} + \diag\big(B\diag(c)(v^*-1_q)\big)
	\le \lambda = \lambda_\varnothing
$$
by the definition of $\lambda$, which implies that $v^*$ is a minimizer
in the definition of $\lambda_\varnothing$ (cf.\ Proposition 
\ref{prop:lambda0}). But Lemma \ref{lem:firstordbulk} shows the latter 
is unique, so that $v^*=v^*_\varnothing$. 
Thus \eqref{eq:firstopt2edge} yields 
$\lambda_{\rm max}(\diag(c)^{\frac{1}{2}}B\diag(c)^{\frac{1}{2}}
- \diag(v^*_\varnothing)^{-1}) \le 0$.
\end{proof}

The difficulty in applying this lemma is that the phase transition 
criterion is not explicit as it involves $v^*_\varnothing$.
In the rest of the proof, we will exploit the special properties of the 
present model to explicitly characterize the phase transition.

\subsubsection{Proof of Theorem \ref{thm:krz}}
\label{sec:aniso2}

The following fact could be viewed as the basic reason behind the 
special properties of the present model.

\begin{lem}
\label{lem:krzmagic}
Suppose a vector $v>0$ and $\mu\in\mathbb{R}$ satisfy
$$
	\frac{1}{v}+B\diag(c)(v-1_q)=\mu 1_q,\qquad
	\lambda_{\rm max}\big(\diag(c)^{\frac{1}{2}}B\diag(c)^{\frac{1}{2}}
	- \diag(v)^{-1}\big) = 0.
$$
Then we must have $\mu=1$.
\end{lem}

\begin{proof}
The key idea is that the first equation in the statement is 
equivalent to
\begin{equation}
\label{eq:krzmagic}
	\big(\diag(c)^{\frac{1}{2}}B\diag(c)^{\frac{1}{2}}
        - \diag(v)^{-1}\big)\diag(c)^{\frac{1}{2}}(v-1_q) =
	(\mu-1)\diag(c)^{\frac{1}{2}}1_q.
\end{equation}
As $B,c$ have positive entries, the Perron-Frobenius theorem and the
second equation in the statement yield an eigenvector $w>0$ of
$\diag(c)^{\frac{1}{2}}B\diag(c)^{\frac{1}{2}}- \diag(v)^{-1}$ with
eigenvalue $0$. Taking the inner product of the above equation with $w$
yields $0=(\mu-1)\langle w,\diag(c)^{\frac{1}{2}}1_q\rangle$, which 
implies
$\mu=1$ as $\langle w,\diag(c)^{\frac{1}{2}}1_q\rangle>0$.
\end{proof}

Using this result, we can explicitly determine $v_\varnothing^*$ on the 
boundary of the phase region $\lambda=\lambda_\varnothing$ (cf.\ Lemma 
\ref{lem:phasereg}). This is the key step in the proof.

\begin{lem}
\label{lem:phasebd}
If $\lambda_{\rm max}(\diag(c)^{\frac{1}{2}}B\diag(c)^{\frac{1}{2}}
- \diag(v^*_\varnothing)^{-1}) = 0$, then $v_\varnothing^*=1_q$.
\end{lem}

\begin{proof}
By Lemma \ref{lem:krzmagic}, the assumption and \eqref{eq:firstopt1}
imply that $\lambda_\varnothing=1$. Thus
$$
	\big(\diag(c)^{\frac{1}{2}}B\diag(c)^{\frac{1}{2}}
        - \diag(v_\varnothing^*)^{-1}\big)
	\diag(c)^{\frac{1}{2}}(v_\varnothing^*-1_q) = 0
$$
by \eqref{eq:krzmagic}. Now note that the Perron-Frobenius theorem and
\eqref{eq:firstopt2} provide an eigenvector $w>0$ of
$\diag(c)^{\frac{1}{2}}B\diag(c)^{\frac{1}{2}}
- \diag(v_\varnothing^*)^{-2}$ with eigenvalue $0$. Taking the inner 
product of the above equation with $w$ yields
$$
	0 =
	\langle w,
	(\diag(v_\varnothing^*)^{-2}
        - \diag(v_\varnothing^*)^{-1})
	\diag(c)^{\frac{1}{2}}(v_\varnothing^*-1_q)\rangle =
	-\sum_i w_i c_i^{\frac{1}{2}}
	\big((v^*_\varnothing)_i^{-1}-1\big)^2,
$$
which evidently implies $v^*_\varnothing=1_q$.
\end{proof}

Lemmas \ref{lem:phasereg} and \ref{lem:phasebd} show that we must have
$\lambda_{\rm max}(\diag(c)^{\frac{1}{2}}B\diag(c)^{\frac{1}{2}})=1$
on the boundary between the phase regions. This will enable us to 
fully characterize the phase regions by a continuity argument, completing
the proof of Theorem \ref{thm:krz}

\begin{proof}[Proof of Theorem \ref{thm:krz}]
The approximation by $\lambda,\lambda_\varnothing$ and
the estimate on $\lambda_\varnothing$ were proved in
Propositions \ref{prop:lambda} and \ref{prop:lambda0} and in
Lemma \ref{lem:lambda0est}, respectively.
The remainder of the proof will be completed in three steps to be
proved below:
\medskip
\begin{enumerate}[1.]
\itemsep\abovedisplayskip
\item $\lambda_{\rm max}(\diag(c)^{\frac{1}{2}}B\diag(c)^{\frac{1}{2}})>1$
implies $\lambda_\varnothing<\lambda$.
\item $\lambda_\varnothing<\lambda$ implies $\lambda=1$.
\item $\lambda_{\rm max}(\diag(c)^{\frac{1}{2}}B\diag(c)^{\frac{1}{2}})<1$
implies $\lambda<1$.
\end{enumerate}
\medskip
Indeed, combining steps 1 and 2 yields part \textit{c} of the theorem,
while combining steps 2 and 3 yields part \textit{a} of the theorem (as 
$\lambda_\varnothing\le\lambda$).
Part \textit{b} of the theorem now follows by applying the theorem to 
$B\leftarrow tB$ and letting $t\to 1$ from above and below.

It remains to prove each of the above steps.

\medskip

\textbf{Step 1.} Suppose $\lambda_{\rm 
max}(\diag(c)^{\frac{1}{2}}B\diag(c)^{\frac{1}{2}})>1$. By Lemma 
\ref{lem:phasereg} and as $\lambda_\varnothing\le\lambda$, it suffices to 
show that $\lambda_{\rm max}(\diag(c)^{\frac{1}{2}}B\diag(c)^{\frac{1}{2}}
- \diag(v^*_\varnothing)^{-1}) > 0$.

Consider first the special case $B=2\,1_q1_q^*$, so that
$\lambda_{\rm max}(\diag(c)^{\frac{1}{2}}B\diag(c)^{\frac{1}{2}})=2$
as $\sum_i c_i=1$. Then \eqref{eq:firstopt1} 
shows that $v_\varnothing^*$ is proportional to $1_q$, so it suffices to 
minimize over $v\leftarrow t1_q$ in the definition of 
$\lambda_\varnothing$. A straightforward computation yields
$v_\varnothing^*=2^{-\frac{1}{2}}1_q$ and thus
$\lambda_{\rm max}(\diag(c)^{\frac{1}{2}}B\diag(c)^{\frac{1}{2}}
- \diag(v^*_\varnothing)^{-1}) > 0$.

For general $B$, choose a continuous family $t\mapsto B(t)$
so that
$B(0)=2\,1_q1_q^*$, $B(1)=B$, and $\lambda_{\rm
max}(\diag(c)^{\frac{1}{2}}B(t)\diag(c)^{\frac{1}{2}})>1$ to all 
$t\in[0,1]$. Denote by $v_\varnothing^*(t)$ the minimizer in the 
definition of $\lambda_\varnothing$ for $B\leftarrow B(t)$. As the 
minimizer $v_\varnothing^*(t)$ is unique by Lemma \ref{lem:firstordbulk}, 
it follows by a routine argument that $t\mapsto v_\varnothing^*(t)$ is 
continuous. On the other hand, Lemma \ref{lem:phasebd} ensures that
for all $t\in[0,1]$
$$
	\alpha(t) :=
	\lambda_{\rm max}(\diag(c)^{\frac{1}{2}}B(t)\diag(c)^{\frac{1}{2}}
	- \diag(v^*_\varnothing(t))^{-1}) \ne 0:
$$
otherwise we would have $v_\varnothing^*(t)=1_q$ 
and thus $\lambda_{\rm 
max}(\diag(c)^{\frac{1}{2}}B(t)\diag(c)^{\frac{1}{2}})=1$
for some $t$, which entails a contradiction.
As we showed that $\alpha(0)>0$ and $\alpha(t)\ne 0$ for all $t$, it 
follows by continuity that $\alpha(1)>0$. This is the desired claim.

\medskip

\textbf{Step 2.} Suppose that $\lambda_\varnothing<\lambda$. To show
this implies $\lambda=1$, it suffices by \eqref{eq:firstopt1edge} and
Lemma \ref{lem:krzmagic} to show that
$\lambda_{\rm max}(\diag(c)^{\frac{1}{2}}B\diag(c)^{\frac{1}{2}}
- \diag(v^*)^{-1})=0$.

Suppose the latter is not the case. Then \eqref{eq:firstopt1edge}
and the definition of $\lambda$ imply
\begin{multline*}
	\lambda_{\rm max}\big(
		\diag(c)^{\frac{1}{2}}B\diag(c)^{\frac{1}{2}} +
		\diag(B\diag(c)(v^*-1_q))\big)
	< \lambda \\
	= \lambda_{\rm max}\big(
		\diag(v^*)^{-1} + \diag(B\diag(c)(v^*-1_q))\big).
\end{multline*}
Then $v^*$ must also be a minimizer of the quantity on the second line:
otherwise we could slightly decrease the quantity on the second line
while preserving the strict inequality on the first line, contradicting 
the definition of $\lambda$. This implies by the definition of 
$\lambda_\varnothing$ that $\lambda=\lambda_\varnothing$, which 
contradicts the assumption of step 2.

\medskip

\textbf{Step 3.}
Suppose that 
$\lambda_{\rm max}(\diag(c)^{\frac{1}{2}}B\diag(c)^{\frac{1}{2}})<1$.
Then it follows readily that $\lambda\le 1$ by choosing $v\leftarrow 1_q$ 
in the definition of $\lambda$.

Now suppose that $\lambda=1$. Then $v^*=1_q$ would be a minimizer in 
the definition of $\lambda$. The same argument as in the proof of 
step 2 now shows that $v^*$ must also be a minimizer in the definition of
$\lambda_{\varnothing}$, so that $v_\varnothing^*=1_q$. The latter 
contradicts \eqref{eq:firstopt2}. Thus we have shown that
$\lambda<1$, concluding the proof.
\end{proof}

\section{Applications: proofs}
\label{sec:pfappl}

\subsection{Decoding node labels on graphs}
\label{sec:pfmasked}

\begin{proof}[Proof of Theorem \ref{thm:masked}]
Define
$$
	Y' := \frac{Y}{(4kp(1-p))^{\frac{1}{2}}},\qquad\quad
	\theta' := \frac{k^{\frac{1}{2}}(1-2p)}{(4p(1-p))^{\frac{1}{2}}}.
$$
Then we clearly have
$$
	\EE Y' = \frac{\theta'}{k}\diag(x)A\diag(x),\qquad\quad
	\EE[(Y'-\EE Y')^2]=\id.
$$
Moreover, as $A1_d=k1_d$, the Perron-Frobenius theorem yields
$$
	\lambda_{\rm max}(\EE Y')=\theta',\qquad\quad
	v_{\rm max}(\EE Y')=d^{-\frac{1}{2}}x,
$$
while $1_{(\theta'-\delta,\theta']}(\EE Y')=d^{-1}xx^*$ for 
$\delta:=\frac{\theta'}{k}\lambda$.
Note for future reference that 
the assumptions of the theorem imply that $\theta'=(1+o(1))\theta$ and 
$k\gg (\log d)^4$.

Let $A=\sum_{i=1}^d \lambda_i v_iv_i^*$ be an eigendecomposition of 
$A$ so that $\lambda_1=k$ and $|\lambda_i|=\mathrm{s}_i$. Then
$A_r := \sum_{i=1}^r \lambda_i v_i v_i^*$ has rank at most $r$ 
and $\|A-A_r\|\le \mathrm{s}_{r+1}$. Define
$$
	X := Y' - \EE Y' +
	\frac{\theta'}{k}
        \diag(x)A_r\diag(x).
$$
As $\EE Y = (1-2p)\diag(x)A\diag(x)$, we can estimate
$$
	\|(4kp(1-p))^{\frac{1}{2}}X - Y\| \le
	k^{-\frac{1}{2}}\theta\,\mathrm{s}_{r+1}.
$$
On the other hand, we have
$$
	\mathbf{P}\big[
	|\lambda_{\rm max}(X)-\mathrm{B}(\theta')| >
	Ck^{-\frac{1}{2}}\sqrt{r} +
	Ck^{-\frac{1}{6}}(\log d)^{\frac{2}{3}} 
	\big]
	\le \frac{C}{d^2}
$$
by applying Theorems \ref{thm:main}, \ref{thm:univ}, and \ref{thm:bbp} 
with $t=3\log d$ and using that $\sigma(X)=1$, $\sigma_*(X)\le 
v(X)\lesssim k^{-\frac{1}{2}}$, $R\lesssim k^{-\frac{1}{2}}$,
and $k\gg (\log d)^4$. Therefore
$$
	\mathbf{P}\Big[
	|\lambda_{\rm max}(Y')-\mathrm{B}(\theta')| >
	C
	k^{-1}\Big\{
	\min_{1\le r\le k}\big\{\theta\,\mathrm{s}_{r+1} + \sqrt{rk}\big\} 
	+ k^{\frac{5}{6}}(\log d)^{\frac{2}{3}} 
	\Big\}\Big]
	\le \frac{C}{d^2},
$$
where we optimized over the choice of $r$. The analogous estimates follow
readily if we replace $Y'\leftarrow Y'+s1_{(\theta'-\delta,\theta']}(\EE Y')$
and $\theta'\leftarrow\theta'+s$ for $|s|\le\delta$.

To proceed, note that the assumption of the theorem and
$\theta'=(1+o(1))\theta$ imply
$$
	k^{-1}\Big\{
	\min_{1\le r\le k}\big\{\theta\,\mathrm{s}_{r+1} + \sqrt{rk}\big\} 
	+ k^{\frac{5}{6}}(\log d)^{\frac{2}{3}} 
	\Big\} \ll \min\{\delta,1\}.
$$
We can therefore conclude using Theorem \ref{thm:eigenvec} that
$$
	\mathbf{P}\bigg[
	\bigg|
	\frac{1}{d}|\langle x,v_{\rm max}(Y)\rangle|^2
	-
	\bigg(
	1-\frac{1}{\theta^2}
	\bigg)_+	
	\bigg| >
	t+\frac{o(\min\{\delta,1\})}{t}
	\bigg]\le \frac{C}{d^2}
$$
for $0<t\le\delta$. It remains to choose $0<t\le\delta$
so that
$t+\frac{o(\min\{\delta,1\})}{t}=o(1)$.

Finally, the existence of an estimator $\hat x(Y)$ follows from Lemma 
\ref{lem:round} below.
\end{proof}

At the end of the proof we used the following general rounding procedure.

\begin{lem}
\label{lem:round}
Let $x\in\{-1,+1\}^d$ and $v\in S^{d-1}$ satisfy
$\frac{1}{d}|\langle x,v\rangle|^2\ge\varepsilon$.
Then there exists a randomized estimator $\hat x\in \{-1,+1\}^d$,
whose construction depends only on $d,v,\varepsilon$, such that
$\frac{1}{d}|\langle x,\hat x\rangle|\ge \tfrac{\varepsilon}{8}$
with probability $1-\frac{64}{d\varepsilon^2}$.
\end{lem}

\begin{proof}
Fix $c>0$ that will be chosen shortly, and
construct $\hat x$ by choosing each entry to be an independent random 
sign so that $\EE[\hat x_i] = \frac{v_i\sqrt{d}}{c}\, 1_{|v_i|\sqrt{d}\le c}$. Then
$$
	\frac{1}{d}|\EE[\langle x,\hat x\rangle]| =
	\Bigg|
	\sum_{i\in[d]}
	\frac{x_iv_i}{c\sqrt{d}}
	\, 1_{|v_i|\sqrt{d}\le c}\Bigg|
	\ge
	\frac{\sqrt{\varepsilon}}{c}
	-
	\Bigg|
	\sum_{i\in[d]}
	\frac{x_iv_i}{c\sqrt{d}}
	\, 1_{|v_i|\sqrt{d}> c}\Bigg|
	\ge
	\frac{\sqrt{\varepsilon}}{c}
	-
	\frac{1}{c^2}.
$$
Choosing $c=\frac{2}{\sqrt{\varepsilon}}$ yields
$\frac{1}{d}\EE|\langle x,\hat x\rangle|\ge 
\frac{1}{d}|\EE[\langle x,\hat x\rangle]| \ge \frac{\varepsilon}{4}$.

Now note that $\mathrm{Var}(\frac{1}{d}|\langle x,\hat x\rangle|) \le
\mathrm{Var}(\frac{1}{d}\langle x,\hat x\rangle) 
\le \frac{1}{d}$. We can therefore estimate 
$$
	\mathbf{P}\big[
	\tfrac{1}{d}|\langle x,\hat x\rangle| <
	\tfrac{\varepsilon}{4}-t 
	\big]
	\le
	\mathbf{P}\big[
	\big|\tfrac{1}{d}|\langle x,\hat x\rangle| -
	\tfrac{1}{d}\EE|\langle x,\hat x\rangle|\big| > t
	\big]
	\le \frac{1}{dt^2}
$$
by Chebyshev's inequality. Choosing $t=\frac{\varepsilon}{8}$ yields
the conclusion.
\end{proof}

\subsection{Tensor PCA}
\label{sec:pftensor}

We begin with some basic observations.

\begin{lem}
\label{lem:tensor1}
$M$ is a $d\times d$ self-adjoint random matrix with $d={n\choose\ell}$,
such that
$$
	\EE[(M-\EE M)^2]=\sigma(M)^2\id
	\qquad\text{with}\qquad
	\sigma(M)^2 = 
	{\ell \choose p/2} {n-\ell \choose p/2}.
$$
Moreover, we have
$$
	v(M)^2 = {p\choose p/2}{n-p\choose \ell-p/2}.
$$
In particular, $\sigma(M)^2 \asymp n^{\frac{p}{2}}$ and
$v(M)^2 \asymp n^{\ell-\frac{p}{2}}$ as $n\to\infty$ (with $p,\ell$ 
fixed).
\end{lem}

\begin{proof}
We readily compute for $|S|=|T|=\ell$
$$
	\EE[(M-\EE M)^2]_{S,T} =
	\sum_{|R|=\ell} \EE[Z_{S\triangle R}Z_{T\triangle R}] 
	= 
	|\{R\subseteq[n]:|R|=\ell,|R\triangle S|=p\}|
	\,1_{S=T}
$$
using $S\triangle R=T\triangle R$ if and only if $S=T$.
As $|R|=|S|$, we have $|R\backslash S|=|S\backslash R|=
\frac{1}{2}|R\triangle S|$, so each $R$ satisfying
$|R|=\ell,|R\triangle S|=p$ is formed by replacing $\frac{p}{2}$ elements 
of $S$ by $\frac{p}{2}$ elements not in $S$. The number of all such $R$ 
is evidently $\sigma^2$.

Now note that $|S|=|T|=\ell$ and $|U|=p$ satisfy
$M_{S,T}-\EE M_{S,T} = Z_U$ if and only if $S\triangle T = U$. 
Thus given $U$, all such $S,T$ are formed by choosing $\frac{p}{2}$ 
elements of $U$ to place in $S$ (the remaining ones are placed in $T$), 
then choosing $\ell-\frac{p}{2}$ elements not in $U$ to place in
$S\cap T$. Thus there are $m:={p\choose p/2}{n-p\choose \ell-p/2}$ matrix
elements of $M-\EE M$ that coincide with each independent standard 
Gaussian variable $Z_U$. As $\mathrm{Cov}(M)$ is block-diagonal 
with blocks of the form $1_m1_m^*$, the conclusion follows.
\end{proof}

Next, we show that $\EE M$ is approximately of low rank.

\begin{lem}
\label{lem:tensor2}
Let $\mathrm{s}_1\ge\cdots\ge\mathrm{s}_d$ be the singular values
of $\EE M$. Then
$$
	\lambda_{\rm max}(\EE M)=
	\mathrm{s}_1 = 
	{\ell \choose p/2} {n-\ell \choose p/2}	\lambda ,\qquad\quad
	\mathrm{s}_{r+1} \le \frac{p}{n}\mathrm{s}_1,
$$
where $r = {n \choose \ell-p/2}\asymp n^{\ell-\frac{p}{2}}$ as
$n\to\infty$ (with $p,\ell$ fixed).
\end{lem}

\begin{proof}
The first statement is \cite[eq.\ (14)]{WEM19}. 
Next, by \cite[Proposition A.1]{WEM19}, the matrix $\frac{\EE 
M}{\lambda}$
has $\ell+1$ \emph{distinct} eigenvalues $\mu_0>\cdots>\mu_\ell$ 
so that $\mu_m$ has multiplicity ${n\choose m}-{n\choose m-1}$.
Moreover, it is shown in the proof of \cite[Lemma A.3]{WEM19}
that $|\mu_m|\le \frac{p}{n}\mu_0$ for $m>\ell-\frac{p}{2}$. It follows
readily that $\mathrm{s}_{r+1} \le \frac{p}{n}\mathrm{s}_1$ for
$$
	r = \sum_{m=0}^{\ell-p/2}
	\bigg( {n\choose m}-{n\choose m-1}\bigg) =
	{n \choose \ell-p/2},
$$
concluding the proof.
\end{proof}

We can now complete the proof of Theorem \ref{thm:kikuchi}.

\begin{proof}[Proof of Theorem \ref{thm:kikuchi}]
Let $d={n\choose\ell}\asymp n^\ell$ and
$r={n\choose \ell-p/2}\asymp n^{\ell-\frac{p}{2}}$. 
By Lemma \ref{lem:tensor2}, we can decompose $\EE M = B+(\EE M-B)$
so that $B$ has rank $r$, $\lambda_{\rm max}(B)=\lambda k_*$, and
$\|\EE M-B\|\le \frac{\lambda p}{n}k_*$. Now define the random matrix
$$
	X := k_*^{-\frac{1}{2}}(M-\EE M+B).
$$
Then $\EE[(X-\EE X)^2]=\id$ and $v(X) \lesssim n^{\frac{\ell-p}{2}}$
by Lemma \ref{lem:tensor1}. We obtain
$$
	\mathbf{P}\big[
	|\lambda_{\rm max}(X)-\mathrm{B}(\lambda k_*^{\frac{1}{2}})|
	>
	C n^{\frac{4\ell-3p}{4}} +
	C n^{\frac{\ell-p}{4}}(\log d)^{\frac{3}{4}}
	\big] \le \frac{1}{d^\beta}
$$
by applying Corollary \ref{cor:norm} and Theorem \ref{thm:bbp} with 
$t\leftarrow \sqrt{\beta\log d}$, where we used that $\sigma_*(X)\le v(X)$ and 
$C$ depends on $\beta$. Therefore
$$
	\mathbf{P}\big[
	|\lambda_{\rm max}(k_*^{-\frac{1}{2}}M)
	-\mathrm{B}(\lambda k_*^{\frac{1}{2}})|
	>
	C p n^{-1} \lambda k_*^{\frac{1}{2}} +
	C n^{\frac{4\ell-3p}{4}} +
	C n^{\frac{\ell-p}{4}}(\log d)^{\frac{3}{4}}
	\big] \le \frac{1}{d^\beta}.
$$
Finally, note that as $p,\ell$ are integers, $\ell<\frac{3p}{4}$ implies
that $4\ell\le 3p-1$. Thus 
$$
	n^{\frac{4\ell-3p}{4}}\le n^{-\frac{1}{4}},\qquad\qquad
	n^{\frac{\ell-p}{4}} \le n^{-\frac{p}{16}} \le n^{-\frac{1}{4}}.
$$
Since $\mathrm{B}(\lambda k_*^{\frac{1}{2}})=2$ when
$\lambda\le k_*^{-\frac{1}{2}}$ and
$\mathrm{B}(\lambda k_*^{\frac{1}{2}})\ge \mathrm{B}(1+\varepsilon)>2$
when $\lambda\ge (1+\varepsilon) k_*^{-\frac{1}{2}}$, the final part of 
the theorem follows readily.
\end{proof}

\begin{rem}
It is readily read off from the proof that the final part of the statement 
of Theorem \ref{thm:kikuchi} can be impoved in various ways: the exponent 
of $n^{-\frac{1}{5}}$ in the definition of the test can be improved in a 
manner depending on the choice of $\ell,p$, while the conclusion remains 
valid when $\varepsilon = n^{-\alpha}$ for a suitable choice of 
$\alpha>0$. Since the precise exponents provided by the proof are not 
expected to be optimal, we have stated a slightly weaker result for 
simplicity of exposition.
\end{rem}

\subsection{Spike detection in block-structured models}
\label{sec:aniso}

\begin{proof}[Proof of Theorem \ref{thm:aniso}]
Note that $X,X_\varnothing$ are precisely of the form \eqref{eq:krz} 
with $\mathbf{B}=\frac{1}{\boldsymbol{\Delta}}$ and $z = x$.
Moreover, that
$\sigma(X)^2 \le \frac{2}{d}\|\mathbf{B}1_d\|_\infty \le 2\beta$
and $\sigma_*(X)^2\le v(X)^2\le \frac{4\beta}{d}$ follows from
a straightforward computation. Thus
$$
	\mathbf{P}\big[
	|\lambda_{\rm max}(X)-\lambda_{\rm max}(X_{\rm free})|
	>
	C\beta^{\frac{1}{2}}d^{-\frac{1}{4}}
	(\log d)^{\frac{3}{4}}
	\big]
	\le e^{-d^{\frac{1}{2}}}
$$
by applying Corollary \ref{cor:norm} with $t\leftarrow d^{\frac{1}{4}}$,
and the analogous bound holds for $X_{\varnothing}$. Now note that for
$\lambda,\lambda_\varnothing$ as in Theorem \ref{thm:krz}, we have
$$
	|\lambda_{\rm max}(X_{\rm free})-\lambda| \le
	\sqrt{\frac{8q\beta}{d}},
	\qquad\quad
	|\lambda_{\rm max}(X_{\varnothing,\rm free})-\lambda_\varnothing| \le
	\sqrt{\frac{8q\beta}{d}}.
$$
The conclusion follows from Theorem \ref{thm:krz}, where we set
$\mu=\lambda_\varnothing$.
\end{proof}

\begin{rem}
\label{rem:gensignal}
The assumption that $x\in \{-1,+1\}^d$ was used in the proof only in order 
to ensure the assumption of section \ref{sec:krzdefn} that
$\sum_{i\in C_k}z_i^2 = |C_k|$ for each $k$.

Let us consider instead the case that $x$ is a random vector with i.i.d.\ 
entries such that $\EE[x_i^2]=1$, as is assumed in \cite{KKP23}.
Then the above condition does not hold exactly for $z\leftarrow x$, but it 
holds approximately by the law of large numbers. In particular, in this 
case we may choose $z\in\mathbb{R}^d$ to be defined by
$$
	z_i = 
	\frac{x_i}{\big(\frac{1}{|C_k|}\sum_{j\in C_k}
	x_j^2\big)^\frac{1}{2}}
	\quad\text{for all }i\in C_k,~k\in[q].	
$$
Then $z$ satisfies the assumption of section \ref{sec:krzdefn} by 
construction, while the law of large numbers ensures that we have
$$
	\bigg\|
	\frac{1}{d}\diag(z)\mathbf{B}\diag(z)-
	\frac{1}{d}\diag(x)\mathbf{B}\diag(x)
	\bigg\|
	= o(1)
$$
with probability $1-o(1)$ as long as $\min_k|C_k|\to\infty$ 
and $q,\beta$ do not grow too rapidly. We can therefore replace $x$ by $z$ in 
the analysis up to a negligible error, and the remainder of the analysis
proceeds verbatim as in the proof of Theorem \ref{thm:aniso}.

The above argument is readily implemented without any further assumption 
in the asymptotic setting where $q,\Delta,c$ and the distribution of $x_i$ 
are fixed as $d\to\infty$ by using the classical law of large numbers. 
However, implementing this procedure in a nonasymptotic setting would 
require us to quantify the error in the law of large numbers. This is 
readily accomplished in many situations, but the details of the bounds 
depend on the precise assumptions that are made on the distribution of 
$x_i$ (for example, if they are subgaussian, we may use Bernstein's 
inequality to obtain nonasymptotic bounds.) As this part of the 
argument is completely independent of the random matrix analysis, we 
do not pursue it further here.
\end{rem}

We now turn to the proof of Theorem \ref{thm:anisoeigen}.

\begin{proof}[Proof of Theorem \ref{thm:anisoeigen}]
Let $B=\frac{1}{\Delta}$.
We aim to apply Lemma \ref{lem:perturb} with $P=\frac{1}{d}xx^*$. 
Let us therefore define $X_t := X+\frac{t}{d}xx^*$, so that
in the notation of Lemma \ref{lem:meanA}
$$
	\EE X_t = \mathrm{A}\big(
	\diag(c)^{\frac{1}{2}}(B+t1_q1_q^*)\diag(c)^{\frac{1}{2}} -
	\diag(Bc),
	-Bc
	\big).
$$
Note that the precise form 
of $\EE X$ played no role in the proof of Proposition \ref{prop:lambda}, 
so that it transfers verbatim to the present setting. In particular, if
we define
\begin{align*}
	\lambda_t :=
	\inf_{v>0} \max\big\{
	&
	\lambda_{\rm max}\big(
		\diag(c)^{\frac{1}{2}}(B+t1_q1_q^*)\diag(c)^{\frac{1}{2}} +
		\diag\big(B\diag(c)(v-1_q)\big)
	\big), \\
	&
	\lambda_{\rm max}\big(
		\diag(v)^{-1} + \diag\big(B\diag(c)(v-1_q)\big)
	\big)
	\big\},
\end{align*}
then the proofs of Proposition \ref{prop:lambda} and Theorem 
\ref{thm:aniso} readily yield
$$
	\mathbf{P}\Big[
	|\lambda_{\rm max}(X_t)-\lambda_t|
	>
	C\beta^{\frac{1}{2}}\Big(\tfrac{(\log 
		d)^{\frac{3}{4}}}{d^{\frac{1}{4}}}+
		\tfrac{q^{\frac{1}{2}}}{d^{\frac{1}{2}}}\Big)
	\Big]
	\le e^{-d^{\frac{1}{2}}}
$$
for every $t\in\mathbb{R}$. We therefore obtain for any $t>0$
$$
	\frac{\lambda_0-\lambda_{-t}}{t}-o(1)
	\le
	\frac{1}{d}
	|\langle x,v_{\rm max}(X)
	\rangle|^2 \le \frac{\lambda_t-\lambda_0}{t} + o(1)
$$
with probability $1-o(1)$ as $d\to\infty$ using Lemma \ref{lem:perturb}.

The major simplification of the asymptotic setting where $q,B,c$ are fixed 
as $d\to\infty$ is that the definition of $\lambda_t$ is then independent 
of $d$. To conclude the proof it therefore suffices to gain a qualitative, 
rather than quantitative, understanding of the behavior of 
$\lambda_t-\lambda_0$. In the remainder of the proof, we will consider the
three cases $\mathrm{SNR}(\Delta)<1$, $\mathrm{SNR}(\Delta)>1$, and
$\mathrm{SNR}(\Delta)=1$ separately.

\medskip

\textbf{Case 1.} Suppose that $\mathrm{SNR}(\Delta)<1$, so that
$\lambda_0=:\lambda=\lambda_\varnothing$ by Theorem \ref{thm:krz}.
Denote by $v^*$ the minimizer in the definition of $\lambda_0$. 
Then it can be read off from the proof of Lemma \ref{lem:phasereg} and
from Lemma \ref{lem:phasebd} that
\begin{align*}
	&\lambda_{\rm max}\big(
		\diag(c)^{\frac{1}{2}}B\diag(c)^{\frac{1}{2}} +
		\diag\big(B\diag(c)(v^*-1_q)\big)
	\big) < \lambda_0,\\
	&
	\lambda_{\rm max}\big(
		\diag(v^*)^{-1} + \diag\big(B\diag(c)(v^*-1_q)\big)
	\big) = \lambda_0.
\end{align*}
Thus choosing $v\leftarrow v^*$ in the definition of 
$\lambda_t$ shows that $\lambda_t\le\lambda_0$ when $t>0$ is sufficiently 
small. The conclusion follows immediately.

\medskip

\textbf{Case 2.} Suppose that $\mathrm{SNR}(\Delta)>1$, so that $\lambda_0 
=: \lambda > \lambda_\varnothing$ by Theorem \ref{thm:krz}. Denote by 
$v^*$ the minimizer in the definition of $\lambda_0$. Then it follows from 
step 2 in the proof of Theorem \ref{thm:krz} that we must have 
\begin{align}
\label{eq:aniso1st}
	&
	\lambda_{\rm max}\big(
		\diag(c)^{\frac{1}{2}}B\diag(c)^{\frac{1}{2}} +
		\diag\big(B\diag(c)(v^*-1_q)\big)
	\big) = \lambda_0,\\
\label{eq:aniso2nd}
	&
	\lambda_{\rm max}\big(
		\diag(v^*)^{-1} + \diag\big(B\diag(c)(v^*-1_q)\big)
	\big) = \lambda_0.
\end{align}
Moreover, $v^*$ is not a minimizer of the left-hand side of 
\eqref{eq:aniso2nd}, as that would contradict 
$\lambda>\lambda_\varnothing$. Now note that the largest eigenvalue in
\eqref{eq:aniso1st} is simple and the associated eigenvector $w>0$ has 
strictly positive entries by the Perron-Frobenius theorem.
Thus $w$ is not orthogonal to 
$\diag(c)^{\frac{1}{2}}1_q$, so we must have\footnote{%
	Let $M$ be a self-adjoint matrix whose top eigenvalue is
	simple with eigenvector $w$, and let $x$ be a vector not 
	orthogonal to $w$. Then
	we have $\frac{d}{dt}\lambda_{\rm 
	max}(M-txx^*)|_{t=0}= - |\langle x,w\rangle|^2<0$.
}%
\begin{equation}
\label{eq:aniso3rd}
	\lambda_{\rm max}\big(
		\diag(c)^{\frac{1}{2}}(B-t1_q1_q^*)\diag(c)^{\frac{1}{2}} +
		\diag\big(B\diag(c)(v^*-1_q)\big)
	\big) < \lambda_0	
\end{equation}
for every $t>0$. But as $v^*$ is not a minimizer of \eqref{eq:aniso2nd}, 
we can slightly perturb $v_*$ to decrease the latter while preserving the
strict inequality in \eqref{eq:aniso3rd}.
This shows that $\lambda_{-t}<\lambda_0$ for every $t>0$. The conclusion 
follows immediately.

\medskip

\textbf{Case 3.} Suppose that $\mathrm{SNR}(\Delta)=1$.
Then $\lambda_0=:\lambda = 1$ by Theorem \ref{thm:krz}, and
$v^*=1_q$ is a minimizer in the definition of $\lambda_0$.

Denote by $b>0$ be the Perron-Frobenius eigenvector of 
$B\diag(c)$, and choose $s>0$ sufficiently large that
$\lambda_{\rm max}(\diag(c)^{\frac{1}{2}}1_q1_q^*
\diag(c)^{\frac{1}{2}}- s\diag(b))\le 0$.
Then choosing $v\leftarrow 1_q-tsb$ in the definition of
$\lambda_t$ readily yields
$$
	\lambda_t - \lambda_0 \le
	\max\Big\{0,
	\max_i \Big\{ \frac{1}{1-tsb_i} - 1 - tsb_i\Big\}
	\Big\}
$$
for all sufficiently small $t>0$, where we used that $\lambda_{\rm 
max}(\diag(c)^{\frac{1}{2}}B\diag(c)^{\frac{1}{2}})=: 
\mathrm{SNR}(\Delta)=1$ implies that the Perron-Frobenius eigenvalue of 
$B\diag(c)$ is $1$.

To conclude, note that $\frac{1}{1-x}-1-x = \frac{x^2}{1-x}$, so 
we have shown that 
$\lambda_t - \lambda_0 \le O(t^2)$ for $t>0$ sufficiently small.
The conclusion follows immediately.
\end{proof}

\subsection{Contextual stochastic block models}
\label{sec:pfcontext}

\begin{proof}[Proof of Theorem \ref{thm:context}]
Let $d=n+p$, and partition $[d]=C_1\sqcup C_2$ into $C_1=\{1,\ldots,n\}$ 
and $C_2=\{n+1,\ldots,n+p\}$. Define $B\in\mathrm{M}_2(\mathbb{R})_{\rm sa}$
and $\hat z\in\mathbb{R}^d$ as
$$
	B = \frac{d}{n}
	\begin{bmatrix}
	\lambda^2 & \mu \\
	\mu & 0
	\end{bmatrix},
	\qquad\quad
	\hat z = \begin{bmatrix}
	v \\ u\sqrt{p} \end{bmatrix}.
$$
Then the random matrix $\hat X$ is of the form \eqref{eq:krz} with 
$z\leftarrow \hat z$.

Next, define the random matrix $X$ as in \eqref{eq:krz} with
$$
	z = \begin{bmatrix}
        v \\ \frac{u\sqrt{p}}{\|u\|} \end{bmatrix}.
$$
The random matrix $X$ satisfies all the assumptions of section 
\ref{sec:krzdefn}. On the other hand, as $u\sim N(0,\frac{1}{p}\id_p)$, we 
have $\|u\|=1+o(1)$ with probability $1-o(1)$ by the law of large numbers.
Using $\|\diag(x)M\diag(y)\|\le \max_{ij}|M_{ij}|\,\|x\|\,\|y\|$, we 
obtain
$$
	\|X-\hat X\| 
	\le
	\frac{1}{d}
	\max_{i,j}|B_{ij}| \, 
	(\|z\|+\|\hat z\|)
	\,\|z-\hat z\| = o(1)
$$
with probability $1-o(1)$, where we used that $\max_{i,j}|B_{ij}| \to
(1+\frac{1}{\gamma})\max\{\lambda^2,\mu\}$.

To reason about $\hat v$, we would like to apply
Lemma \ref{lem:perturb} to $\hat X_t = \hat X + \frac{t}{n} 
\II_{C_1}zz^*\II_{C_1}^*$. Approximating $\hat X$ by $X$ and reasoning
as in the proof of Theorem \ref{thm:anisoeigen}, we have
$$
	\lambda_{\rm max}(\hat X_t) = \lambda_t + o(1)
	\text{ with probability }1-o(1),
$$
where $\lambda_t$ is defined by
\begin{align*}
	\lambda_t :=
	\inf_{v>0} \max\big\{
	&
	\lambda_{\rm max}\big(
		t e_1 e_1^* +
		\diag(c)^{\frac{1}{2}}\bar B\diag(c)^{\frac{1}{2}} +
		\diag\big(\bar B\diag(c)(v-1_2)\big)
	\big), \\
	&
	\lambda_{\rm max}\big(
		\diag(v)^{-1} + \diag\big(\bar B\diag(c)(v-1_2)\big)
	\big)
	\big\}
\end{align*}
where we define
$$
	\bar B = (1+\tfrac{1}{\gamma})
	\begin{bmatrix}
	\lambda^2 & \mu \\[.1cm]
	\mu & 0
	\end{bmatrix},
	\qquad\quad
	c = \begin{bmatrix}
	\frac{\gamma}{1+\gamma} \\[.1cm] \frac{1}{1+\gamma}
	\end{bmatrix}.
$$
We can now follow the remainder of the proof of Theorem
\ref{thm:anisoeigen} verbatim to show that there is asymptotically
positive overlap between $v$ and $\hat v$ if and only if
$$
	\frac{1}{2}\Big(
	\lambda^2 + \sqrt{\lambda^4+\tfrac{4\mu^2}{\gamma}}\Big)
	=
	\lambda_{\rm max}\big(\diag(c)^{\frac{1}{2}}\bar B
	\diag(c)^{\frac{1}{2}}\big)>1.
$$
Now note that
$$
	\frac{1}{2}\Big(
	\lambda^2 + \sqrt{\lambda^4+\tfrac{4\mu^2}{\gamma}}\Big) = 1
	\quad\text{if and only if}\quad
	\lambda^2+\tfrac{\mu^2}{\gamma} = 1 
$$
and both $\frac{1}{2}\big(
\lambda^2 + \sqrt{\lambda^4+\tfrac{4\mu^2}{\gamma}}\big)$ and
$\lambda^2+\tfrac{\mu^2}{\gamma}$ are monotone in $\lambda^2$ and $\mu^2$, 
so
$$
	\frac{1}{2}\Big(
	\lambda^2 + \sqrt{\lambda^4+\tfrac{4\mu^2}{\gamma}}\Big) > 1
	\quad\text{if and only if}\quad
	\lambda^2+\tfrac{\mu^2}{\gamma} > 1. 
$$
This concludes the proof.
\end{proof}

\subsection{Sample covariance error}
\label{sec:pfscov}

In the following, we define $\hat\Sigma = \frac{1}{n}XX^*$ and $\Sigma$ as 
in \eqref{eq:scov}---\eqref{eq:scovrk1} and let $\delta=\frac{p}{n}$, 
$d=\max\{n,p\}$. We begin by applying Theorem \ref{thm:quad}.

\begin{lem}
\label{lem:quadapp}
For $n \ge (\log d)^3$, we have
\begin{align*}
&	\mathbf{P}\big[\big|\|\hat\Sigma\|-
	\|\tfrac{1}{n}X_{\rm free}X_{\rm free}^*\|\big|
	>
	C(1+\lambda+\delta)\,n^{-\frac{1}{4}}
	(\log d)^{\frac{3}{4}}
	\big] \le
	e^{-Cn^{\frac{1}{2}}},
\\
&	\mathbf{P}\big[\mathrm{d_H}\big(\spc(\hat\Sigma-\Sigma),
	\spc(
	\tfrac{1}{n}X_{\rm free}X_{\rm free}^*-\Sigma\otimes\id)\big)
	>
	C(1+\lambda+\delta)\,n^{-\frac{1}{4}}
	(\log d)^{\frac{3}{4}}
	\big] \le
	e^{-Cn^{\frac{1}{2}}}.
\end{align*}
\end{lem}

\begin{proof}
We readily compute 
$$
	\sigma_*(X)^2=v(X)^2=\|\Sigma\| =1+\lambda
$$
and
$$
	\sigma(X)^2=n\max\{1+\lambda,\delta+\tfrac{\lambda}{n}\}
	\le 2n (1+\lambda+\delta).
$$
Moreover, note that $\|X_{\rm free}\| \le 2\sigma(X)$ by \cite[p.\ 208]{Pis03}.
The conclusion now follows readily by applying
Theorem \ref{thm:quad} with $X\leftarrow n^{-\frac{1}{2}}X$,
$t\leftarrow
2 (1+\lambda+\delta)^{1/2} n^{-\frac{1}{4}} (\log d)^{\frac{3}{4}}$,
and either $B\leftarrow 0$ or $B\leftarrow -\Sigma$, respectively.
(Note that while Theorem \ref{thm:quad} is formulated for $d\times d$
matrices $X$, it is applicable here as we can always add enough zero rows 
or columns to $X$ to make
it $d\times d$ without changing the relevant norms.)
\end{proof}

We must now estimate the spectra of the free operators that appear in 
the above lemma. In principle, this can be achieved using methods that are 
sketched in the work of Lehner \cite[\S 5]{Leh97}, which requires some 
lengthy computations. These computations are considerably simplified by a 
quadratic counterpart of Lehner's formula obtained in \cite{EvH24}. For 
our present purposes, we may work with the following result that arises as 
a special case of the general formulas in \cite{EvH24}.

\begin{lem}
\label{lem:emre}
Let $X$ be the $p\times n$ random matrix whose columns are i.i.d.\ 
$N(0,\Sigma)$, and denote the eigenvalues of $\Sigma$ as
$\mu_1\ge\cdots\ge\mu_p\ge 0$.
Then
\begin{align*}
	\|\tfrac{1}{n}X_{\rm free}X_{\rm free}^*\| &=
	\inf_{0<a<1} \inf_{x\in\Delta_p} \max_{i\in[p]}
	\bigg\{ \frac{\mu_i}{nax_i} + \frac{\mu_i}{1-a}\bigg\},
\\
	\lambda_{\rm max}(
	\tfrac{1}{n}X_{\rm free}X_{\rm free}^*-\Sigma\otimes\id) &=
	\inf_{0<a<1} \inf_{x\in\Delta_p} \max_{i\in[p]}
	\bigg\{ \frac{\mu_i}{nax_i} + \frac{a\mu_i}{1-a}\bigg\},
\\
	-\lambda_{\rm min}(
	\tfrac{1}{n}X_{\rm free}X_{\rm free}^*-\Sigma\otimes\id) &=
	\inf_{a>0} \inf_{x\in\Delta_p} \max_{i\in[p]}
	\bigg\{ \frac{\mu_i}{nax_i} + \frac{a\mu_i}{1+a}\bigg\},
\end{align*}
where $\Delta_p := \{x\in\mathbb{R}^p: x>0,~\sum_i x_i=1\}$.
\end{lem}

\begin{proof}
Since the quantities to be computed are independent of the choice of basis 
in $\mathbb{R}^p$, we may assume without loss of generality that $\Sigma$ 
is a diagonal matrix, so that $X$ has independent entries $X_{ij}\sim 
N(0,\mu_i)$. We can also assume that $\mu_p>0$, since otherwise all 
quantities in the statement remain unchanged if we remove the zero rows.
With these simplifications, \cite[Corollary 1.5]{EvH24} yields
\begin{align*}
	\|
        \tfrac{1}{n}X_{\rm free}X_{\rm free}^*\|
	&=
	\inf_{\substack{v>0\\\frac{1}{n}\sum_k\mu_kv_k<1}}
	\max_{i\in[p]}\bigg\{
	\frac{1}{v_i} +
	\frac{\mu_i}{1-\frac{1}{n}\sum_k\mu_kv_k}
	\bigg\},
\end{align*}
and making the change of variables $x_i = 
\frac{\mu_iv_i}{\sum_k\mu_kv_k}$,
$a=\frac{1}{n}\sum_k\mu_kv_k$ yields the first equation in the statement. 
The other two equations follow analogously.
\end{proof}

In our setting, we have $\mu_1=1+\lambda$ and $\mu_2=\cdots=\mu_p=1$. We 
claim that it then suffices to minimize in the above 
variational principles only over vectors $x$ such that $x_2=\cdots=x_p$. 
That the latter yields an upper bound is obvious (as we are restricting 
the infimum to a smaller set). For the lower bound, we may use that 
$$
	\max_{2\le i\le n}\frac{1}{x_i}\ge \frac{1}{n-1}\sum_{i=2}^n 
	\frac{1}{x_i}\ge \frac{1}{\frac{1}{n-1}\sum_{i=2}^n x_i}
$$
by convexity to argue that for any vector $x$, the function being 
optimized can only decrease 
if we replace all $x_2,\ldots,x_n$ by their average.

\begin{lem}
\label{lem:sion}
In the setting of Theorem \ref{thm:scov}, we have
\begin{align*}
	&\big|\|\tfrac{1}{n}X_{\rm free}X_{\rm free}^*\| -
	\mathrm{S}(\lambda,\delta)\big|
	\le C(1+\lambda+\delta)\,n^{-\frac{1}{2}},
\\
	&\big|
	\lambda_{\rm max}(
	\tfrac{1}{n}X_{\rm free}X_{\rm free}^*-\Sigma\otimes\id)
	-\mathrm{H}_+(\lambda,\delta)\big|
        \le C(1+\lambda+\delta)\,n^{-\frac{1}{2}}.
\end{align*}
\end{lem}

\begin{proof}
Restricting the infimum over $x$ in Lemma \ref{lem:emre} to
$x=(b,\frac{1-b}{p-1},\cdots,\frac{1-b}{p-1})$ yields
\begin{align*}
	\|\tfrac{1}{n}X_{\rm free}X_{\rm free}^*\| &=
	\inf_{0<a,b<1} 
	\max
	\bigg\{ \frac{1+\lambda}{nab} + \frac{1+\lambda}{1-a},
	\frac{p-1}{na(1-b)} + \frac{1}{1-a}\bigg\},
\\
	\lambda_{\rm max}(
	\tfrac{1}{n}X_{\rm free}X_{\rm free}^*-\Sigma\otimes\id) &=
	\inf_{0<a,b<1} 
	\max
	\bigg\{ \frac{1+\lambda}{nab} + \frac{(1+\lambda)a}{1-a},
	\frac{p-1}{na(1-b)} + \frac{a}{1-a}\bigg\}
\end{align*}
by the above observation. We can rewrite the first line as
\begin{align*}
	\|\tfrac{1}{n}X_{\rm free}X_{\rm free}^*\| &=
	\inf_{x\in\Delta_3}
	\sup_{0<\pi<1}
	\bigg\{
	\pi\bigg(
	\frac{1+\lambda}{nx_1} + \frac{1+\lambda}{x_3}\bigg) +
	(1-\pi)\bigg(
	\frac{p-1}{nx_2} + \frac{1}{x_3}\bigg)\bigg\}
\\
	&=
	\sup_{0<\pi<1}
	\bigg(
	\sqrt{\frac{\pi(1+\lambda)}{n}} +
	\sqrt{\frac{(1-\pi)(p-1)}{n}} +	
	\sqrt{1+\pi\lambda}
	\bigg)^2,
\end{align*}
where we used the Sion minimax theorem and
$\inf_{x\in\Delta_r}\sum_{i=1}^r \frac{a_i}{x_i} =
\big(\sum_{i=1}^r \sqrt{a_i}\big)^2$.
By exactly the same argument, the second line becomes
\begin{multline*}
	\lambda_{\rm max}(
	\tfrac{1}{n}X_{\rm free}X_{\rm free}^*-\Sigma\otimes\id) =
\\
	\sup_{0<\pi<1}
	\bigg\{
	\bigg(
	\sqrt{\frac{\pi(1+\lambda)}{n}} +
	\sqrt{\frac{(1-\pi)(p-1)}{n}} +	
	\sqrt{1+\pi\lambda} 
	\bigg)^2
	- (1+\pi\lambda)\bigg\},
\end{multline*}
where we used that $\frac{a}{1-a}=\frac{1}{1-a}-1$.

To conclude the proof, we note that
\begin{align*}
	\mathrm{S}(\lambda,\delta) &=
	\sup_{0<\pi<1}\big(\sqrt{(1-\pi)\delta}+
	\sqrt{1+\pi\lambda}\,\big)^2,\\
	\mathrm{H}_+(\lambda,\delta) &=
	\sup_{0<\pi<1}\big\{\big(\sqrt{(1-\pi)\delta}+
	\sqrt{1+\pi\lambda}\,\big)^2-
	(1+\pi\lambda)\big\}.
\end{align*}
The conclusion now follows readily.
\end{proof}

It remains to estimate the smallest eigenvalue of the centered case. The 
proof is similar to that of Lemma \ref{lem:sion}, but differs in the 
details of the computation.

\begin{lem}
\label{lem:sion2}
In the setting of Theorem \ref{thm:scov}, we have
$$
	\big|
	\lambda_{\rm min}(
	\tfrac{1}{n}X_{\rm free}X_{\rm free}^*-\Sigma\otimes\id)
	-\mathrm{H}_-(\lambda,\delta)\big|
        \le C(1+\lambda+\delta)\,n^{-\frac{1}{2}}.
$$
\end{lem}

\begin{proof}
As in Lemma \ref{lem:sion}, we may
restrict the infimum over $x$ in the last equation display of Lemma 
\ref{lem:emre} to $x=(b,\frac{1-b}{p-1},\cdots,\frac{1-b}{p-1})$. This 
yields
\begin{multline*}
	-\lambda_{\rm min}(
	\tfrac{1}{n}X_{\rm free}X_{\rm free}^*-\Sigma\otimes\id)
	= \\
	\sup_{0<\pi<1}
	\inf_{a>0} \inf_{0<b<1} 
	\bigg\{ 
	\pi\bigg(
	\frac{1+\lambda}{nab} + \frac{(1+\lambda)a}{1+a}\bigg)
	+ (1-\pi)\bigg(
	\frac{p-1}{na(1-b)} + \frac{a}{1+a}\bigg)\bigg\},
\end{multline*}
where we used the Sion minimax theorem as in the proof of Lemma 
\ref{lem:sion}.

Now note that for $u,v>0$, we have $\inf_{0<b<1}
\big(\frac{u}{b}+\frac{v}{1-b}\big) = (\sqrt{u}+\sqrt{v})^2$ and
$$
	\varphi(u,v) := 
	\inf_{a>0}\bigg\{\frac{a}{1+a}u+\frac{1}{a}v\bigg\}
	=
	\begin{cases}
	2\sqrt{uv}-v & \text{if }v<u,\\
	u & \text{otherwise}.
	\end{cases}
$$
We therefore obtain
\begin{multline*}
	-\lambda_{\rm min}(
	\tfrac{1}{n}X_{\rm free}X_{\rm free}^*-\Sigma\otimes\id)
	= \\
	\sup_{0<\pi<1}
	\varphi\bigg(
	1+\pi\lambda,
	\bigg(\sqrt{\frac{\pi(1+\lambda)}{n}} 
	+ \sqrt{\frac{(1-\pi)(p-1)}{n}}\bigg)^2
	\bigg).
\end{multline*}
Using the explicit formula for $\varphi$, we readily estimate
$$
	\bigg|\lambda_{\rm min}(
        \tfrac{1}{n}X_{\rm free}X_{\rm free}^*-\Sigma\otimes\id)
	+
	\sup_{0<\pi<1}\varphi(1+\pi\lambda,(1-\pi)\delta)
	\bigg|
	\le C(1+\lambda+\delta)n^{-\frac{1}{2}}.
$$
Now note that $(1-\pi)\delta<1+\pi\lambda$ holds if and only if
$\pi>\frac{\delta-1}{\delta+\lambda}$. Therefore
$$
	\sup_{0<\pi<1}\varphi(1+\pi\lambda,(1-\pi)\delta) =
	\sup_{\frac{\delta-1}{\delta+\lambda}\le\pi<1}
	\big\{
	2\sqrt{(1+\pi\lambda)(1-\pi)\delta} - 
	(1-\pi)\delta
	\big\},
$$
where we used that $\varphi(1+\pi\lambda,(1-\pi)\delta)$ is
increasing for $0<\pi<\frac{\delta-1}{\delta+\lambda}$.
The supremum on the right-hand side
equals $-\mathrm{H}_-(\lambda,\delta)$, concluding the proof.
\end{proof}

Combining Lemmas \ref{lem:quadapp}, \ref{lem:sion}, and \ref{lem:sion2} 
completes the proof of Theorem \ref{thm:scov}.

\SkipTocEntry\subsection*{Acknowledgments}

The authors thank Raphael Barboni, Charles Bordenave, Pravesh Kothari, 
Florent Krzakala, and Lenka Zdeborov{\'a} for useful discussions, and an 
anonymous referee for several helpful suggestions. GC was supported in 
part by the MUR Excellence Department Project MatMod@TOV awarded to the 
Department of Mathematics, University of Rome Tor Vergata, CUP 
E83C18000100006. RvH was supported in part by NSF grants DMS-2054565 and 
DMS-2347954.

\bibliographystyle{abbrv}
\bibliography{ref}

\end{document}